\documentclass[12pt]{article}
\usepackage[margin=1in]{geometry}
\usepackage[utf8]{inputenc}
\usepackage{graphicx} % Required for inserting images
\usepackage{amsfonts}       % blackboard math symbols
\usepackage{nicefrac}       % compact symbols for 1/2, etc.
\usepackage{microtype}      % microtypography
\usepackage{algorithm}
\usepackage{algpseudocode}
\usepackage{caption}
\usepackage{amsfonts}
\usepackage{amsmath}
\usepackage{amssymb}
\usepackage{upgreek}
\usepackage{fancyhdr}
\usepackage{titlesec}
\usepackage{indentfirst}
\usepackage{booktabs}
\usepackage{verbatim}
\usepackage{color}
\usepackage{tcolorbox}
\usepackage{mathtools}
\usepackage{bm}
\usepackage{amsthm}
\usepackage{wasysym}
\usepackage{empheq}
\usepackage{hyperref}
\usepackage{caption}
\usepackage{subcaption}
\usepackage{ulem}
\usepackage{booktabs}

\newtheorem{theorem}{Theorem}[section]
\newtheorem{corollary}{Corollary}

\newtheorem{lemma}[theorem]{Lemma}
\newtheorem{proposition}[theorem]{Proposition}

\newtheorem{assumption}{Assumption}

\newtheorem{definition}{Definition}
\newtheorem{remark}{Remark}

\def\b{\beta}
\def\d{\delta}
\def\e{\epsilon}

\def\lam{\lambda}

\def\s{\sigma}
\def\th{\theta}

\def\th{\theta}

\def\D{\Delta}

\def\Sig{\Sigma}

\def\E{\mathbb E}

\def\R{\mathbb R}
\def\S{{\mathbb S}}

\def\mL{\mathcal L}

\def\mN{\mathcal N}

\def\l{\left}
\def\r{\right}
\def\la{\left\langle}
\def\ra{\right\rangle}
\def\ll{\left\lVert}
\def\rl{\right\rVert}
\def\lv{\left\lvert}
\def\rv{\right\rvert}
\def\({\left(}
\def\){\right)}

\def\pt{\partial}
\def\nb{\nabla}

\def\ts{\times}
\def\qd{\quad}

\def\h{\hat}
\def\t{\tilde}

\def\dt{{\Delta t}}

\def\hV{\hat{V}}
\def\tV{\t{V}}
\def\hv{\hat{v}}
\def\bV{\bar{V}}
\def\hlam{\hat{\lam}}
\def\hmu{\hat{\mu}}

\def\hs{\hat{\Sig}}
\def\ts{\tilde{\Sig}}
\def\ts{\t{\Sig}}

\def\Linf{{L^\infty}}

\def\st{s_t}
\def\hst{\hat{s}_t} 
\def\lammin{{\lam_{\text{min}}}}

\newcommand{\coef}[1]{a^{(#1)}}

\title{PhiBE: A PDE-based Bellman Equation for Continuous Time Policy Evaluation}
\author{Yuhua Zhu\thanks{Department of Statistics and Data Science, University of California - Los Angeles, 
       Los Angeles, CA 90095, USA; e-mail: yuhuazhu@ucla.edu}}
\date{}

\begin{document}

\maketitle

\begin{abstract}

    In this paper, we study policy evaluation in continuous-time reinforcement learning (RL), where the state follows an unknown stochastic differential equation (SDE) but only discrete-time data are available.  We first highlight that the discrete-time Bellman equation (BE) is not always a reliable approximation to the true value function because it ignores the underlying continuous-time structure. We then introduce a new bellman equation, PhiBE, which integrates the discrete-time information into a continuous-time PDE formulation. By leveraging the smooth structure of the underlying dynamics, PhiBE  provides a provably more accurate approximation to the true value function, especially in scenarios where the underlying dynamics change slowly or the reward oscillates. Moreover, we extend PhiBE to higher orders, providing increasingly accurate approximations. %We conduct the error analysis for both BE and PhiBE with explicit dependence on the discounted coefficient, the reward and the dynamics.
    
    We further develop a model-free algorithm for PhiBE under linear function approximation and establish its convergence under model misspecification, together with finite-sample guarantees. In contrast to existing continuous-time RL analyses, where the model misspecification error diverges as the sampling interval $\dt\to 0$ and the sample complexity typically scales as $O(\dt^{-4})$, our misspecification error is independent of $\dt$ and the resulting sample complexity improves to $O(\dt^{-1})$ by exploiting the smoothness of the underlying dynamics. 
    Moreover, we identify a fundamental trade-off between discretization error and sample error that is intrinsic to continuous-time policy evaluation: finer time discretization reduces bias but amplifies variance, so excessively frequent sampling does not necessarily improve performance. This is an insight that does not arise in classical discrete-time RL analyses. Numerical experiments are provided to validate the theoretical guarantees we propose.

%Specifically, finer time discretization reduces bias but increases variance, leading to an explicit nonasymptotic characterization of how $\Delta t$ and the sample size jointly determine performance. decreasing $\Delta t$ reduces discretization error but increases variance, implying that excessively frequent sampling does not necessarily improve performance.  To the best of our knowledge, such an explicit trade-off, together with a nonasymptotic characterization of its dependence on $\Delta t, n$ and $\beta$, has not been systematically identified in prior analyses.
    %that diverges as the sampling interval shrinks, the approximation error of PhiBE remains remains well-conditioned and independent of the discretization step by exploiting the smoothness of the underlying dynamics.
    
\end{abstract}

\section{Introduction}
Reinforcement learning (RL) has achieved remarkable success in  {\it digital environments}, with applications such as AlphaGo \cite{Silver2016}, strategic gameplay \cite{Mnih2015}, and fine-tuning large language models \cite{Ziegler2019}. In these applications, the system state changes only after an action is taken. 
In contrast, many systems in the {\it physical world}, the state evolves continuously in time, regardless of whether actions are taken in continuous or discrete time. Although these systems are fundamentally continuous in nature, the available data are typically collected at discrete time points. In addition, 
the time difference between data points can be irregular, sparse, and outside our control.  {Examples arise in the following applications: 
\begin{itemize}
    \item Health care: for instance, in diabetes treatment, a patient’s glucose level changes continuously, but measurements are recorded only when tests are performed \cite{guo2023general,amos2018differentiable,romero2024actor,emerson2023offline,cobelli2009diabetes, battelino2019clinical};
    \item Robotic control \cite{karimi2023dynamic,Kober2013,siciliano1999robot} and autonomous driving \cite{sciarretta2004optimal}: where the environment evolve continuously but sensor data is collected at discrete intervals;
    \item Financial markets \cite{merton1975optimum,Moody2001, wang2020continuous, whittle1981risk}: Asset prices move continuously but trades and quotes occur at discrete times; 
    \item  Plasma fusion reactors \cite{degrave2022magnetic}:  The plasma evolves in continuous time, while all experimental data are recorded as discrete-time samples via sensors. 
\end{itemize} }
This paper directs its focus toward addressing continuous-time RL problems \cite{wang2020reinforcement, doya2000reinforcement,basei2022logarithmic, jia2022policy, jia2023q} where the underlying dynamics follow unknown stochastic differential equations (SDEs) but only discrete-time observations are available. Since one can divide the RL problem into policy evaluation (PE) and policy update \cite{sutton1999policy,konda1999actor, tallec2019making, jia2022policy}, we first concentrate on the continuous-time PE problem in this paper.

%One of the key challenges in the RL problem in the physical world is addressing the mismatch between the continuous-time dynamics and the discrete-time data. 
Currently, there are two main approaches to addressing this challenge. One approach is to learn the continuous-time dynamics from discrete-time data and formulate the problem as an policy evaluation problem with known dynamics \cite{kamalapurkar2016model, della2023model,lee2021policy, vamvoudakis2010online}. 
The key advantage of this approach is that it preserves the continuous-time nature of the problem, enhancing the stability and interpretability of the resulting algorithms. 
However, identifying continuous-time dynamics from discrete-time data is often challenging, and in general, ill-posed: there are infinitely many continuous-time dynamics can yield the same discrete-time transitions. Consequently, a misspecified continuous-time model may introduce significant errors, which can propagate to the learned optimal policy and lead to suboptimal decision-making.

Another common approach to address the continuous-time PE problem involves discretizing time and treating it as a Markov reward process. This method yields an approximated value function satisfying a {(discrete-time)} Bellman equation (BE), thereby one can use RL algorithms such as temporal difference (TD) \cite{sutton2018reinforcement}, gradient TD \cite{precup2001off}, Least square TD \cite{bradtke1996linear} to solve the BE.  
This approach is appealing for several reasons.
First, it relies only on the discrete-time transition dynamics, thereby avoiding the non-identifiability issues inherent in the first approach. 
Second, many RL algorithms are model-free, eliminating the need to explicitly learn the transition dynamics, making these plug-and-play algorithms convenient to implement in practice.
Despite these advantages, it has been observed empirically that the RL algorithms can be sensitive to time discretization \cite{tallec2019making, munos2006policy, park2021time, baird1994reinforcement}. As illustrated in Figure~\ref{fig:eq0}, such algorithms often yield poor approximations (see Section \ref{sec:linear_galerkin} for the details of Figure \ref{fig:eq0}). 
Crucially, this failure is not due to stochasticity or insufficient data; rather, it stems from a more fundamental issue: the (discrete-time) Bellman equation, is not a faithful approximation of the underlying continuous-time problem. We show in the paper that the solution to the Bellman equation is sensitive to time discretization, the change rate of the rewards and the discount coefficient as shown in Figure \ref{fig:eq0}. 
{Fundamentally, the (discrete-time) Bellman equation is designed for discrete-time decision-making processes, where only the discrete-time transition dynamics are utilized. In contrast, the continuous-time dynamics considered in this paper are governed by a SDE, a structure that the (discrete-time) Bellman equation does not capture. To faithfully approximate the true continuous-time value function, it is therefore necessary to move beyond the standard MDP and (discrete-time) Bellman-equation frameworks.}

%The central question we aim to address in this paper is whether, with the same discrete-time information and the same computational cost, one can approximate the true solution more accurately than the Bellman equation.

{The central question we aim to address in this paper is: if the dynamics are known to follow a standard SDE but only discrete-time data are available, can we design algorithms that \emph{exploit this structure} to decisively outperform generic RL methods for policy evaluation?}

{First, we view the continuous-time PE problem in terms of the solution to a PDE. Then, we tackle the overall goal in two steps. The first step is to approximate the PDE operator with the discrete-time transition dynamics, and the second step is to solve the resulting equation in a model-free way.}
In the first step, we proposed a PDE-based Bellman equation, called PhiBE, which integrates discrete-time information into a continuous PDE.  %The core concept revolves around utilizing discrete-time data to approximate the dynamics rather than the value function. Furthermore, we extend this framework to higher-order PhiBE, which enhances the approximation of the true value solution with respect to the time discretization. 
PhiBE combines the advantages of both existing frameworks while mitigating their limitations.
First, PhiBE is a PDE that preserves the continuous-time differential equation structure throughout the learning process.
At the same time, PhiBE relies only on discrete-time transition distributions, which is similar to the BE, making it directly compatible with discrete-time data and enabling model-free algorithm design.
Moreover, our framework overcomes key drawbacks of existing methods. 
Unlike the continuous-time PDE approach, which requires estimating the continuous-time dynamics and may suffer from identifiability issues, PhiBE depends solely on discrete-time transitions and does not require explicit dynamic modeling. 
Compared with the BE, PhiBE yields a more accurate approximation of the exact solution as it utilizes the underlying differential equation structure. In addition to that, PhiBE is more robust to rapidly changing reward functions compared to BE, which allows greater flexibility in designing reward functions to effiectively achieve RL objectives. 
Moreover, higher-order PhiBE achieves comparable error to BE with sparser data collection, enhancing the efficiency of RL algorithms. 
As illustrated in Figure \ref{fig:eq0}, when provided with the same discrete-time information, the exact solution derived from PhiBE is closer to the true value function than BE. 

{In the second step,} we introduce a model-free algorithm for approximating the solution to PhiBE when only discrete-time data are accessible. {This algorithm can be viewed as a continuous-time counterpart to the least-squares temporal difference (LSTD) \cite{bradtke1996linear} by replacing the (discrete-time) Bellman equation with the PhiBE equation. Since PhiBE enjoys a smaller discretization error than the Bellman equation, the resulting algorithm achieves a significantly improved approximation accuracy for free with a slightly modification of the standard RL algorithm.} As depicted in Figure \ref{fig:eq0}, with exactly the same data and the same computational cost, the proposed algorithm outperforms the RL algorithms drastically.

\begin{figure}
\centering
     \begin{subfigure}[b]{0.32\textwidth}
         \centering
         \includegraphics[width=\textwidth]{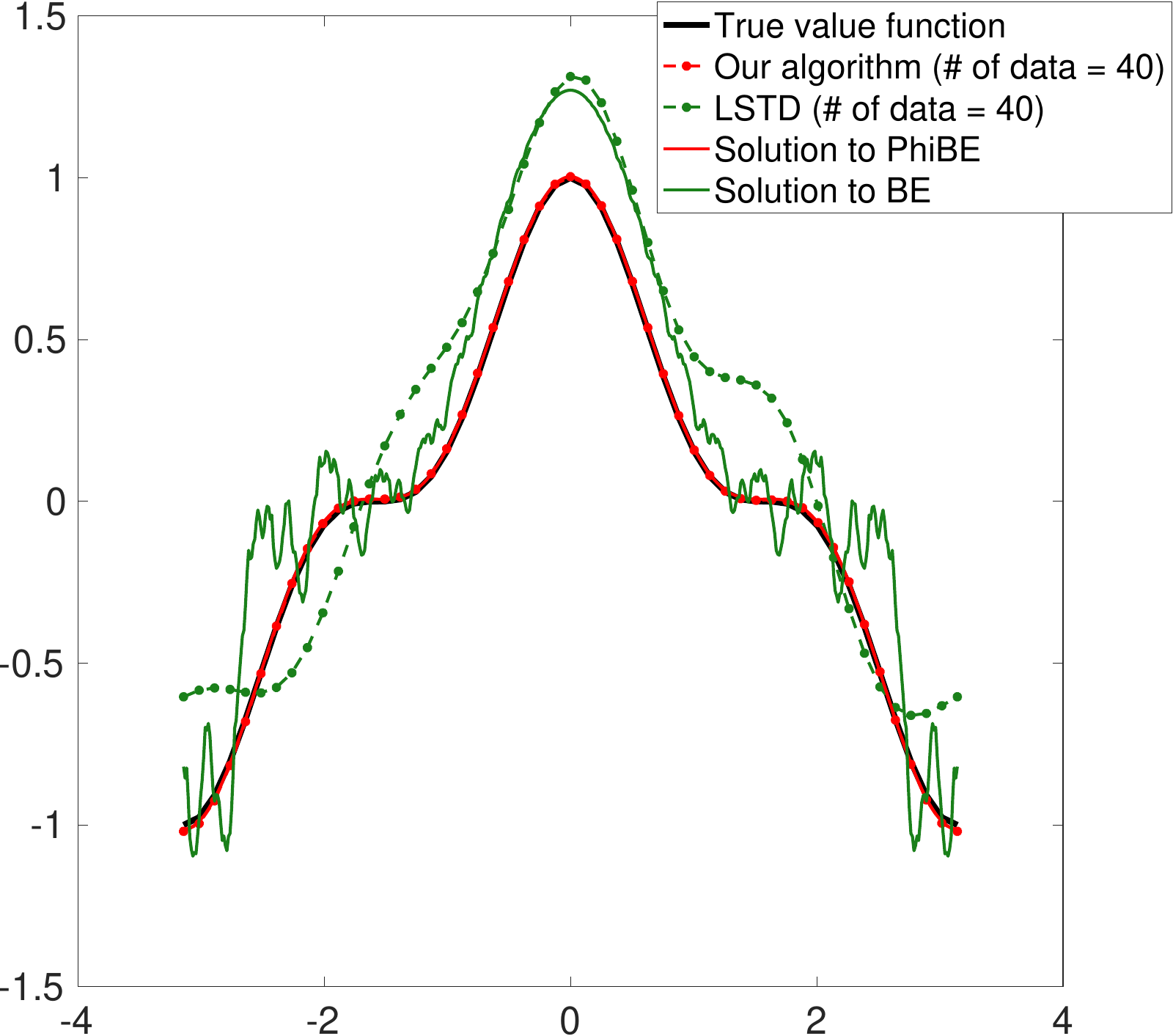}
         \caption{$\dt = 5, \b = 0.1, V(s) = \cos^3(s)$.}
     \end{subfigure}
     \hfill
     \begin{subfigure}[b]{0.30\textwidth}
         \centering
         \includegraphics[width=\textwidth]{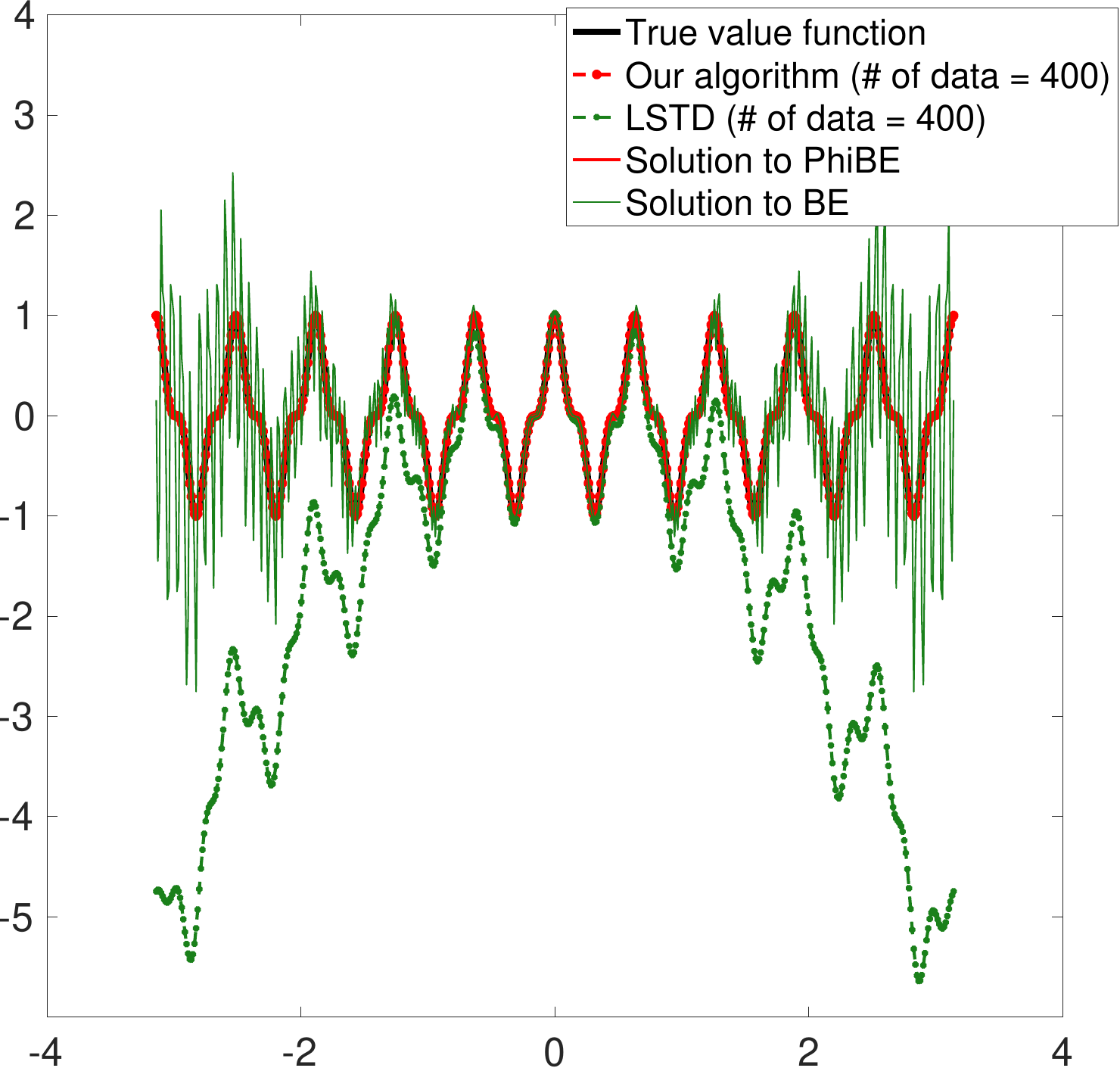}
         \caption{$\dt = 0.5, \b = 0.1, \\ V(s) = \cos^3(10s)$.}
     \end{subfigure}
     \hfill
     \begin{subfigure}[b]{0.32\textwidth}
         \centering
         \includegraphics[width=\textwidth]{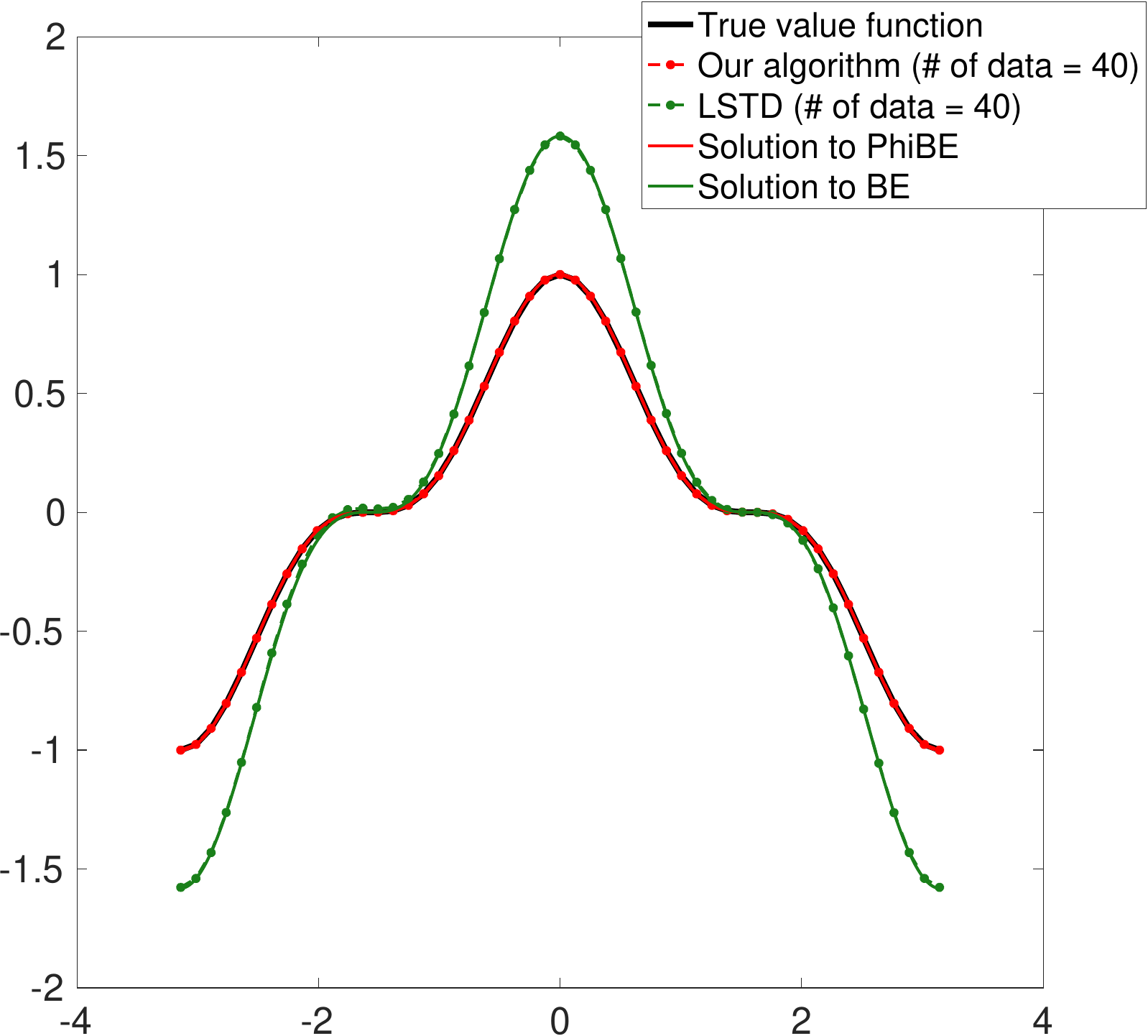}
         \caption{$\dt = 0.1, \b = 10, V(s) = \cos^3(s)$.}
     \end{subfigure}
     \caption{Here the data are collected every $\dt$ unite of time, $\b$ is the discount coefficient, and $V(s)$ is the true value function. In our setting, a larger discount coefficient indicates that future rewards are discounted more. LSTD \cite{bradtke1996linear} is a popular RL algorithm for linear function approximation. The PhiBE is proposed in Section \ref{sec:phibe} and the algorithm is proposed in Section \ref{sec:algo}. %We use the same data and linear bases in our algorithm as in LSTD.
     }
     \label{fig:eq0} 
\end{figure}

The error analysis for continuous-time reinforcement learning differs fundamentally from that of classical discrete-time RL. 
For continuous-time PE under function approximation, the total error can be decomposed into three components, as illustrated in Figure~\ref{fig:error-decomposition}. The first component is the \emph{discretization error}, which arises from the mismatch between the underlying continuous-time dynamics and the discrete-time observations used for learning. 
The second component is the \emph{algorithmic error}, which may be caused by model misspecification or by finite iterations when iterative solvers are employed. 
In this work, we use an offline solver, and therefore the algorithmic error refers to a pure model misspecification error.
The third component is the \emph{finite-sample error}, which results from having only a finite number of observations. We denote by $V^{n,\Phi}_{\Delta t}$ the value function approximation obtained from $n$ samples collected at time interval $\dt$ using linear bases $\Phi(s)$. 
As $n \to \infty$, the estimator converges to $V^{\Phi}_{\Delta t}$, which still differs from the true value function $V$ due to discretization error and the limited expressiveness of the bases.
As the approximation space is enriched, i.e., the number of bases $p\to\infty$, $V^{\Phi}_{\Delta t}$ further converges to $V_{\Delta t}$, the solution of the discretized Bellman equation in the MDP framework or, the new proposed equation, PhiBE in our paper.
Finally, as the data collection frequency $\dt \to 0$, $V_{\Delta t}$ converges to the true continuous-time value function $V$.
Most of the existing reinforcement learning literature focuses on analyzing the finite-sample and algorithmic errors, namely the gap between $V^n_{\Delta t}$ and $V_{\Delta t}$ \cite{munos2008finite,yang2019sample,jin2020provably,cai2020provably,zhang2021convergence}. 
In contrast, comparatively little work has addressed the discretization error arising from the time discretization itself.

In this paper, we analyze all three error components within the proposed PhiBE framework and its associated algorithm.
In addition, we explicitly characterize the discretization error incurred by the MDP formulation, namely the gap between the (discrete-time) Bellman equation and the true continuous-time value function. %Within the MDP framework, the discretization error corresponds to the gap between the solution of the Bellman equation and the true continuous-time value function.
In another word, the discretization error associated with the (discrete-time) Bellman equation represents the best approximation achievable by the standard policy evaluation algorithms derived from the (discrete-time) Bellman equation, including temporal difference learning \cite{sutton1988learning}, least-squares temporal difference methods \cite{bradtke1996linear}, and gradient-based temporal difference methods \cite{sutton2009fast}.
%We do not reanalyze the finite-sample and algorithmic errors for the MDP framework, as these aspects have been extensively studied in the existing literature.

Our analysis leads to two notable observations.
First, in terms of discretization error, PhiBE can provide a more accurate approximation than the MDP framework when the reward function exhibits high oscillation and the system dynamics evolve smoothly.
This regime is common in practice because reward functions are often designed to vary sharply in order to distinguish rewards from penalties, while physical systems typically evolve smoothly over time.
These properties suggest that the MDP framework may be suboptimal for certain continuous-time control problems, a limitation that has also been observed empirically \cite{de2024idiosyncrasy}.
Second, for continuous-time PE, we identify an inherent trade-off with respect to the time discretization $\Delta t$. More frequent data collection reduces discretization bias but increases statistical variance, requiring more samples to achieve the same accuracy.
To the best of our knowledge, the only prior work reporting a similar trade-off is \cite{zhang2023managing}, which provides analytical results only for linear quadratic policy evalution without function approximation.
In contrast, our results apply to general controlled diffusion processes and incorporate function approximation.

\begin{figure}
\centering
{\includegraphics[width=0.8\textwidth]{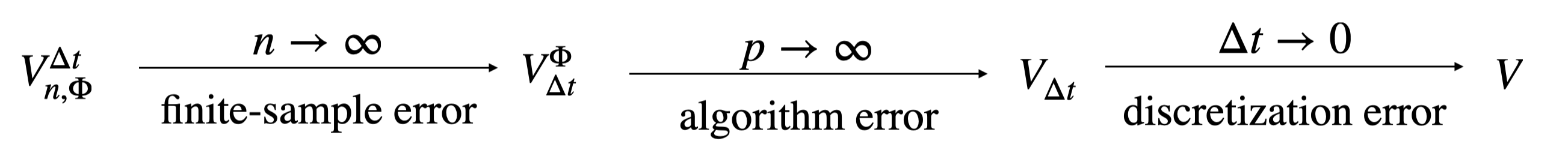}}
\caption{Error decomposition for continuous-time RL}
\label{fig:error-decomposition} 
\end{figure}

\color{black}
\paragraph{Contributions}
{
\begin{itemize}
\item We show that the classical Bellman equation, which does not exploit the underlying continuous-time structure, is a first-order approximation to the continuous-time PE problem, whose error deteriorates when the reward oscillates. To address this limitation, we develop the PhiBE framework that explicitly utilizes the structure of the underlying dynamics, yielding provably smaller approximation error than the Bellman equation when rewards oscillate or dynamics evolve slowly. Furthermore, we derive higher-order extensions to further enhance the approximation accuracy.
\item We propose a Galerkin method to approximate the solution to the PhiBE under linear function approximation when the discrete-time transition dynamics are given. We further establish its convergence under model misspecification. In contrast to the divergent results in existing RL analyses as the data frequency $\dt \to 0$, we show, by exploiting the smooth structure of the underlying dynamics, that the approximation error remains well-conditioned and independent of $\dt$.
\item We propose a model-free algorithm to approximate the Galerkin solution to the PhiBE when only data are available. We further provide a sample complexity bound guaranteeing an $\e$-accurate approximation to the Galerkin solution. In classical discrete-time RL analyses, the error scales as $O(\e^{-2}\b^{-4}\dt^{-4})$.  By contrast, our analysis exploits the smoothness of the underlying SDE and the uniform ellipticity of the infinitesimal generator, improving the dependence on the time discretization from $O(\dt^{-4})$ to $O(\dt^{-1})$.
\item Combining all the error, we identify an inherent trade-off with respect to the time discretization $\dt$ in continuous-time PE. More frequent data collection reduces discretization bias but increases statistical variance, requiring more samples to achieve the same accuracy. To the best of our knowledge, such an explicit trade-off, together with a nonasymptotic characterization of its dependence on $\Delta t, n$ and $\beta$, has not been systematically analyzed in prior analyses.
\item The PhiBE formulation integrates discrete-time information with continuous-time PDE. By leveraging the continuous-time structure, it achieves smaller discretization error than classical RL. Moreover, due to its structural similarity to the discrete-time Bellman equation, existing RL algorithms can be seamlessly adapted to the PhiBE framework, achieving improved approximation accuracy essentially for free.
\end{itemize}
}

%{Add a remark on the benefit in discretization error will be dominate in sample complexity}

\subsection{Related Work}
\subsubsection{Continuous-time RL}
There are primarily two approaches to address continuous-time RL from the continuous-time stochastic optimal control perspective. One involves employing machine learning techniques to learn the dynamics from discrete-time data and subsequently transforming the problem into model-based continuous-time RL problem \cite{kamalapurkar2016model, della2023model,lee2021policy, vamvoudakis2010online}. However, the underlying continuous dynamics are often unidentifiable when only discrete-time data are available.
{Another approach involves developing an algorithm for the continuous-time optimal control problem and then adapts it to discrete-time data \cite{basei2022logarithmic,jia2021policy, jia2022policy}.  Typically, numerical summation in discrete time is used to approximate continuous integrals (e.g., Algorithm 2 in~\cite{basei2022logarithmic} or Equation (19) in~\cite{jia2021policy}). While these methods offer valuable theoretical insights, two issues remain underexplored.    First, most analyses guarantee convergence to the true value function when continuous-time data are available, but do not explicitly account for discretization error arising when only discrete samples are available.  Second, although these works assume dynamics governed by a standard SDE, the resulting algorithms are not tailored to this structure and remain broadly applicable to other dynamics. 
By contrast, the central question of this paper is: when the dynamics follow a standard SDE but only discrete data are available, can we exploit this structure to build algorithms outperform generic RL methods? To the best of our knowledge, this is the first work to address this question systematically and establish corresponding theoretical guarantees.}

%Another approach involves an algorithm that converges to the true value function using continuous-time information and then discretizes it when only discrete-time data are available. For instance, \cite{basei2022logarithmic} presents a policy gradient algorithm tailored for linear dynamics and quadratic rewards. \cite{jia2021policy, jia2022policy} introduces a martingale loss function for continuous-time PE and RL.  These algorithms converge to the true value function when continuous-time data is available.

The proposed method fundamentally differs from the existing literature in three key aspects. First, unlike model-based approaches, which end up solving a PDE with continuous-time information, we integrate discrete-time information into the PDE formulation. Second, unlike alternative methodologies that view it as a Markov reward process, %directly approximating the value function using discrete-time data, 
which neglects the smoothness of the underlying dynamics, our method results in a PDE that incorporates gradients of the value function, which ensures that the solution closely approximates the true value function under smooth dynamics. {Third, in contrast to most continuous-time RL algorithms that require long trajectory data, our formulation depends only on the current and next states. Consequently, algorithms derived from it require access only to transition pairs rather than full trajectories, making the approach more flexible and broadly applicable to diverse data sources.}

{
To sum up, our work complements this line of research by bridging the gap between continuous-time algorithms and discrete-time data. 
We treat the discrete-time data before the algorithm design. 
The PhiBE formulation offers a systematic way to approximate the continuous-time problem from discrete-time data.
First, from a theoretical standpoint, because our learning algorithms retain the continuous-time PDE form, existing convergence analyses developed in the PDE \cite{guo2025policy,kerimkulov2020exponential,tang2023policy, TZZ} can be directly applied or adapted to analyze the proposed methods. 
Second, from a numerical perspective,  owing to the structural similarity between the PhiBE and the (discrete-time) Bellman equation, both depend only on the current and next states, most existing RL algorithms developed for solving Bellman equations can be directly applied to PhiBE. }

{
\subsubsection{Numerical schemes for HJB equation}
Although the formulation of the PhiBE closely resembles the PDE associated with the continuous-time PE problem, our method is fundamentally different from traditional numerical schemes for solving PDEs in three key aspects.
First, classical schemes require full knowledge of the system dynamics, whereas our model-free algorithm does not.
Second, numerical schemes approximate trajectories based on known dynamics, while our approach leverages exact trajectory data to approximate the underlying dynamics. Consequently, the corresponding error analysis framework also differs.
Third, the objectives of error analysis are distinct. Numerical schemes typically focus on the convergence order of the discretization, whereas one of the aims in this paper is to compare the first-order PhiBE formulation with the discrete-time MDPs formulation. Since both are first-order formulations, our analysis emphasizes the leading constant, which ultimately determines the practical performance of continuous-time PE. (See more discussion after Remark \ref{rmk:a-diff}.)
}

{
\subsubsection{Model-free algorithms for RL}
Due to the structural similarity between PhiBE and (discrete-time) Bellman equation, our model-free algorithm for solving the PhiBE is closely related to model-free algorithms for the BE, particularly the LSTD method \cite{bradtke1996linear}. Conceptually, the proposed solver serves as the continuous-time analogue of LSTD by replacing the (discrete-time) Bellman equation with the PhiBE equation. Both approaches are model-free, employ linear function approximation, and can be derived via a Galerkin projection followed by a stochastic approximation procedure. This connection highlights an important feature of the PhiBE framework: most existing RL algorithms developed for the BE can be readily adapted to the PhiBE setting with minimal modification. Moreover, because the PhiBE formulation incurs a smaller discretization error than its BE, it achieves higher approximation accuracy given sufficient data (see Remark~\ref{rmk:similarity to lstd} for detailed discussion).
From a theoretical perspective, our error analysis fundamentally departs from that of LSTD. The classical model-misspecification error for LSTD \cite{lazaric2012finite} diverges as the time discretization $\dt$ goes to zero. In contrast, our bound remains stable by exploiting the ellipticity of the infinitesimal generator, yielding a well-conditioned characterization of model misspecification in the continuous-time limit (see more discussion after Theorem \ref{thm:galerkin err}). In addition to that, we also improve the sample complexity bound from $\dt^{-4}$ to $\dt^{-1}$
}

\paragraph{Organization} The setting of the problem is specified in Section \ref{sec:setting}. Section \ref{sec:phibe} introduces the PDE-based bellman equation, PhiBE, and establishes the discretization error of PhiBE and (discrete-time) BE.  In Section \ref{sec:algo}, a model-free algorithm for solving the PhiBE is proposed, algorithmic error and finite-sample error are derived. Numerical experiments are conducted in Section \ref{sec:numerics}. 

\color{black}
\paragraph{Notation}
For function $f(s)\in \R$, $\ll f \rl_\Linf = \sup_{s} f(s)$; $\ll f \rl_\rho = \sqrt{\int f(s)^2 \rho(s)ds }$. 
For vector function $f(s)\in\R^d$, $\nb f(s)\in \R^{d\times d}$ with $(\nb f(s))_{ij} = \pt_{s_i}f_j$; $\ll f \rl_2^2 = \sum_{i=1}^d f_i^2$ represents the Euclidean Norm; $\ll f \rl_\Linf =\sqrt{\sum_{i=1}^d \ll f_i(s) \rl^2_\Linf}$, similar with $\ll f \rl_\rho$. For matrix function $f(s)\in \R^{d\times d}$, $\nb\cdot f = \sum_{l=1}^d\pt_{s_l}f_{i,l}$.

\section{Setting}\label{sec:setting}
Consider the following continuous-time policy evaluation (PE) problem, where the value function  $V(s)\in\R$, defined as follows, is the expected discounted cumulative reward starting from $s$, 
\begin{equation}\label{def of value}
    V(s) = \E\l[\int_{0}^\infty e^{-\b t}r(s_t)dt| s_0 = s\r].
\end{equation}
Here $\b>0$ is a discounted coefficient, $r(s)\in\R$ is a reward function, and the state $s_t\in\S = \R^d$ is driven by the stochastic differential equation (SDE),  
\begin{equation}\label{def of dynamics}
    d s_t = \mu(s_t)dt + \sigma(s_t)dB_t,
\end{equation} 
with unknown drift function $\mu(s)\in \R^d$ and unknown diffusion function $\sigma(s)\in \R^{d\times d}$. 
In this paper, we assume that $\mu(s), \s(s)$ are Lipschitz continuous and the reward function $\ll r \rl_\Linf$ is bounded. This ensures that \eqref{def of dynamics} has a unique strong solution \cite{oksendal2013stochastic} and the infinite horizon integral is bounded.

\begin{remark}
Continuous-time reinforcement learning \cite{wang2020reinforcement} can be formulated as follows:
\[
    \begin{aligned}
        \max_{a_t = \pi(s_t)}\quad &V^\pi(s) = \E\l[\int_{0}^\infty e^{-\b t}r(s_t, a_t)dt| s_0 = s\r]\\
        s.t.\quad &ds_t = \mu(s_t, a_t) + \sigma(s_t, a_t)dB_t, \quad a_t = \pi(s_t).
    \end{aligned}
\]
where the goal is to find the optimal policy $\pi(s)$ that maximizes the value function under unknown drift and diffusion.  
In this paper, we focus on the unregularized optimal control problem and therefore restrict attention to deterministic feedback policies $a_t = \pi(s_t)$. 
For stochastic policies (see, e.g., \cite{jia2021policy, wang2020reinforcement}), additional challenges may arise (see, e.g., \cite{jia2025accuracy, szpruch2024optimal}), which we leave for future study.  

The policy evaluation problem is a fundamental step in reinforcement learning algorithms, such as actor-critic and PPO \cite{konda1999actor, schulman2017proximal}.  
Its goal is to assess the quality of a fixed policy $\pi$ by evaluating the value function under that policy.  
The continuous-time PE problem \eqref{def of value}--\eqref{def of dynamics} considered in this paper arises by denoting  
$
r^\pi(s) = r(s, \pi(s)), \quad 
\mu^\pi(s) = \mu(s, \pi(s)), \quad 
\sigma^\pi(s) = \sigma(s, \pi(s)).
$
For simplicity, we omit the subscript $\pi$ throughout, as the policy is fixed in the entire paper.  
\end{remark}

\color{black}
We aim to approximate the continuous-time value function $V(s)$ when only discrete-time information is available. To be more specific, we consider the following two cases:
\begin{itemize}
    \item [Case 1. ] The transition distribution $\rho(s',\dt|s)$ in discrete time $\dt$, driven by the continuous dynamics \eqref{def of dynamics}, is given. Here $\rho(s',\dt|s)$ represents the probability density function of $s_\dt$ given $s_0 = s$. 
    \item [Case 2. ] Trajectory data generated by the continuous dynamics \eqref{def of dynamics} and collected at discrete time $j\dt$ is given.  Here the trajectory data  $s = \{s^l_{0},s^l_{\dt}, \cdots, s^l_{m\dt}\}_{l=1}^I$ contains $I$ independent trajectories, and the initial state $s^l_0$ of each trajectory are sampled from a distribution $\rho_0(s)$. 
\end{itemize}

%When the discrete transition distribution is given (Case 1), one can explicitly formulate the {(discrete-time)} Bellman equation. One can also estimate the discrete transition distribution from the trajectory data, which is known as model-based RL. 
The error analysis in Section \ref{sec:phibe} is conducted under Case 1. We demonstrate that the  {(discrete-time)} Bellman equation is not always the optimal equation to solve continuous-time PE problems even when the discrete-time transition dynamics are known, and consequently, all the RL algorithms derived from it are not optimal either. To address this, we introduce a Physics-informed Bellman equation (PhiBE) and establish that its exact solution serves as a superior approximation to the true value function compared to the classical Bellman equation. When only trajectory data is available (Case 2), one can also use the data to solve the PhiBE, referred to as model-free RL, which will be discussed in Section \ref{sec:algo}.

%\begin{assumption}\label{basic ass stoch}   Lipschitz continuous\end{assumption} Before we dive into the approximation of the value function, we will make the following assumptions so that one has the existence and uniqueness $s_t$, $V(s)$, and the stationary distribution of the stochastic dynamics. 

\section{A PDE-based Bellman Equation (PhiBE)}\label{sec:phibe}
In Section \ref{sec:be}, we first introduce the {(discrete-time)} Bellman equation, followed by an error analysis to demonstrate why it is not always a good approximation. Then, in Section \ref{subsec:phibe}, we propose the PhiBE, a PDE-based Bellman equation, considering both the deterministic case (Section \ref{sec:deterministic}) and the stochastic case (Section \ref{sec:stochastic}). The error analysis provides guidance on when PhiBE is a better approximation than the BE.

\subsection{{(Discrete-time)} Bellman equation}\label{sec:be}
By approximating the definition of the value function \eqref{def of value} in discrete time, one obtains the approximated value function,
\begin{equation}\label{eq:numerical-integral}
    \tV(s) =  \E\l[\sum_{j =  0 }^\infty e^{-\b \dt j}r(s_{j\dt })\dt |s_0 = s\r].
\end{equation}
In this way, it can be viewed as a policy evaluation problem in Markov Decision Process, where the state is $s\in\S$, the reward is $r(s)\dt$, and the discount factor is $e^{-\b \dt}$ and the transition dynamics is $\rho(s',\dt|s)$. Therefore, the approximated value function $\tV(s)$ satisfies the following {(discrete-time) Bellman equation \cite{sutton2018reinforcement}. In this paper, we will simply refer to it as the Bellman equation (abbreviated as BE).}
\begin{definition}[Definition of BE]
    \begin{equation}\label{bellman}
    \tV(s) = r(s)\dt +  e^{-\b \dt}\E_{s_\dt \sim \rho(s',\dt|s)}[\tV(s_{\dt})|s_0 = s].
\end{equation}
\end{definition}

When the discrete-time transition distribution is not given, one can utilize various RL algorithms to solve the Bellman equation \eqref{bellman} using the trajectory data. Note that some continuous-time RL papers approximate the value function by $V(s) \approx \sum_{i=0}^T e^{-\beta i\dt} r(s_{i\dt})$, which in fact is also approximating the solution of the Bellman equation, owing to the equivalence between~\eqref{bellman} and~\eqref{eq:numerical-integral}. 
However, if the exact solution to the Bellman equation is not a good approximation to the true value function, then all the RL algorithms derived from it will not effectively approximate the true value function. In the theorem below, we provide an upper bound for the distance between the solution $\tV$ to the above BE and the true value function $V$ defined in \eqref{def of value}.

\begin{theorem}[Discretization error for BE]\label{thm:rl}
Assume that $\ll r\rl_\Linf,\ll  \mL_{\mu,\Sig} r \rl_\Linf$ are bounded, then the solution $\tV(s)$ to the BE \eqref{bellman} approximates the true value function $V(s)$ defined in \eqref{def of value} with an error 
    \[
    \ll  V(s) - \tV(s) \rl_\Linf \leq \frac{\frac12(\ll \mL_{\mu,\Sig} r\rl_\Linf + \b \ll r \rl_\Linf)}{\b} \dt +o(\dt),
    \]
    where 
    \begin{equation}\label{def of mL}
    \mL_{\mu,\Sig} = \mu(s)\cdot\nb + \frac12\Sig : \nb^2,
    \end{equation}
    with $\Sig = \s \s^\top$, and $\Sig : \nb^2 = \sum_{i,j}\Sig_{ij}\pt_{s_i}\pt_{s_j}$.
\end{theorem}
\begin{remark}[Assumptions on $\ll  \mL_{\mu,\Sig} r \rl_\Linf$]
    One sufficient condition for the assumption to hold is that $\ll \mu(s) \rl_\Linf$, $\ll \Sig(s) \rl_\Linf, \ll \nb^k r(s) \rl_\Linf $ for $k=0,1, 2$ are all bounded. However, $\ll  \mL_{\mu,\Sig} r \rl_\Linf$ is less restrictive than the above and allows, for example, linear dynamics  $\mu(s) = \lam s, \Sig = 0$, with the derivative of the reward decreasing faster than a linear function at infinity, $\ll s\cdot \nb r(s) \rl_\Linf \leq C$.
\end{remark}

The proof of the theorem is given in Section \ref{sec:proof of thm rl}. In fact, by expressing the true value function $V(s)$ as the sum of two integrals, one can more clearly tell where the error in the BE comes from. Note that $V(s)$, as defined in \eqref{def of value}, can be equivalently written as,
\begin{equation}\label{eq:twointegral}
    \begin{aligned}
    V(s)=&\E\l[\int_0^\dt e^{-\b t} r(s)dt + \int_{\dt}^\infty e^{-\b t} r(s_t)dt|s_0 = s\r]\\
    =&\E\l[\int_0^\dt e^{-\b t} r(s)dt|s_0 = s\r] + e^{-\b \dt}\E\l[V(s_{t+\dt})|s_0 = s\r]
\end{aligned}
\end{equation}
One can interpret the Bellman equation defined in \eqref{bellman} as an equation resulting from approximating $\E\l[\frac{1}{\dt}\int_0^\dt e^{-\b t} r(s_t)dt|s_0 = s\r]$ in \eqref{eq:twointegral} by $ r(s)$. The error between these two terms can be bounded by:
\begin{equation}\label{rlerror}
    \lv \E\l[\l.\l(\frac{1}{\dt}\int_0^\dt e^{-\b t} r(s_t)dt\r) - r(s_0)\r|s_0 = s\r] \rv \leq \frac{1}{2}\l(\b \ll r \rl_\Linf + \ll \mL_{\mu,\Sig} r\rl_\Linf\r)\dt + o(\dt),
\end{equation}
characterizes the error of $\ll V - \tV \rl_\Linf$ in Theorem \ref{thm:rl}. %This is the intuition why the error is first order in terms of $\dt$, and it will heavily depend on the derivatives of the reward function. 

Theorem \ref{thm:rl} indicates that the solution $\tV$ to the Bellman equation \eqref{bellman} approximates the true value function with a first-order error of $O(\dt)$. Moreover, the coefficient before $\dt$ suggests that for the same time discretization $\dt$, when $\b$ is small, the error is dominated by the term $\ll \mL_{\mu,\Sig}r(s)\rl_\Linf$, indicating that the error increases when the reward changes rapidly. Conversely, when $\b$ is large, the error is mainly affected by $\ll r \rl_\Linf$, implying that the error increases when the magnitude of the reward is large.

{Fundamentally, the Bellman equation is designed for discrete-time decision-making processes and depends only on the discrete-time transition kernel $\rho(s',\dt | s)$. This kernel may arise from any underlying mechanism. As a result, the Bellman equation does not utilize or encode any assumption about how $\rho(s',\dt \mid s)$ is generated. In particular, it does not exploit the assumption made in this paper that the discrete-time dynamics are induced by a continuous-time SDE. To incorporate this continuous-time structure into the approximation, one must go beyond the standard MDP framework.
}
The question that the rest of this section seeks to address is whether,  given the same discrete-time information, i.e., the transition distribution $\rho(s',\dt|s)$, time discretization $\dt$, and discount coefficient $\b$, one can combine the differential equation structure with the discrete-time transition dynamics and achieve a more accurate estimation of the value function?

\subsection{A PDE-based Bellman equation}\label{subsec:phibe} 
In this section, we introduce a PDE-based Bellman equation, referred to as PhiBE. We begin by discussing the case of deterministic dynamics in Section \ref{sec:deterministic} to illustrate the idea clearly. Subsequently, we extend our discussion to the stochastic case in Section \ref{sec:stochastic}.

\subsubsection{Deterministic Dynamics}\label{sec:deterministic}
When $\s(s) \equiv 0$ in \eqref{def of dynamics}, the dynamics become deterministic, which can be described by the following ODE, 
\begin{equation}\label{deter_dyn}
\frac{ds_t}{dt} = \mu(s_t).
\end{equation}
If the discrete-time transition dynamics $p(s',\dt|s) = p_\dt(s)$ is given, where $p_{\dt}(s)$ provides the state at time $t+\dt$ when the state at time $t$ is $s$, then the BE in deterministic dynamics reads as follows,
\[
\frac{1}{\dt}\tV(s) = r(s) + \frac{e^{-\b \dt}}{\dt}\tV(p_{\dt}(s)).
\]

The key idea of the new equation is that, instead of approximating the value function directly, one approximates the dynamics. First note that the value function defined in \eqref{def of value} can be equivalently written as {the solution to the following PDE,
\begin{equation}\label{deter_V}
\b V(s) = r(s) + \mu(s)\cdot \nb V(s).
\end{equation}
The above equation is known as the Hamilton–Jacobi–Bellman (HJB) equation, often referred to as the continuous-time Bellman equation. In this paper, we will consistently refer to it as the HJB equation. A detailed derivation can be found in Appendix~\ref{appendix:derivation} or in classical control textbooks such as \cite{stroock1997multidimensional}.
}
Applying a finite difference scheme, one can approximate $\mu(s_t)$ by
\[
\mu(s_t) = \frac{d}{dt}s_t \approx  \frac{1}{\dt}(s_{t+\dt} - s_t) =\frac{1}{\dt}( p_\dt(s_t) - s_t),
\]
and substituting it back into \eqref{deter_V} yields
\[
\b \hV(s_t) = r(s_t) +  \frac{1}{\dt}(s_{t+\dt} - s_t)\cdot \nb \hV(s_t).
\] 
Alternatively, this equation can be expressed in the form of a PDE as follows,
\begin{equation}\label{def of first order_deter}
    \b \hV(s) = r(s) +  \frac{1}{\dt}(p_\dt(s) - s)\cdot \nb \hV(s),
\end{equation}
Note that the error now arises from
\[
\lv \mu(s_t) - \frac{s_{t+\dt} - s_t}{\dt} \rv \leq \frac{1}{2}\ll \mu \cdot \nb \mu \rl_\Linf \dt,
\]
which only depends on the dynamics.  As long as the dynamics change slowly, i.e.,  the acceleration of dynamics $\ll \frac{d^2}{dt^2}s_t\rl_\Linf = \ll \mu \cdot \nb \mu \rl_\Linf$ is small, the error diminishes.

We refer to \eqref{def of first order_deter} as PhiBE for deterministic dynamics, an abbreviation for the physics-informed Bellman equation, because it incorporates both the current state and the state after $\dt$, similar to the Bellman equation, while also resembling the form of the PDE \eqref{deter_V} derived from the true continuous-time physical environment. However, unlike the Bellman equation, which is blind to the underlying SDE structure, PhiBE explicitly exploits this structure by involving the gradient of the value function, thereby embedding the smoothness induced by the continuous-time dynamics. At the same time, unlike the true PDE \eqref{deter_V}, which requires direct access to the continuous-time dynamics and is therefore non-identifiable from  discrete-time information, PhiBE depends only on the discrete-time transition kernel. This makes PhiBE amenable to model-free algorithm design while still leveraging continuous-time structure.

%PhiBE combines both continuous PDE form and discrete transition information $p_{\Delta t}(s)$.

One can derive a higher-order PhiBE by employing a higher-order finite difference scheme to approximate $\mu(s_t)$. For instance, the second-order finite difference scheme
\[
\mu(s_t) \approx \hmu_2(s_t):=\frac{1}{\dt}\l[ -\frac12(s_{t+2\dt} - s_t) + 2(s_{t+\dt} - s_t)\r],
\]
resulting in the second-order PhiBE,
\begin{equation*}
\b \hV_2(s) = r(s) +  \frac{1}{\dt}\l[ -\frac12(p_{\dt}(p_{\dt}(s)) - s) + 2(p_{\dt}(s) - s)\r] \cdot \nb\hV_2(s).
\end{equation*}
In this approximation, $\ll \mu(s) - \hmu_2(s) \rl_\Linf $ has a second order error $O(\dt^2)$.  We summarize $i$-th order PhiBE in deterministic dynamics in the following Definition. 

\begin{definition}[i-th order PhiBE in deterministic dynamics]\label{def: higher order deter} 
When the underlying dynamics are deterministic, then the $i$-th order PhiBE is defined as,
\begin{equation}\label{phibe_deter}
    \b \hV_i(s) = r(s) + \hmu_i(s) \cdot \nb \hV_i(s),
\end{equation}
where
\begin{equation}\label{def of higher order deter mu}
    \hmu_i(s) = \frac1\dt\sum_{j=1}^i\coef{i}_j\l(\underbrace{p_{\dt}\circ\cdots\circ p_{\dt}}_{j}(s) - s\r),
\end{equation}
and 
\begin{equation}\label{def of A b}
(\coef{i}_0,\cdots, \coef{i}_i)^\top = A_i^{-1}b_i, \quad\text{with }(A_i)_{kj} = j^k, \ (b_i)_k = \l \{
\begin{aligned}
&0, \quad k \neq 1\\
&1, \quad k = 1
\end{aligned}\r. \text{ for } 0\leq j, k \leq i.
\end{equation}
\end{definition}

\begin{remark}\label{rmk:a-diff}
Note that $\hmu_i(s)$ can be equivalently written as
\[
\hmu_i(s) = \frac1\dt \sum_{j=1}^i\coef{i}_j[s_{j\dt} - s_0|s_0 = s] .
\]
There is an equivalent definition of $\coef{i}$, given by
\begin{equation}\label{def of a}
   \sum_{j=0}^i\coef{i}_j j^k = \l \{
\begin{aligned}
&0, \quad k \neq 1,\\
&1, \quad k = 1,
\end{aligned}\r.\quad\text{ for } 0\leq j, k \leq i.
\end{equation}
\end{remark}

\paragraph{Discussion on the relation to numerical schemes for the HJB equation.}
The PhiBE formulation is fundamentally different from classical numerical schemes for the HJB equation in three key aspects.  

First of all, the goal of two approaches are different. Classical numerical schemes aim to approximate the solution of the PDE
$\mathcal{L}_{\mu,\sigma} V(s) = 0$, where the operator $\mathcal{L}_{\mu,\sigma}$ is explicitly given.  In contrast, the goal of our paper is not only must we approximate the solution to the PDE, but we must do so without fully knowing the PDE operator. We tackle the overall goal in two steps. The first step is to approximate the PDE operator $\mathcal{L}_{\mu,\sigma}$, and the second step is to solve the resulting equation. The PhiBE formulation \eqref{phibe_deter} is constructed to address the first task, namely approximating the unknown operator $\mathcal{L}_{\mu,\sigma} V(s) = 0$ by $\mathcal{L}_{\hat{\mu},\hat{\sigma}} V(s) = 0$ using only partial information in the form of the discrete-time transition kernel $\rho(s',\dt \mid s)$, with no direct access to $\mu(s), \sigma(s)$. 
Consequently, the ``discretization error'' analyzed in this section refers to the error incurred from observing $\rho(s',\dt \mid s)$ at $\dt > 0$ rather than the infinitesimal kernel $\lim_{\dt \to 0} \rho(s',\dt \mid s)$. It does not refer to the numerical error of solving the PDE.

Second, the {discretization formulations} differ. Classical schemes discretize the state space over a uniform mesh $h$, e.g., $V'(s) \approx \tfrac{1}{h}(V(s+h)-V(s))$, while such uniform meshing cannot be directly applied in RL because trajectory data are irregularly distributed. Consequently, numerical error in classical schemes arises from spatial discretization, whereas in PhiBE it originates from time discretization through the approximation of $\hat \mu$ and $\hat \Sigma$. Moreover, although the difinition of $\hat \mu(s)$ superficially resembles a finite-difference expression, its meaning is fundamentally different. In classical methods, one starts from known dynamics $\mu(s)$ and uses numerical schemes to approximate the trajectory $s_{j\Delta t}$. As a result, numerical error propagates and accumulates over time. In contrast, PhiBE takes the opposite approach: the trajectories $s_{j\Delta t}$ are observed from data and used to approximate the underlying dynamics $\mu(s)$. This inversion of roles leads to a distinct analytical treatment of convergence and convergence rates for $\hat V_i(s)$ that does not accumulate over time, in sharp contrast to classical numerical discretization errors.

Finally, the {error analysis frameworks and emphasis} are distinct. Classical analyses establish convergence by verifying that the discrete operator is monotone, stable, and consistent, then applying comparison principles and perturbation arguments to derive convergence rates. In contrast, our analysis first quantifies the discrepancy between the data-driven dynamics $(\hat \mu, \hat \Sigma)$ and the true dynamics $(\mu, \Sigma)$, a component never addressed in traditional numerical analyses, and then propagates this discrepancy to the value function via PDE stability. Furthermore, while classical numerical analyses focus on the {order of accuracy} in the mesh size $h$ and typically omit constants, our comparison with discrete-time RL requires explicit characterization of these constants, as both standard RL and PhiBE are first-order approximations. The leading constant, which depends on the reward and dynamics, ultimately determines the practical performance of continuous-time RL algorithms.

In summary, PhiBE differs fundamentally from classical numerical schemes for HJB equations. Nevertheless, because the continuous-time RL problem is connected to a PDE through the PhiBE formulation, one could, in principle, estimate $\rho_{\Delta}$ from data, compute $(\hat \mu, \hat \Sigma)$ explicitly, and then apply standard numerical schemes to solve the resulting PhiBE equation. The classical convergence results would then apply to this \emph{model-based} approach. However, this paper will focus on developing a \emph{model-free} algorithm for solving the PhiBE equation directly from data without explicitly estimating $\hat \mu$ or $\hat \Sigma$.

\color{black}

The error analysis of PhiBE in the deterministic dynamics is established in the following theorem.
\begin{theorem}[Discretization error for determinisitc PhiBE]\label{thm:v-vhat_deter}
Assume that $\ll \nb r(s) \rl_\Linf$, $\ll\mL^i_{\mu} \mu(s) \rl_\Linf$  are bounded. Additionally, assume that $\ll \nb \mu(s) \rl_\Linf<\b$, then the solution $\hV_i(s)$ to the PhiBE \eqref{phibe_deter} is an ith-order approximation to the true value function $V(s)$ defined in \eqref{def of value} with an error 
\[
\ll \hV_i(s) -  V(s) \rl_\Linf \leq 2C_i \frac{\ll \nb r\rl_\Linf \ll \mL^i_{\mu}\mu\rl_\Linf}{(\b - \ll \nb \mu \rl_\Linf)^2} \dt^i,
\]
where 
\begin{equation}\label{def of mLmu}
    \mL_\mu = \mu\cdot \nb,
\end{equation}
and $C_i$ is a constant defined in \eqref{def of C_i} that only depends on the order $i$.
\end{theorem}
See Section \ref{sec:proof of them Phibe deter} for the proof of Theorem \ref{thm:v-vhat_deter}. Several remarks regarding the above theorem are in order. 
\paragraph{1st-order PhiBE v.s. BE}
By Theorem \ref{thm:v-vhat_deter}, the distance between the first-order PhiBE solution and the true value function can be bounded by
\[
\ll \hV_1 - V \rl_\Linf \leq \frac{2\ll \nb r \rl_\Linf \ll \mu \cdot\nb \mu \rl_\Linf }{(\b - \ll \nb \mu \rl_\Linf)^2}\dt.
\]
Comparing it with the difference between the BE solution and the true value function in deterministic dynamics,
\[
\ll \tV - V \rl_\Linf \leq \frac{\ll \mu\nb r \rl_\Linf +\b \ll r \rl_\Linf }{2\b }\dt, 
\]
one observes that when the change of the reward  is rapid, i.e., $\ll \nb r \rl_\Linf$ is large,  but the change in velocity is slow, i.e.,  $\ll \frac{d^2}{dt^2}s_t\rl_\Linf = \ll \mu\cdot\nb \mu \rl_\Linf$ is small, even though both  $\hV_1$ and $\tV$ are first-order approximations to the true value function, $\hV_1$ has a smaller upper bound. 
\paragraph{Higher-order PhiBE}
The advantage of the higher-order PhiBE is two-fold. Firstly, it provides a higher-order approximation, enhancing accuracy compared to the first-order PhiBE or BE.  Secondly, as demonstrated in Theorem \ref{thm:v-vhat_deter}, the approximation error of the $i$-th order PhiBE decreases as  $\ll \mL^i_\mu\mu\rl_\Linf$ decreases. {If} the \textquotedblleft acceleration", i.e., $\frac{d^2}{dt^2}s_t = \mL_\mu\mu$, of the dynamics is large but the change in acceleration, i.e., $\frac{d^3}{dt^3}s_t = \mL^2_\mu \mu$, is slow, then the error reduction with the second-order PhiBE will be even more pronounced in addition to the higher-order error effect.

\paragraph{Dependence on $\beta$}
When $\beta$ is large, PhiBE always has a smaller discretization error than BE. However, when $\beta$ is small, the BE error scales as $O(\beta^{-1})$, whereas the PhiBE error scales as $O(\beta^{-2})$. 
Thus, when $\beta$ is small, the gap in discretization error
between the two methods naturally narrows. However, as shown later in
Section~\ref{sec:algo}, the finite-sample error of the model-free algorithms for both BE and
PhiBE scales as $O(\beta^{-2})$. As a result, in data-driven settings, the potential advantage of BE for small $\beta$ is dominated
by finite-sample error scales as $O(\beta^{-2})$. In contrast, PhiBE retains an advantage in regimes with slowly varying dynamics and rapidly changing reward functions.

In additional to that, in standard RL settings, the discount factor is $\gamma = e^{-\beta \Delta t} \in [0.9,0.99]$, so the resulting 
range of $\beta$ depends on the sampling frequency. When the data are collected at $\Delta t = 0.01$, we expect an $\beta = O(1)$, i.e.,
$\beta \in [1,10]$. For data collected at $\Delta t = 0.1$, one obtains a smaller $\beta \in [0.1,1]$. 
In the experiments in Section~\ref{sec:numerics}, we therefore evaluate both methods across a broad range 
$\beta \in [0.1,10]$. 
\paragraph{When does PhiBE outperform BE?}
In summary, PhiBE outperforms BE in terms of discretization error whenever one of the following conditions is present:
(i) the dynamics evolve slowly;
(ii) the reward function varies sharply; or
(iii) the discount rate is not extremely small.
In summary, these features arise frequently in continuous-time RL applications. First, slow dynamics are common in physical and biomedical systems, where state variables evolve gradually and smoothly. For instance, in glucose–insulin regulation, blood glucose levels follow physiological ODE models with inherently slow drift dynamics.
Second, the reward function often needs to sharply penalize excursions outside a healthy range, so that the optimal policy strongly encourages keeping the patient within a safe region. This leads 
to a highly varying reward landscape, i.e., a large $\|\nabla_s r(s)\|$. 
Third, discount coefficients used in practice are rarely associated with extremely small continuous-time rates. In OpenAI Gym environments with $\Delta t \sim 10^{-2}$, the typical choice $\gamma \in [0.9,0.99]$ corresponds to $\beta \in [1,10]$. In clinical settings with larger sampling intervals, objectives often focus on maintaining stability over short or moderate horizons, so the effective discount rate is similarly not small.
Taken together, the three features are common in real-world applications, making the settings where PhiBE outperforms BE not exceptional but the typical regime for continuous-time reinforcement learning.

\color{black}
\paragraph{Assumptions on $\ll\mL^i_{\mu} \mu(s) \rl_\Linf$}
Note that the boundedness assumption of $\ll\mL^i_{\mu} \mu(s) \rl_\Linf$ is required in general to establish that $\hmu_i(s)$ is an $i$-th order approximation to $\mu(s)$. 
A sufficient condition for $\ll\mL^i_{\mu} \mu(s) \rl_\Linf$ being bounded is that $\ll \nb^k \mu(s) \rl_\Linf$ are bounded for all $0\leq k \leq i$. 
Note that the linear dynamics $\mu(s) = A s$ does not satisfy this condition. We lose some sharpness for the upper bound to make the theorem work for all general dynamics. However, we prove in Theorem \ref{thm: linear deter} that PhiBE works when $\mu(s) = A s$, and one can derive a sharper error estimate for this case.

%When the underlying dynamics are linear, one can conduct a sharper error analysis for PhiBE.

\begin{theorem}\label{thm: linear deter}
    When the underlying dynamics follows
    $\frac{d}{dt}s_t = A s_t$, where $s_t\in \R^d$ and $A\in \R^{d\times d}$,
    then the solution to the $i$-th order PhiBE in deterministic dynamics approximates the true value function with an error
    \[
    \ll \hV_i - V \rl_\Linf \leq \frac{C_i}{\beta^2}\ll s\cdot\nb r(s) \rl_\Linf \ll A \rl_2 D_A^{i}\dt^i , \quad D_A = e^{\ll A \rl_2 \dt} \ll A \rl_2
    \]
    where $C_i$ is a constant defined in \eqref{def of C_i} that only depends on the order $i$.
    %If the underlying dynamics follows\[ds_t = \lam s_t dt + \sigma dB_t\]then the solution $\tV(s)$ to the Bellman equation approximates the true value function with an error \[\ll \tV  - V \rl_\rho \leq ;\]
    %while the solution $\hV_1$ to the PhiBE in deterministic dynamics approximates the true value function with an error\[\ll \hV_i - V \rl_\rho \leq ;\]
\end{theorem}

The proof of the above theorem is provided in Section~\ref{proof of them linear}. 
 According to Theorem~\ref{thm: linear deter}, 
even when the dynamics $s_t = e^{At} s_0$ evolve exponentially fast due to positive real parts 
in the eigenvalues of $A$, the PhiBE solution remains a good approximation to the true value 
function provided that $D_A \Delta t$ is small. In particular, when $D_A \Delta t < 1$, 
higher-order PhiBE achieves a smaller approximation error than the first-order PhiBE. %When the linear dynamics change slowly, i.e., $\lam$ is smaller, the first-order PhiBE is a superior approximation to the Bellman equation.

%{when one has the discrete-time dynamics, we one do better than the above? If one has the mapping $f_{\dt}(s)$ that maps the state at time $t$ to the state at time $t+\dt$, can one indicate $\mu(s)$?}

\color{black}
\subsubsection{Stochastic dynamics}\label{sec:stochastic}
When $\sigma(s)\not\equiv 0$ is a non-degenerate matrix, then the dynamics is stochastic and driven by the SDE in \eqref{def of dynamics}. By Feynman–Kac theorem \cite{stroock1997multidimensional}, the value function $V(s)$  satisfies the following HJB equation,
\begin{equation}\label{stoch_PDE}
    \b V(s) = r(s) + \mL_{\mu,\Sig} V(s),
\end{equation} 
where $\mL_{\mu,\Sig}$ is an operator defined in \eqref{def of mL}. However, one cannot directly solve the PDE \eqref{stoch_PDE} as $\mu(s), \s(s)$ are unknown. In the case where one only has access to the discrete-time transition distribution $\rho(s',\dt|s)$, we propose an $i$-th order PhiBE in the stochastic dynamics to approximate the true value function $V(s)$.

\begin{definition}[i-th order PhiBE in stochastic dynamics]\label{def: higher order stoch}
When the underlying dynamics are stochastic, then the $i$-th order PhiBE is defined as,
\begin{equation}\label{def of high order stoch}
    \b \hV_i(s) = r(s) + \mL_{\hmu_i,\hs_i} \hV_i(s),
\end{equation}
where 
\begin{equation}\label{def of higher order stoch mu}
\begin{aligned}
    &\hmu_i(s) =\frac1\dt \sum_{j=1}^i\E_{s_{j\dt}\sim \rho(\cdot, j\dt|s)}\l[\coef{i}_j(s_{j\dt} - s_0)|s_0 = s\r]\\
    &\hs_i(s) =\frac1\dt \sum_{j=1}^i\E_{s_{j\dt}\sim \rho(\cdot, j\dt|s)}\l[\coef{i}_j(s_{j\dt} - s_0)(s_{j\dt} - s_0)^\top|s_0 = s\r]    
\end{aligned}
\end{equation}
where $\mL_{\hmu_i,\hs_i}$ is defined in \eqref{def of mL}, and $\coef{i} = (\coef{i}_0,\cdots, \coef{i}_i)^\top$ is defined in \eqref{def of A b}.
\end{definition}
\begin{remark}
    There is another $i$-th order approximation for $\Sig(s)$,
    \[
    \ts_i(s) = \frac{1}{\dt}\sum_{j=0}^i \coef{i}_j \l(\E[s_{j\dt}^\top s_{j\dt}|s_0 = s] - \E[s_{j\dt}|s_0 = s]^\top\E[s_{j\dt}|s_0 = s]\r).
    \]
    However, the unbiased estimate for $\ts(s)$ 
    \[
    \ts_i(s_0) \approx \frac{1}{\dt}\sum_{j=1}^i \coef{i}_j \l(s_{j\dt}^\top s_{j\dt} - s'_{j\dt}s_{j\dt}\r)
    \]
    requires two independent samples $s_{j\dt}, s'_{j\dt}$ starting from $s_0$, which are usually unavailable in the RL setting. This is known as the \textquotedblleft double Sampling" problem. One could apply a similar idea in \cite{zhu2020borrowing, zhu2022borrowing} to alleviate the double sampling problem when the underlying dynamics are smooth, that is, approximating $s'_{j\dt} \approx s_{(j-1)\dt}+(s_{(j+1)\dt} - s_{j\dt})$. However, it will introduce additional bias into the approximation. We leave the study of this approximation $\ts_i(s)$ or the application of BFF on $\ts_i(s)$ for future research.   
\end{remark}

The first and second-order approximations are presented as follows. The first-order approximation reads,
\[
\hmu_1(s) = \frac1\dt\E\l[(s_{\dt} - s_0)|s_0 = s\r], \quad \hs_1(s) =\frac1\dt \E\l[(s_{\dt} - s_0)(s_{\dt} - s_0)^\top|s_0 = s\r]    ;
\]
and the second-order approximation reads,
\[
\begin{aligned}
    &\hmu_2(s) = \frac1\dt\E\l[2(s_{\dt} - s_0)-\frac12(s_{2\dt} - s_0)|s_0 = s\r],\\
    &\hs_2(s) =\frac1\dt \E\l[2(s_{\dt} - s_0)(s_{\dt} - s_0)^\top-\frac12(s_{2\dt} - s_0)(s_{2\dt} - s_0)^\top|s_0 = s\r]  .
\end{aligned}  
\]

Next, we show the solution $\hV_i(s)$ to the $i$th-order PhiBE provides an $i$th-order approximation to the true value function $V(s)$, under the following assumptions.
\begin{assumption} \label{ass_2} Assumptions on the dynamics:
    \begin{itemize}
        \item [(a)] $\lammin(\Sig(s))>\lammin > 0$ for $\forall s\in\S$, {where $\lammin(A)$ is the smallest eigenvalue of $A$}. 
        \item [(b)] $\ll \nb^k \mu(s) \rl_\Linf ,\ll \nb^k \Sig(s) \rl_\Linf $ are bounded for $0\leq k \leq 2i$.
    \end{itemize}
\end{assumption}
The first assumption ensures the coercivity of the operator $\mL_{\mu, \Sig}$, which is necessary to establish the regularity of $V(s)$. The second assumption is employed to demonstrate that $\hmu_i$ and $\hs_i$ are $i$-th approximations to $\mu,\Sig$, respectively. 

To establish the error analysis, 
we define a weighted $L^2$ inner product and norm under a probability density $\rho(s)$ by
\begin{equation}\label{def of rho}
\la f, g\ra_\rho = \int f(s) g(s)\rho(s) ds, \qd \ll f \rl^2_\rho = \int f^2(s) \rho(s)ds.
\end{equation}
In addition, the density $\rho$ satisfies
\begin{equation}\label{def of Lrho}
    L_\rho = \ll \nb \log \rho \rl_\rho \text{is bounded}.
\end{equation}
One may take $\rho(s)$ to be the stationary distribution (when it exists), namely a density satisfying
\begin{equation}\label{stationary}
    \int\mL_{\mu,\Sig}\phi(s) \rho(s)ds = 0.\qd \text{for }\forall \phi(s)\in C_c^\infty.
\end{equation}
If a stationary distribution does not exist, one may select a density $\rho$ such that
\begin{equation}\label{def of new rho}
\ll - \mu  + \frac12\nb\cdot\Sig + \frac12\Sig\nb
\log\rho\rl_{\Linf(\Omega)}^2 \leq \frac12\beta\lammin, \quad 
\end{equation}
where $\Omega=\{s: \rho(s)\neq 0\}$. If $\Omega \neq \R^d$, then all assumptions and results hold on $\Omega$, rather than on the entire space $\mathbb{R}^d$.

\begin{remark}
Several remarks are in order for the weight $\rho(s)$.  
First, classical results (e.g., \cite{huang2015steady}) imply that when the drift and diffusion satisfy a suitable dissipation condition, a stationary distribution exists.  Furthermore, once a stationary distribution exists, Theorem~1.1 of \cite{bogachev1996regularity} yields the bound
$L_\rho \;\le\; \frac{1}{\lambda_{\min}}
    \,\|\mu + \nabla\!\cdot\!\Sigma\|_{L^\infty}.$

Second, if no stationary distribution exists, one can still choose a density $\rho$ such that \eqref{def of new rho} holds.  
For example, in one of our numerical experiments \eqref{def of ou}, the system admits no stationary solution when $\lambda>0$.  
In this case, one may take
\[
\rho(s)
=
\begin{cases}
\displaystyle \frac{1}{2c}, & s\in[-c,c],\\[4pt]
0, & s\notin[-c,c],
\end{cases}
\qquad 
|c|\le \frac{\beta\sigma^2}{2\lambda},
\]
which ensures \eqref{def of new rho}.  
The trade-off is that under this choice of weighted norm, one can guarantee convergence of the approximation error only on the domain $\Omega=\{s: \rho(s)\neq 0\}$, rather than on the entire space $\mathbb{R}^d$. At the same time, Assumption~\ref{ass_2} on the drift and diffusion is also relaxed: instead of requiring it to hold globally, it only needs to be satisfied on the domain $\Omega$.

Third, if one chooses to work with the Lebesgue measure on $\Omega \subset \mathbb{R}^d$, then, in addition to Assumption~\ref{ass_2} on the drift and diffusion terms, we require
\[
\ll - \mu  + \frac12\nb\cdot\Sig \rl_{\Linf(\Omega)}^2 \leq \frac12\beta\lammin.
\]
This ensures that the requirement \eqref{def of new rho} remains valid when the density $\rho$ is taken to be the (unnormalized) Lebesgue measure.
\end{remark}

\color{black}

The error analysis for BE in the weighted $L^2$ norm is presented in the following theorem. 
\begin{theorem}[Discretization error for BE in weighted $L^2$ norm]\label{thm:rl_rho}
Assume that $\ll r\rl_\rho$, $\ll  \mL_{\mu,\Sig} r \rl_\rho$ are bounded, then the solution $\tV(s)$ to the BE \eqref{bellman} approximates the true value function $V(s)$ defined in \eqref{def of value} with an error 
    \[
    \ll  V(s) - \tV(s) \rl_\rho \leq \frac{2(\ll \mL_{\mu,\Sig} r\rl_\rho + \b \ll r \rl_\rho)}{\b} \dt +o(\dt).
    \]
\end{theorem}
The proof of Theorem \ref{thm:rl_rho} is given in Section \ref{sec:proof of them rl stoch}. Next, the error analysis for PhiBE in stochastic dynamics is presented in the following theorem. 

\begin{theorem}[Discretization error for stochastic PhiBE]\label{thm:v-vhat_stoch}
Under Assumption \ref{ass_2}, and $\dt^i \leq D_{\mu,\Sig,\b}$, the solution $\hV_i(s)$ to the i-th order PhiBE \eqref{def of high order stoch} is an i-th order approximation to the true value function $V(s)$ defined in \eqref{def of value} with an error 
\[
\ll V(s) - \hV_i(s) \rl_{\rho} \leq \l(\frac{ C_{r,\mu,\Sig} }{\b^2} +\frac{ \h{C}_{r,\mu,\Sig} }{\b^{3/2}} \r)\dt^i,
\]
where $D_{\mu,\Sig,\b}, C_{r,\mu,\Sig}, \h{C}_{\mu,\Sig,\b}$ are constants defined in \eqref{def of D}, \eqref{def of Crmusig} depending on $\mu, \Sig, r$.
\end{theorem}
The proof of Theorem \ref{thm:v-vhat_stoch} is given in Section \ref{sec:proof of them Phibe stoch}. 
\begin{remark}[PhiBE v.s. BE]
    Here we discuss two cases. The first case is when the diffusion is known, that is, $\hs_i = \Sig$, then the distance between the PhiBE and the true value function can be bounded by 
    \[
    \begin{aligned}
        &\ll \hV_i - V \rl_\rho 
        \lesssim \frac{1}{\b^2}\ll \mL_{\mu,\Sig}^i\mu\rl_\Linf \l[ \l(\frac{\ll \nb \Sig\rl_\Linf}{\lammin} + \sqrt{\frac{\ll \nb \mu\rl_\Linf}{\lammin}}\r)\ll r \rl_{\rho} + \ll \nb r\rl_{\rho}\r]\dt^i.
    \end{aligned}
    \]
    Similar to the deterministic case, the error of the $i$-th order PhiBE proportional to the change rate of the dynamics $\E[\frac{d^{i+1}}{dt^{i+1}}s_t] = \mL^i_{\mu,\Sig}\mu$. One can refer to the discussion under Theorem \ref{thm:v-vhat_deter} for the benefit of the 1st-order PhiBE and higher-order PhiBE with respect to different dynamics. 
    
    The second case is when both drift and diffusion are unknown. Then  the distance between the first-order PhiBE and the true value function can be bounded by 
    \[
    \begin{aligned}
        &\ll \hV_1 - V \rl_\rho \\
        \lesssim &\frac{\dt}{\b^2}\l[\l(L_\mu + L_{\nb\cdot\Sig} +\frac{L_\Sig}{\lammin}\ll \mu + \nb\cdot\Sig \rl_\Linf\r)  \l( \sqrt{\frac{C_\nb}{\lammin}}\ll r \rl_{\rho,\Linf} + \ll \nb r\rl_{\rho,\Linf}\r)\r]\\
        &+ \frac{\dt}{\b^{3/2}} \frac{1}{\sqrt{\lammin}}L_\Sig \l(\sqrt{\frac{C_\nb}{\lammin}}\ll r \rl_\rho + \ll \nb r \rl_\rho\r),
    \end{aligned}
    \]
    where
    \[
    \begin{aligned}
    &L_\mu \lesssim \ll \mL\mu \rl,\qd L_\Sig \lesssim \ll \mu \mu^\top + \Sig\nb\mu + \mL\Sig \rl_\Linf + \ll \mu\rl_\Linf,\\
    &L_{\nb\cdot\Sig} \lesssim \sqrt{\frac{C_{\nb}}{\lammin}}\ll \mu \mu^\top + \Sig\nb\mu + \mL\Sig \rl_\Linf +\ll \nb\cdot\l(\mu \mu^\top + \Sig\nb\mu + \mL\Sig \r)\rl_\Linf +  \sqrt{\frac{C_{\nb}}{\lammin}}\ll \mu \rl_\Linf \\
    &+ \ll \nb\mu \rl_\Linf ,\\
    &\sqrt{\frac{C_{\nb}}{\lammin}} \lesssim \frac{ \ll \nb\mu \rl_\Linf }{2\lam} + \sqrt{\frac{\ll \nb\Sig \rl_\Linf}{\lammin}},\qd \ll r \rl_{\rho,\Linf} = \ll r \rl_\rho + \ll r \rl_\Linf.
    \end{aligned}
    \]
    Here the operator $\mL$ represents $\mL_{\mu,\Sig}$. This indicates that when $\lammin$ is large or $\nb\mu, \nb\Sig$ are small, the difference between $\hV_1$ and $V$ is smaller. Comparing it with the upper bound $\ll \mL r \rl_\rho + \b \ll r \rl_\rho$ for the BE, which is more sensitive to large $\b$ and reward function, the PhiBE approximation is less sensitive to these factors. When the change in the dynamics is slow, or the noise is large, even the first-order PhiBE solution is a better approximation to the true value function than the BE. 
\end{remark}

\section{Model-free Algorithm for {PhiBE}}\label{sec:algo}
In this section, we assume that one only has access to the discrete-time trajectory data $\{s^l_0,s^l_{\dt},\cdots s^l_{m\dt} \}_{l=1}^I$.  We first revisit the Galerkin method for solving PDEs with known dynamics in Section \ref{sec:galerkin}, and we provide the error analysis of the Galerkin method for PhiBE. Subsequently, we introduce a model-free Galerkin method in Section \ref{sec:galerkin_stoch}. Finally, we conclude this section with an end-to-end error analysis for the model-free algorithm in Section \ref{sec:total-error}.

\subsection{Galerkin Method}\label{sec:galerkin}
Given $p$ bases $\{\phi_i(s)\}_{i=1}^p$, the objective is to find an approximation $\bV(s) = \Phi(s)^\top\th$ to the solution $V(s)$ of the PDE,
\[
\b V(s) - \mL_{\mu,\Sig} V(s) = r(s),
\]
where $\th\in\R^p, \Phi(s) = (\phi_1(s), \cdots, \phi_p(s))^\top$, and $\mL_{\mu,\Sig} $ is defined in \eqref{def of mL}. The Galerkin method involves inserting the ansatz $\bV$ into the PDE and then projecting it onto the finite bases, 
\[
\la \b \bV(s) - \mL_{\mu,\Sig} \bV(s) , \Phi(s)\ra_\rho = \la r(s), \Phi(s)\ra_\rho,
\]
where the same measure $\rho$ defined in \eqref{def of rho} is used here. This results in a linear system of $\th$,
\[
A \th =b, \quad A = \la \b\Phi(s) - \mL_{\mu,\Sig} \Phi(s) , \Phi(s)\ra_\rho, \quad b =  \la r(s), \Phi(s)\ra_\rho.
\]
When the dynamics $\mu(s), \Sig(s)$ are known, one can explicitly compute the matrix $A$ and the vector $b$, and find the parameter $\th = A^{-1}b$ accordingly. 

In continuous-time PE problems, one does not have access to the underlying dynamics $\mu, \Sig$. However, the approximated dynamics $\hmu_i, \hs_i$ is given through PhiBE. Therefore, if one has access to the discrete-time transition distribution, then the parameter $\th = \hat{A}_i^{-1} b$ can be solved for by approximating $A$ with $\h{A}_i$
\begin{equation}\label{def of galerkin A}
    \h{A}_i = \la \b\Phi(s) - \mL_{\hmu_i,\hs_i}\Phi(s), \Phi(s)\ra_\rho,
\end{equation}
where $\hmu_i, \hs_i$ are defined in \eqref{def of higher order stoch mu}.
We give the error estimate of the Galerkin method for PhiBE in the following theorem. 
\begin{theorem}[Galerkin Error]\label{thm:galerkin err}
    The Galerkin solution $\hV^G_i(s) = \th^\top\Phi(s)$ satisfies
    \begin{equation}\label{galerkin}
        \la (\b - \mL_{\hmu_i,\hs_i}) \hV^G_i(s), \Phi\ra_\rho = \la r(s), \Phi(s)\ra_\rho.
    \end{equation}
    For small $\dt^i \leq \min\l\{\eta_{\mu,\Sig,\b}, D_{\mu,\Sig,\b}\r\}$, the Galerkin solution $\hV^G_i(s)$ approximates the solution to the $i$-th order PhiBE defined in \eqref{def of high order stoch} with an error 
    \[
    \ll \hV^G_i(s)  - \hat{V}_i(s) \rl_\rho \leq  \frac{C_G}\b\inf_{V^P = \th^\top\Phi}\ll \hV_i - V^P \rl_{H^1_\rho}. 
    \]
    where $\eta_{\mu,\Sig,\b}, C_G, D_{\mu,\Sig,\b}$ are constants depending on $\mu, \Sig, \b, L^\infty_\rho, L^\infty_\Phi$ defined in \eqref{def of eta}, \eqref{def of c galerkin}, \eqref{def of D} respectively, and $\ll f \rl_{H^1_\rho} = \ll f \rl_\rho + \ll \nb f \rl_\rho$.
\end{theorem}
The proof of the Theorem is given in Section \ref{proof of galerkin error}. 

Several remarks are in order.
First, the above theorem establishes an approximation-error bound under model misspecification, that is, the error arising when the true value function cannot be exactly represented within the chosen linear bases $\Phi(s)$. More specifically, the bound in Theorem~\ref{thm:galerkin err} can be interpreted as the limiting error of the data-driven algorithms (Algorithms~\ref{algo:galerkin_phibe_deter} and~\ref{algo:galerkin_phibe_stoch}) when infinite data are available.

Second, the solver \eqref{galerkin} serves as a continuous-time counterpart to LSTD \cite{bradtke1996linear} by replacing the discrete-time Bellman equation with the PhiBE equation. However, our bound is much tigher than the existing RL literature in terms of $\dt$.  In Theorem~1 of \cite{lazaric2012finite}, the approximation-error term 
$\displaystyle (1-\gamma^2)^{-1/2} \, \| v - \Pi v \|$ arises from model misspecification.  However, this bound diverges as $\dt\to0$ because $\gamma=e^{-\beta\dt}\to1$. This divergence stems from the use of the contraction property of the Bellman operator whose contraction factor is $\gamma$.  Such discrete-time arguments are sharp for general RL problems where the transition dynamics have no structural assumptions. 
However, such contraction-based analysis is not suitable for continuous-time RL problem considered in this paper for two reasons. First, as $\dt \to 0$, one would expect the approximation error to decrease; 
yet, since $\gamma \to 1$, the contraction property vanishes, causing the approximation error to instead worsen and eventually diverge. 
Second, the analysis framework fails to exploit the smoothness and structural information of the underlying SDE.

In contrast, our analysis exploits the ellipticity of the infinitesimal generator
$\mathcal{L}_{\hat \mu, \hat \Sigma}$ and yields an error bound $\frac{C_G}\beta \, \| v - \Pi v \|$, where $C_G$ is a independent of $\dt$. 
This establishes a stable and well-conditioned characterization of model-misspecification error in the continuous-time limit, in sharp contrast to existing discrete-time LSTD analyses. 
Finally, we emphasize that the improved dependence on $\dt$ does not stem from the PhiBE framework itself, but rather from the smoothness of the discrete-time transition dynamics induced by the SDE. We expect that a comparable model-misspecification bound can be derived for the Bellman equation when the discrete-time transitions arise from an underlying SDE, and we leave this direction for future work.

\begin{remark}
The Galerkin error remains bounded even when $\|\nabla \log \rho\|_{L^\infty}$ is unbounded.  However, since the sample-complexity analysis in this paper assumes boundedness of $\|\nabla \log \rho\|_{L^\infty}$, we move the Galerkin error theorem under unbounded $\|\nabla \log \rho\|_{L^\infty}$  assumption to Section~\ref{proof of galerkin error} to maintain internal consistency.  
The corresponding result is stated in Theorem~\ref{thm:unbounded rho galerkin error}.
\end{remark}

\color{black}

\iffalse
Combining the above theorem with Theorem~\ref{thm:v-vhat_stoch}, one can further bound the difference between the Galerkin solution and the true value function, as stated in the following corollary.

\begin{corollary}
    \label{coro:galerkin err}
    Under the same assumption in Theorem \ref{thm:galerkin err}, the Galerkin solution $\hV^G_i(s) = \th^\top\Phi(s)$ to \eqref{galerkin} approximates the true value function defined in \eqref{def of value} with an error 
    \[
    \ll \hV^G_i(s)  - V(s) \rl_\rho \leq \l(\frac{ C_{r,\mu,\Sig} }{\b^2} +\frac{ \h{C}_{r,\mu,\Sig} }{\b^{3/2}} \r)\dt^i+ \frac{{C}_G}\b\inf_{V^P = \th^\top\Phi}\ll \hV_i - V^P \rl_{H^1_{\rho,\infty}} . 
    \]
    where $C_{r,\mu,\Sig},  \h{C}_{r,\mu,\Sig}, \h{C}_G$ are the same constants defined in Theorems  \ref{thm:v-vhat_stoch} and \ref{thm:galerkin err}. 
\end{corollary}
\fi

\subsection{Model-free Galerkin method for PhiBE}\label{sec:galerkin_stoch}
When only discrete-time trajectory data is available, we first develop an unbiased estimate  $\bar{\mu}_i, \bar{\Sig}_i$ for $\hmu_i,\hs_i$ from the trajectory data, 
\begin{equation}\label{def of barmu}
    \begin{aligned}
    &\bar{\mu}_i(s_{j\dt}^l) = \frac{1}{\dt}\sum_{k=1}^i\coef{i}_k(s^l_{(j+k)\dt} - s^l_{j\dt}) ,\\
    &\bar{\Sig}_i(s_{j\dt}^l) = \frac{1}{\dt}\sum_{k=1}^i\coef{i}_k(s^l_{(j+k)\dt} - s^l_{j\dt})(s^l_{(j+k)\dt} - s^l_{j\dt})^\top,
\end{aligned} 
\end{equation}   
with $\coef{i}$ defined in \eqref{def of A b}.
Then, using the above unbiased estimate, one can approximate the matrix $\hat{A}$ and the vector $b$ by
\[
\begin{aligned}
    &\bar{A}_i = \sum_{l=1}^{I}\sum_{j = 0}^{m-i}\Phi(s^l_{j\dt})\l[ \b \Phi (s^l_{j\dt})- \mL_{\bar{\mu}_i(s^l_{j\dt}),\bar{\Sig}_i(s^l_{j\dt})}\Phi(s^l_{j\dt})\r]^\top,\\ 
    &\bar{b}_i = \sum_{l=1}^{I}\sum_{j = 0}^{m-i} r(s^l_{j\dt})\Phi(s^l_{j\dt}) .
\end{aligned}
\]
By solving the linear system $\bar{A}_i\th = \bar{b}_i$, one obtains the approximated value function $\bar{V}(s) = \Phi(s)^\top\th$ in terms of the finite bases. Note that our algorithm can also be applied to stochastic rewards or even unknown rewards as only observation of rewards is required at discrete time.   We summarize the model-free Galerkin method for deterministic and stochastic dynamics in Algorithm \ref{algo:galerkin_phibe_deter} and Algorithm \ref{algo:galerkin_phibe_stoch}, respectively.

\begin{algorithm}
\caption{Model-free Galerkin method for $i$-th order PhiBE in deterministic dynamics}
\label{algo:galerkin_phibe_deter}
\textbf{Given:} discrete time step $\dt$, discount coefficient $\b$, discrete-time trajectory data $\{(s^l_{j_\dt}, r^l_{j_\dt})_{j=0}^m\}_{l = 1}^I$ generated from the underlying dynamics, and a finite bases $\Phi(s) = (\phi_1(s), \cdots, \phi_p(s))^\top$.

\begin{algorithmic}[1]
        \State Calculate $\bar{A}_i$:
        \[
        \bar{A}_i = \sum_{l=1}^{I}\sum_{j = 0}^{m-i}\Phi(s^l_{j\dt})\l[ \b \Phi (s^l_{j\dt})- \bar{\mu}_i(s^l_{j\dt})\cdot \nb \Phi(s^l_{j\dt}) \r]^\top,
        \]
where $ \bar{\mu}_i(s^l_{j\dt})$ is defined in \eqref{def of barmu}.
        \State Calculate $\bar{b}_i$:
        \[
        \bar{b}_i = \sum_{l=1}^{I}\sum_{j = 0}^{m-i} r^l_{j\dt}\Phi(s^l_{j\dt}) .
        \]
        \State Calculate $\th$:
        \[
        \th = \bar{A}_i^{-1}\bar{b}_i.
        \]
    \State Output $\bar{V}(s) = \th^\top\Phi(s)$.
\end{algorithmic}
\end{algorithm}

\begin{algorithm}
\caption{Model-free Galerkin method for $i$-th order PhiBE in stochastic dynamics}
\label{algo:galerkin_phibe_stoch}
\textbf{Given:} discrete time step $\dt$, discount coefficient $\b$, discrete-time trajectory data $\{(s^l_{j_\dt}, r^l_{j_\dt})_{j=0}^m\}_{l = 1}^I$ generated from the underlying dynamics, and a finite bases $\Phi(s) = (\phi_1(s), \cdots, \phi_p(s))^\top$.

\begin{algorithmic}[1]
        \State Calculate $\bar{A}_i$:
        \[
        \bar{A}_i = \sum_{l=1}^{I}\sum_{j = 0}^{m-i}\Phi(s^l_{j\dt})\l[ \b \Phi (s^l_{j\dt})- \bar{\mu}_i(s^l_{j\dt})\cdot \nb \Phi(s^l_{j\dt}) - \frac12\bar{\Sig}_i(s^l_{j\dt}):\nb^2 \Phi(s^l_{j\dt})\r]^\top,
        \]
where $ \bar{\mu}_i(s^l_{j\dt}), \bar{\Sig}_i(s^l_{j\dt})$ are defined in \eqref{def of barmu}.
        \State Calculate $\bar{b}_i$:
        \[
        \bar{b}_i = \sum_{l=1}^{I}\sum_{j = 0}^{m-i} r^l_{j\dt}\Phi(s^l_{j\dt}) .
        \]
        \State Calculate $\th$:
        \[
        \th = \bar{A}_i^{-1}\bar{b}_i.
        \]
    \State Output $\bar{V}(s) = \th^\top\Phi(s)$.
\end{algorithmic}
\end{algorithm}

{
\begin{remark}[Comparison to the LSTD Algorithm]\label{rmk:similarity to lstd}
Algorithms \ref{algo:galerkin_phibe_deter}, \ref{algo:galerkin_phibe_stoch} serves as a continuous-time counterpart to LSTD \cite{bradtke1996linear} by replacing the discrete-time Bellman equation with the PhiBE equation. Both our algorithms and LSTD are model-free and employ linear basis functions.  In addition to that, both methods can be derived by applying the Galerkin method, followed by a stochastic approximation of the resulting system \cite{szepesvari2011least}.  

This connection, however, highlights a key advantage of our formulation. 
Owing to the structural similarity between the PhiBE \eqref{phibe_deter}, \eqref{def of high order stoch} and the discrete-time Bellman equation, both depend only on the current and next states, most existing RL algorithms developed for solving Bellman equations can be directly applied to PhiBE.  Moreover, as shown in Section~\ref{sec:phibe}, our method enjoys a smaller discretization error, allowing one to achieve a significantly improved approximation accuracy \emph{for free} with a slightly modification of the RL algorithm.
\end{remark}
}

\subsection{Sample Complexity of the Model-Free Algorithm for First-Order PhiBE}\label{sec:total-error}

We now establish the sample complexity of the model-free algorithm for solving the first-order PhiBE. A sample complexity analysis for higher-order PhiBE is deferred to future work. We assume that the data consist of $n$ independent short trajectories $\{s_i, s_i'\}_{i=1}^n$. Each trajectory provides a single transition pair: an initial state $s_i \sim \rho(s)$ and its one-step successor $s_i' \sim \rho(s',\dt \mid s_i)$. Here the initial density $\rho(s)$ is the weight we defined in \eqref{def of rho}, and the discrete-time transition $\rho(s',\dt\mid s)$ is driven by the underlying SDE \eqref{def of dynamics}. The model-free algorithm approximates the value function by $\hat V_n(s) = \Phi(s)^\top \theta_n$,
where $\theta_n \in \mathbb{R}^p$ is obtained by solving the empirical linear system
\begin{equation}\label{1st model-free galerkin}
A_n\th_n = b_n, \quad A_n = \frac1n\sum_{i=1}^n\Phi(s_i)\l[ \b \Phi (s_i)- \mL_{\bar{\mu}_1^i,\bar{\Sig}_1^i}\Phi(s_i)\r]^\top, \quad b_n = \frac1n\sum_{i=1}^nr(s_i)\Phi(s_i),
\end{equation}
where 
\[
\bar{\mu}_1^i = \frac{1}{\dt}(s_i' - s_i) ,\quad 
\bar{\Sig}_1^i = \frac{1}{\dt}(s_i' - s_i)(s_i' - s_i)^\top.
\]
This estimator targets the Galerkin solution $\hat V_1^G(s) = \Phi(s)^\top \theta$ of the first-order PhiBE defined in \eqref{galerkin} when the discrete-time transition dynamics are known. Here $\th$ solves
\begin{equation}\label{1st galerkin}
    A\th = b, \quad A = \int \Phi(s)\l[ \b \Phi (s)- \mL_{\hat{\mu}_1,\hat{\Sig}_1}\Phi(s)\r]^\top\rho(s)ds, \quad b = \int r(s)\Phi(s)\rho(s)ds.
\end{equation}
The theorem below establishes both consistency and finite-sample guarantees.

\begin{theorem}[Sample Complexity]\label{thm:sample-complexity}
Assume that $\ll \nb^k\Phi \rl \leq L_{\Phi}$ $\rho$-a.s. for $k = 0,1, 2$,  and the Gram matrix for the bases satisfies
\begin{equation}\label{ass on bases}
G_\Phi = \int \Phi(s)\Phi(s)^\top \rho(s)\,ds \succeq \lambda_\Phi I .
\end{equation}
Suppose further that $\Delta t \le \eta_{\mu,\Sigma,\beta}$ (same $\eta_{\mu,\Sigma,\beta}$ in Theorem \ref{thm:galerkin err}). Then, for any sufficiently small $\e>0$, if
\[
n \gtrsim \frac{\log(2p/\delta)R^2L_\Phi^8\sigma_{\max}^2}{ \beta^4 \e^2\dt\lam_\Phi^4}
\]
where $R = \|r\|_{L^\infty}$ and $\sigma_{\max} = \|\sigma\|_{L^\infty}$, we have, with probability at least $1-\delta$,
\[
\|\hat V_n - \hat V_1^G\|_{L^\infty} \le \varepsilon .
\]
Equivalently, if $n \gtrsim \frac{L_\Phi^4 \log(2p/\delta)\sigma_{\max}^2\lam_\Phi^{-2}}{\beta^2\dt} $, 
then with probability $1-\delta$,
\[
\ll V_n - \hat{V}_1^G \rl_\Linf \leq \l(L_\Phi^4 R\sigma_{\max}\lam_\Phi^{-2}\sqrt{\log(2p/\delta)} \r)\frac{1}{\beta^2\sqrt{{n\dt}}}  
\]
\end{theorem}

The proof of the above Theorem is given in Section \ref{proof of sample-complexity}.
Several remarks are in order. 
The first part of Theorem~\ref{thm:sample-complexity} provides an explicit sample complexity bound guaranteeing an $\e$-accurate approximation to the Galerkin solution. The second part of the Theorem provides an finite-sample bound to the Galerkin solution, which scale as $O(\b^{-2}\dt^{-1/2} n^{-1/2})$. In classical discrete-time RL analyses (e.g., Theorem~1 of \cite{lazaric2012finite}), the error scales as $O((1-\gamma)^{-2}n^{-1/2})$. Note that their result is stated in the form $O((1-\gamma)^{-1}V_{\max}n^{-1/2})$. Since $V_{\max} = R(1-\gamma)^{-1}$,  the resulting error bound in fact scales as $(1-\gamma)^{-2}$ rather than $(1-\gamma)^{-1}$.
Substituting $\gamma = e^{-\beta \Delta t}$ leads to an order of $O(\b^{-2}\dt^{-2} n^{-1/2})$. By contrast, our analysis exploits the smoothness of the underlying SDE and the uniform ellipticity of the infinitesimal generator $\mathcal{L}_{\hat\mu_1,\hat\Sigma_1}$, improving the dependence on the time discretization from $\dt^{-2}$ to $\dt^{-1/2}$. 

\begin{corollary}[Overall Error Bound]\label{coro:overall-error}
Under the same assumption in Theorem \ref{thm:sample-complexity}, then with probability at least $1-\delta$,
\[
\ll V_n - V(s) \rl_\rho \lesssim \frac{ C_{r,\mu,\Sig} }{\b^2}\dt+ \frac{{C}_G}{\beta}\inf_{V^P = \th^\top\Phi}\ll \hV_i - V^P \rl_{H^1_{\rho,\infty}} + \frac{C_{r,\Sig,\Phi}}{\beta^2\sqrt{{n\dt}}} 
\]
where $C_{r,\Sig,\Phi} = L_\Phi^4 R\sigma_{\max}\lam_\Phi^{-2}\sqrt{\log(2p/\delta)}$, and $C_{r,\mu,\Sig}, C_G$ are the same constants as in Theorems~\ref{thm:v-vhat_stoch} and \ref{thm:galerkin err}.
\end{corollary}

Several remarks are in order. 
The total error consists of three components: a discretization error, a model approximation error, and a finite-sample error. Among these terms, only the model error is independent of the time step $\Delta t$. The discretization and sample errors exhibits opposite dependencies on the  date collection frequency $\dt$, i.e., 
%If the functional space spanned by the linear bases is sufficiently rich so that the model error is negligible, then for $n$ samples collected at time interval $\Delta t$,
\[
\|\hat V_n - V\|
=
O\!\left(
\frac{\Delta t}{\beta^2}
+
\frac{1}{\beta^2 \sqrt{n\Delta t}}
\right).
\]
This bound highlights a fundemental trade-off inherent in continuous-time PE. Decreasing $\Delta t$ reduces the discretization bias but increases the variance of the estimator, thereby requiring a larger sample size to maintain the same level of accuracy. As a result, excessively frequent sampling does not necessarily improve overall performance.

This phenomenon is not specific to the PhiBE framework but is intrinsic to continuous-time reinforcement learning. As illustrated in Figure~\ref{fig: data_er_stoch} of our numerical experiments, the model-free RL algorithm (LSTD) exhibit a similar trade-off.
Similar behavior was observed in \cite{zhang2023managing}, who studied Monte Carlo sampling algorithm for linear--quadratic (LQ) policy evaluation. Their theoretical analysis is restricted to the LQ setting, with nonlinear dynamics examined only through numerical experiments. In contrast, our analysis applies to general controlled diffusion processes and provides explicit non-asymptotic error bounds. Moreover, their approach evaluates the value function only at discrete states and does not involve function approximation, whereas our algorithm produces a global approximation of the value function via function approximation, leading to a substantially more involved analysis. Finally, our results explicitly characterize the dependence on the discount parameter $\beta$, which is not addressed in \cite{zhang2023managing}.

\color{black}
%Then, one approximates the integral over $s$ and $\hmu(s), \hs(s)$ by the trajectory data, \begin{equation}\label{galerkin_phibe}\begin{aligned}&\l[\sum_{j = 1}^J \Phi(s^j_0)\l(\b \Phi(s^j_0)^\top - \l(\sum_{k=1}^ia_k(s^j_{k\dt} - s^j_0)\r)\cdot \nb\Phi(s^j_0)^\top\r.\r.\\&\l.\l.-\frac12\l(\sum_{k=1}^ia_k(s^j_{k\dt} - s^j_0)(s^j_{k\dt} - s^j_0)^\top\r):\nb^2\Phi(s^j_0)^\top\r) \r]\th = \sum_{j = 1}^J \Phi(s^j_0)^\top r(s^j_0) \end{aligned}\end{equation} By solving the above linear system, one obtains the i-th order approximation to the value function using the trajectory data. 

%\begin{remark}{explanation on intuition and theoretical for future study}\end{remark}

\section{Numerical Experiments}\label{sec:numerics}
All MATLAB code used to generate the numerical results is available.\footnote{%
The repository is available at \url{https://github.com/yuhuazhumath/phibe-matlab}.}
\subsection{Deterministic dynamics}\label{sec:linear_galerkin}
We first consider deterministic dynamics,  where the state space  is defined as $\S = [-\pi,\pi]$. We consider two kinds of underlying dynamics, one is linear,
\begin{equation}\label{eq:deter_linear}
    \frac{d}{dt}s_t = \lam s_t,
\end{equation}
and the other is nonlinear,
\begin{equation}\label{eq:deter_nonlinear}
    \frac{d}{dt}s_t = \lam\sin^2(s_t).
\end{equation}
The reward is set to be $r(s) = \beta\cos(ks)^3-\lam s(-3k\cos^3(ks)\sin(ks))$ for the linear case and $r(s) = \beta\cos^3(ks)-\lam\sin^2(s)(-3k\cos^2(ks)\sin(ks))$ for the nonlinear case, where the value function can be exactly obtained, $V(s) = \cos^3(ks)$ in both cases. We use periodic bases $\{\phi_k(s_1)\}_{k=1}^{2M+1} = \frac{1}{\sqrt{\pi}}\{\frac{1}{\sqrt{2}}, \cos(ms_1), \sin(ms_1)\}_{m=1}^M$ with $M$ large enough so that the solution can be accurately represented by these finite bases {($M = 4$ for the low frequency value function and $M = 30$ for the high frequency value function)}. 

For the linear dynamics, the discrete-time transition dynamics are 
\[
p_\dt(s) = e^{\lam \dt}s.
\]
Hence, one can express the BE as
\begin{equation}\label{linear_BE}
    \tV(s) = r(s)\dt + e^{-\b \dt} \tV(e^{\lam \dt} s),
\end{equation}
and $i$-th order PhiBE as
\begin{equation}\label{linear_PhiBE}
    \b\hV_i(s) = r(s) + \frac1\dt \l[\sum_{k=1}^i\coef{i}_k(e^{\lam k\dt} s - s)\r] \cdot \nb \hV_i(s),
\end{equation}
respectively for $\coef{i}$ defined in \eqref{def of A b}. 
For the nonlinear dynamics, we approximate $p_{\dt}(s)$ and generate the trajectory data numerically,
\[
s_{t+\d} = s_t + \d \lam\sin^2(s_t)
\]
with $\d = 10^{-4}$ sufficiently small. 

The trajectory data are generated from $J$ different initial values $s_0\sim$Unif$[-\pi,\pi]$, and each trajectory has $m=4$ data points, $\{s_0, \cdots, s_{(m-1)\dt}\}$. Algorithm \ref{algo:galerkin_phibe_deter} is used to solve for the PhiBE, and LSTD is used to solve for BE. LSTD is similar to Algorithm \ref{algo:galerkin_phibe_deter} except that one uses $\t{A}$ derived from the BE \eqref{bellman},
\begin{equation}\label{galerkin_be}
    \begin{aligned}
    &\t{A} = \sum_{l=1}^{I}\sum_{j = 0}^{m-2}\Phi(s^l_{j\dt})\l[\Phi (s^l_{j\dt})- e^{-\beta\dt} \Phi(s^l_{(j+1)\dt})\r]^\top 
\end{aligned}
\end{equation}
instead of $\bar{A}_i$. 

In Figure \ref{fig:eq0}, the data are generated from the linear dynamics \eqref{eq:deter_linear} with $\lam = 0.05$ and collected at different $\dt$. {For the low-frequency value function (Figure \ref{fig:eq0}/(a), (c)), we use $40$ data points with $J = 10, m = 4$, 
whereas for the high-frequency value function (Figure~\ref{fig:eq0}/(b)), we use $400$ data points with 
$J = 100, m = 4$.
}  We compare the solution to the second-order PhiBE with the solution to BE (when the discrete-time transition dynamics are known), and the performance of LSTD with the proposed Algorithm \ref{algo:galerkin_phibe_deter} (when only trajectory data are available) with different data collection interval $\dt$, discount coefficient $\b$ and oscillation of reward $k$. Note that the exact solution to BE is computed as $\tV(s) = \sum_{i=0}^Ir(e^{\lam \dt i}s)$ with {$I =500/\dt$} large enough, and the exact solution to PhiBE is calculated by applying the Galerkin method to \eqref{linear_PhiBE}. 

In Figure \ref{fig:eq1}, the data are generated from the nonlinear dynamics \eqref{eq:deter_nonlinear} and collected at different $\dt$. {For the low-frequency value function (Figure \ref{fig:eq1}/(a), (c)), we use $80$ data points with $J = 20, m = 4$, 
whereas for the high-frequency value function (Figure~\ref{fig:eq1}/(b), we use $400$ data points with 
$J = 100, m = 4$.
} We compare the solutions to the first-order and second-order PhiBE with the solution to the BE (when the discrete-time transition dynamics are known), and the performance of LSTD with the proposed Algorithm \ref{algo:galerkin_phibe_deter} (when only trajectory data are available)  with different $\dt, \b, k, \lam$. 

In Figure \ref{fig: data_er}, the distances of the solution from PhiBE, BE to the true value function are plotted as $\dt\to0$; the distances of the approximated solution by Algorithm \ref{algo:galerkin_phibe_deter} and LSTD to the true value function are plotted as the amount of data increases. {We set $J = [10, 10^2, 10^3, 10^3]$ and $m=4$ for the data size.}  Here, the distance is measured using the $L^2$ norm
\begin{equation}\label{def of l2 error}
    D(V,\hV) = \sqrt{\int_{-\pi}^\pi (V(s) - \hV(s))^2 ds}.
\end{equation}

In Figures \ref{fig:eq0} and \ref{fig:eq1}, when the discrete-time transition dynamics are known, PhiBE solution is much closer to the true value function compared to the BE solution in all the experiments. Especially, the second-order PhiBE solution is almost identical to the exact value function. Additionally, when only trajectory data is available, one can approximate the solutions to PhiBE very well with only $40$ or $400$ data points. Particularly, when $\dt = 5$ is large, the solution to PhiBE still approximates the true solution very well, which indicates that one can collect data sparsely based on PhiBE. Moreover, the solution to PhiBE is not sensitive to the oscillation of the reward function, which implies that one has more flexibility in designing the reward function in the RL problem.  Besides, unlike BE, the error increases when $\b$ is too small or too large, while the error for PhiBE decays as $\b$ increases. Furthermore, it's noteworthy that in Figure \ref{fig:eq1}/(b) and (c), for relatively large changes in the dynamics indicated by $\ll \nb \mu \rl \leq \lam = 5 $ and $2$, respectively, PhiBE still provides a good approximation.

In Figure \ref{fig: data_er}/(a) and (b), one can observe that the solution for BE approximates the true solution in the first order, while the solution for $i$-th order PhiBE approximates the true solution in $i$-th order. 
In Figure \ref{fig: data_er}/(c) and (d), one can see that as the amount of data increases, the error from the LSTD algorithm stops decreasing when it reaches $10^{-1}$. This is because the error between BE and the true value function $\ll \tV - V \rl = O(\dt)$ dominates the data error. On the other hand, for higher-order PhiBE, as the amount of data increases, the performance of the algorithm improves, and the error can achieve $O(\dt^i)$.

\begin{figure}
\centering
     \begin{subfigure}[b]{0.32\textwidth}
         \centering
         \includegraphics[width=\textwidth]{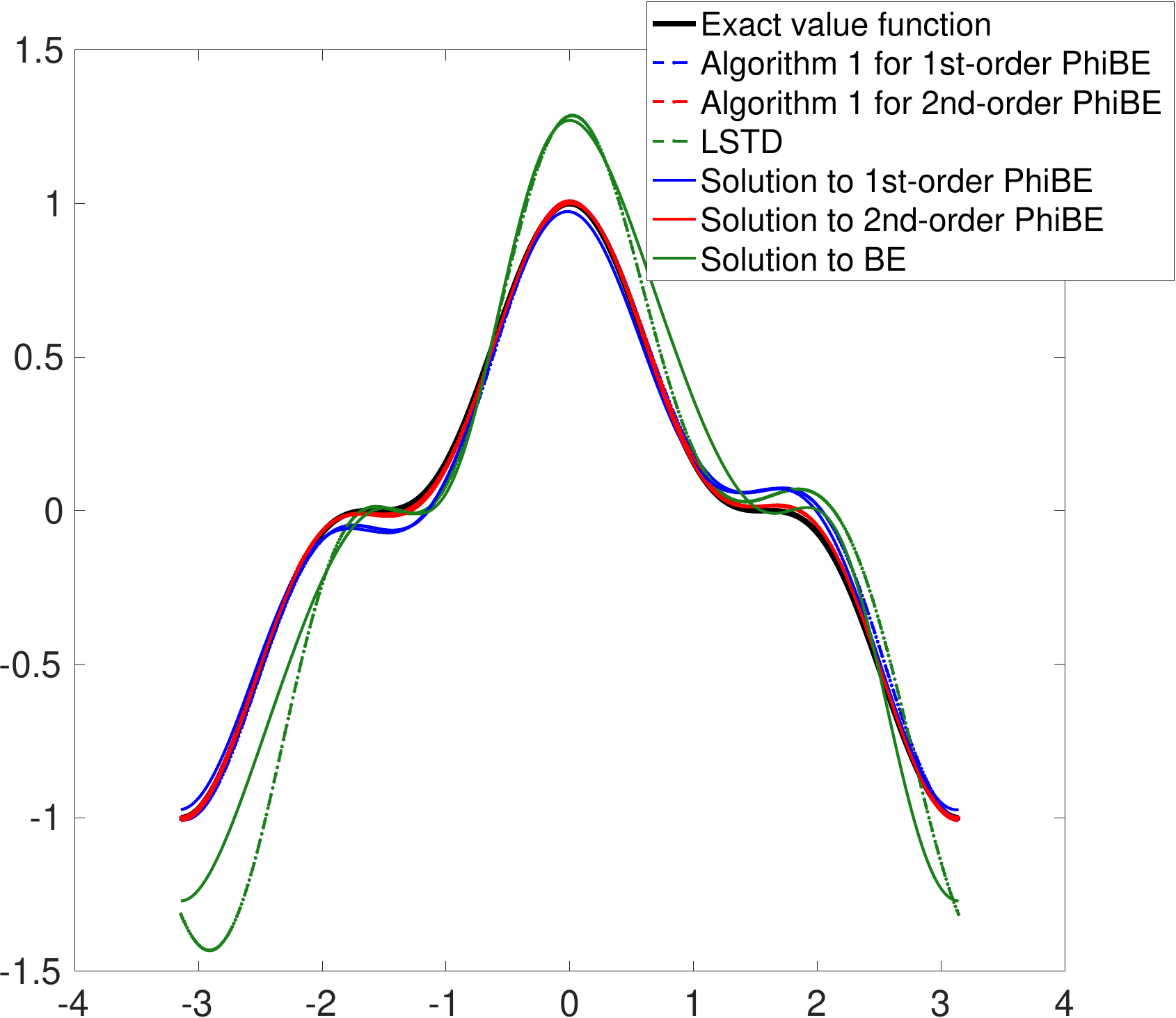}
         \caption{Nonlinear dynamics \eqref{eq:deter_nonlinear} with $\dt = 5, \b = 0.1, k = 1,\lam = 0.1$}
     \end{subfigure}
     \hfill
     \begin{subfigure}[b]{0.32\textwidth}
         \centering
         \includegraphics[width=\textwidth]{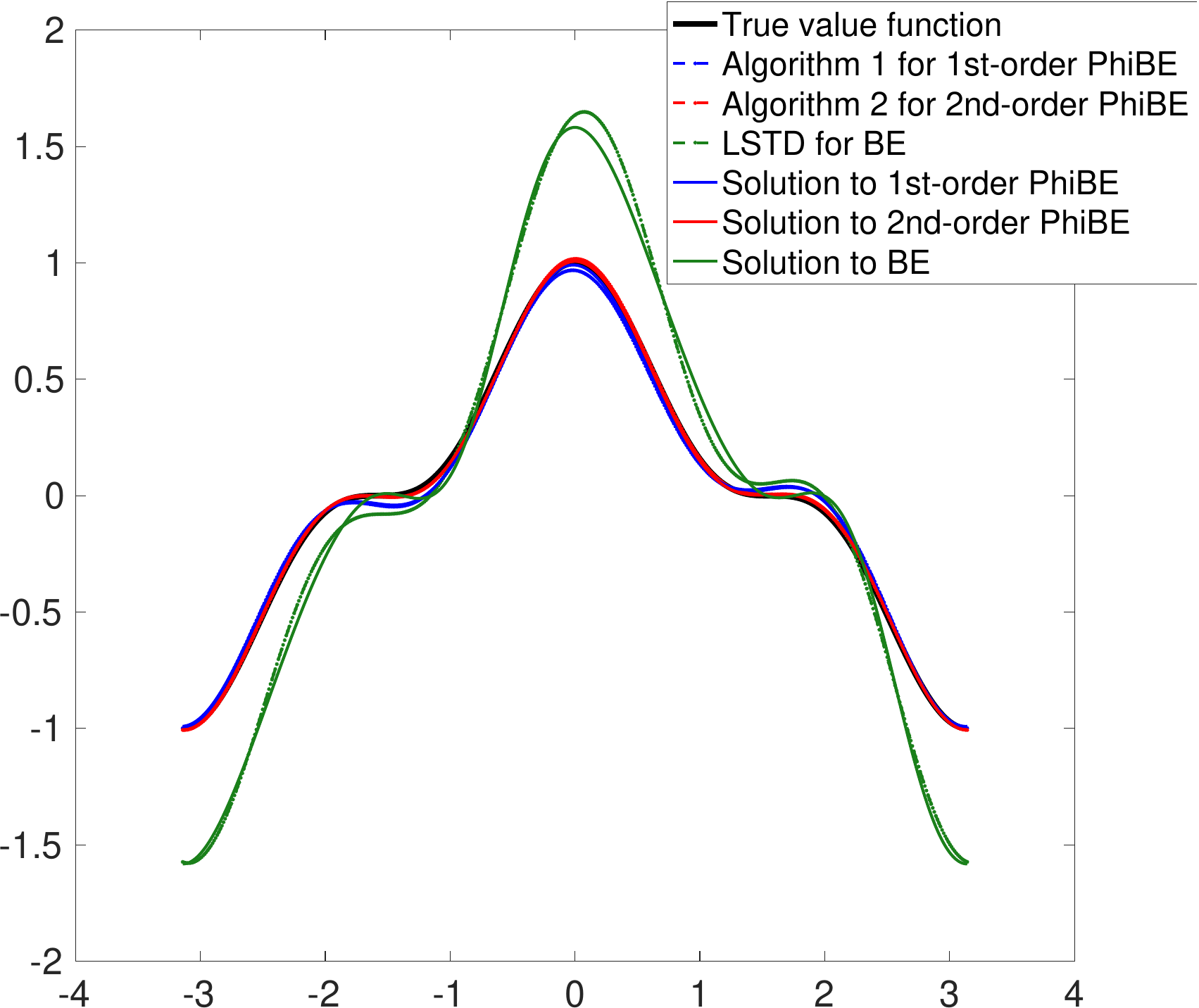}
         \caption{Nonlinear dynamics \eqref{eq:deter_nonlinear} with $\dt = 0.1, \b = 10, {k = 1},\lam = 5$}
     \end{subfigure}
     \hfill
     \begin{subfigure}[b]{0.32\textwidth}
         \centering
         \includegraphics[width=\textwidth]{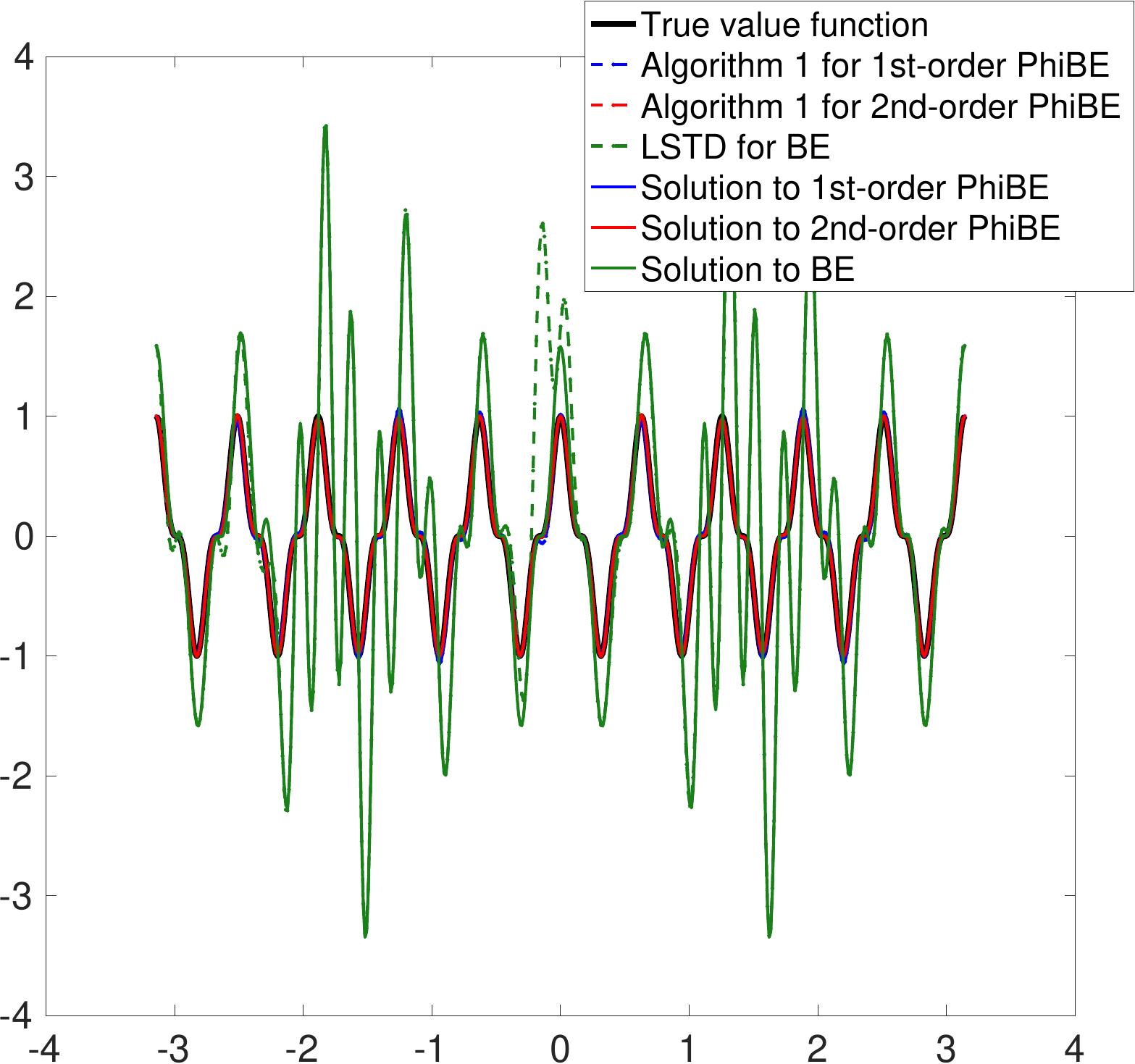}
         \caption{Nonlinear dynamics \eqref{eq:deter_nonlinear} with $\dt = 0.1, \b = 10, k = 10,\lam = 2$}
     \end{subfigure}
     \caption{The PhiBE solution and the BE solution, when the discrete-time transition dynamics are given, are plotted in solid lines. The approximated PhiBE solution based on Algorithm \ref{algo:galerkin_phibe_deter} and the approximated BE solution based on LSTD, when discrete-time data are given, are plotted in dash lines. Both algorithms utilize the same data points.}
     \label{fig:eq1} 
\end{figure}

\begin{figure}
\centering
     \begin{subfigure}[b]{0.24\textwidth}
         \centering
         \includegraphics[width=\textwidth]{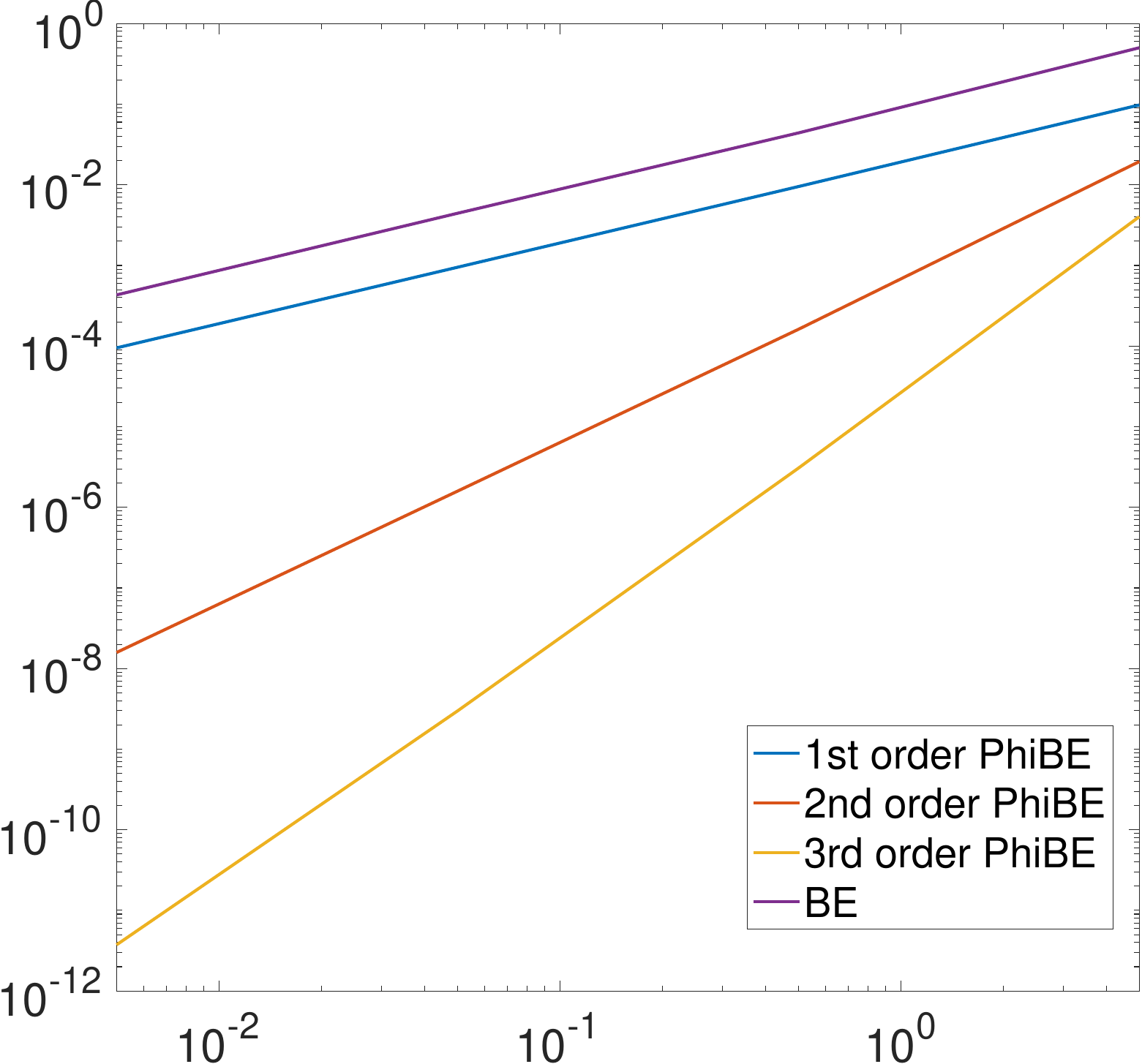}
         \caption{Linear dynamics \eqref{eq:deter_nonlinear}.}
     \end{subfigure}
     \hfill
     \begin{subfigure}[b]{0.24\textwidth}
         \centering
         \includegraphics[width=\textwidth]{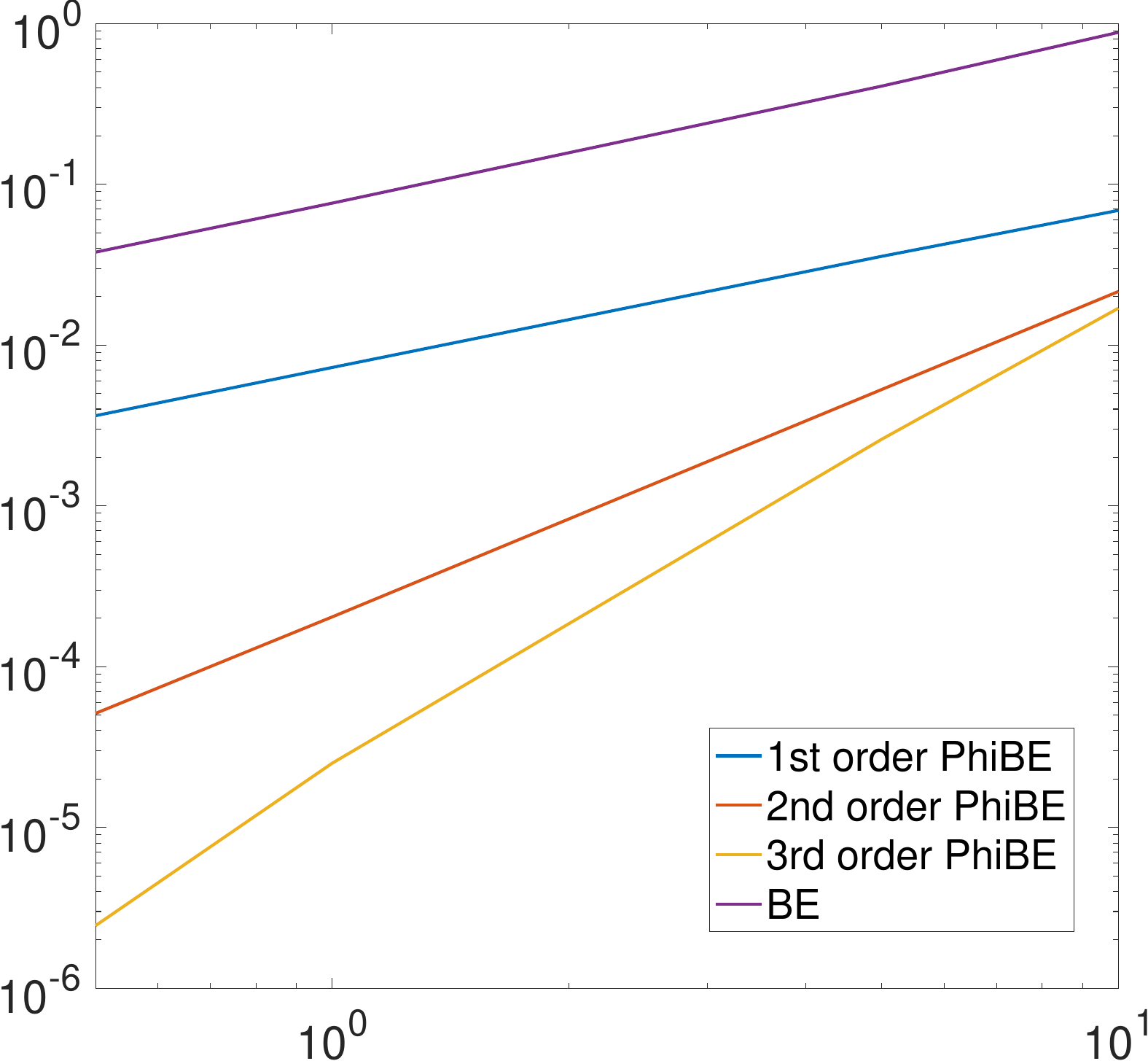}
         \caption{Nonlinear dynamics \eqref{eq:deter_nonlinear}}
     \end{subfigure}
     \hfill
     \begin{subfigure}[b]{0.24\textwidth}
         \centering
         \includegraphics[width=\textwidth]{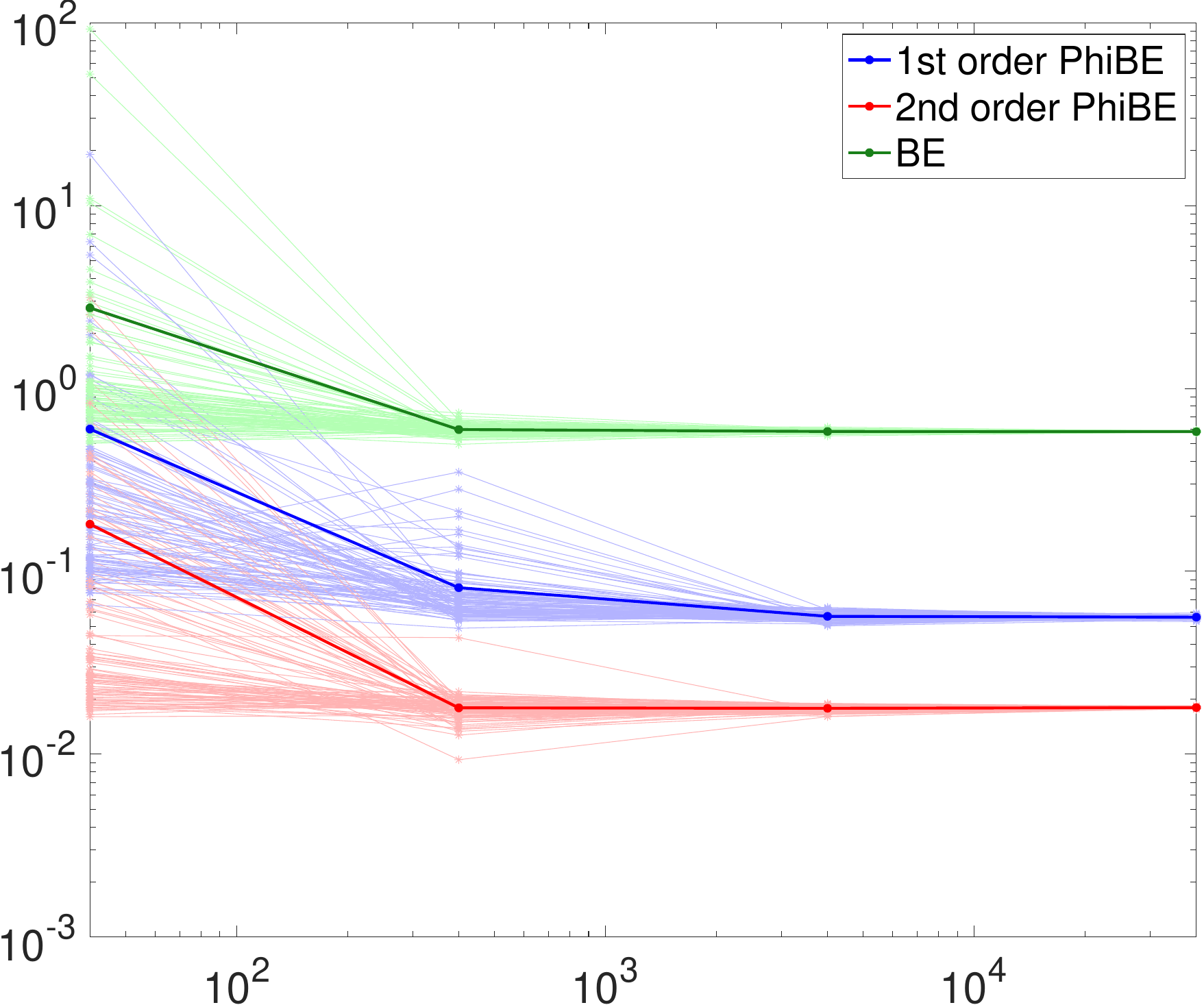}
         \caption{Linear dynamics \eqref{eq:deter_nonlinear}.}
     \end{subfigure}
     \hfill
     \begin{subfigure}[b]{0.24\textwidth}
         \centering
         \includegraphics[width=\textwidth]{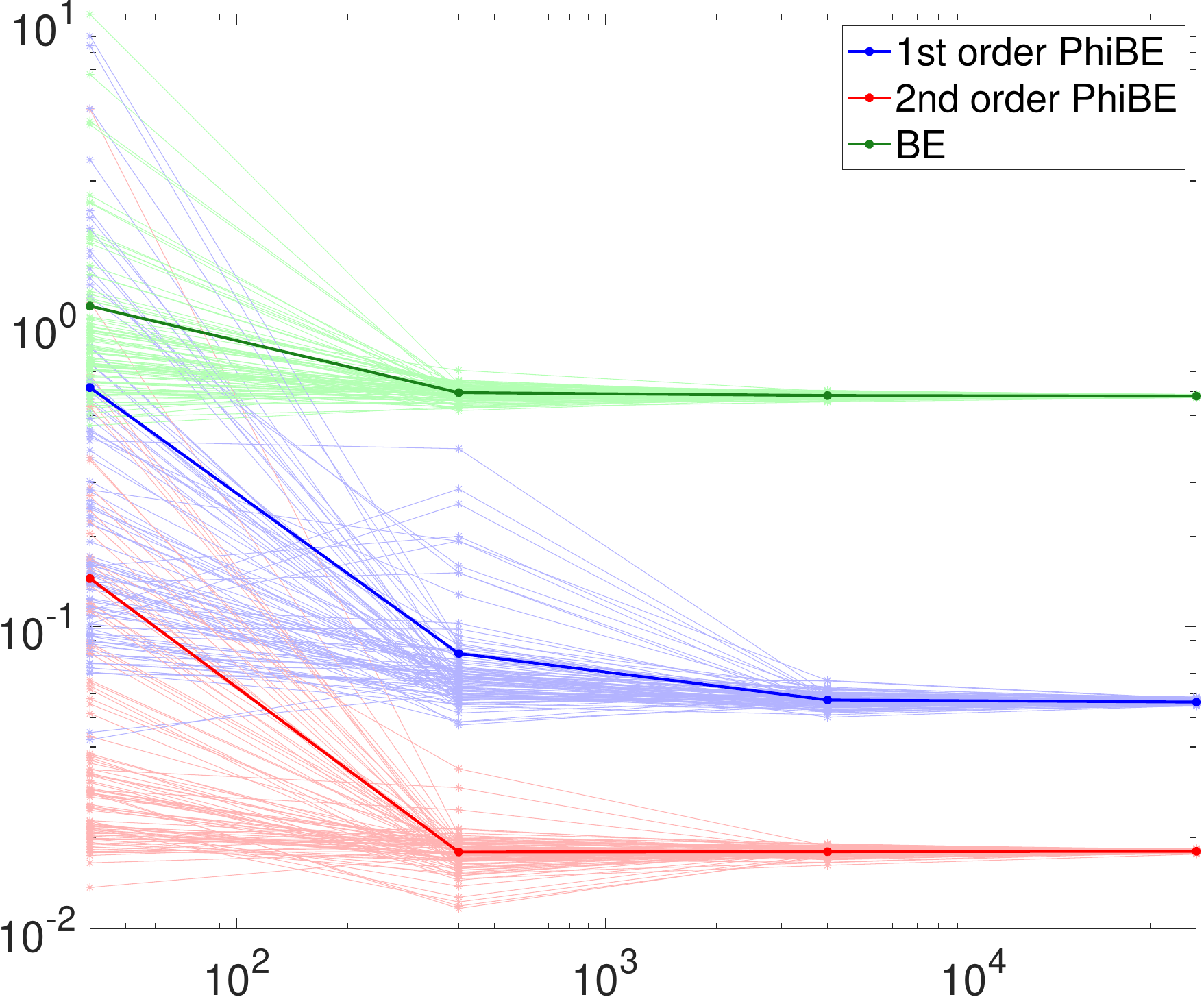}
         \caption{Nonlinear dynamics \eqref{eq:deter_nonlinear}.}
     \end{subfigure}
     \caption{The $L^2$ error \eqref{def of l2 error} of the PhiBE solutions and the BE solutions with decreasing $\dt$ are plotted in the left two figures. The $L^2$ error \eqref{def of l2 error} of the approximated PhiBE solutions and the approximated BE solutions with increasing amount of data collected every $\dt = 5$ unit of time are plotted in the right two figures. {The solid lines are the average over $100$ simulations for all different number of data.} We set $\lam = 0.05, \b = 0.1, k = 1$ in both linear and nonlinear cases. }
     \label{fig: data_er} 
\end{figure}

\newpage
\subsection{Stochastic dynamics}
We consider the Ornstein–Uhlenbeck process,
\begin{equation}\label{def of ou}
    ds(t) = \lam sdt + \s dB_t,
\end{equation}
with $\lam = 0.05, \s = 1$. Here the reward is set to be $r(s) = \b \cos^3(ks)-\lam s (-3k\cos^2(ks)\sin(ks)) $ $- \frac12\s^2(6k^2\cos(s)\sin^2(ks) -3k^2\cos^3(ks))$, where the value function can be exactly obtained, 
$V(s) = \cos^3(ks)$. For OU process, since the conditional density function for $s_t$ given $s_0 = s$ follows the normal distribution with expectation $se^{\lam t}$, variance $\frac{\s^2}{2\lam}(e^{2\lam t}-1)$. Both PhiBE and BE have explicit forms. One can express PhiBE as,
\begin{equation}\label{phibe_stoch}
    \begin{aligned}
        \b \hV_i(s) =& r(s) + \frac1\dt\sum_{k=1}^i\coef{i}_k(e^{\lam k\dt} - 1)s \nb \hV(s) \\
        &+ \frac1{2\dt}\sum_{k=1}^i\coef{i}_k\l[\frac{\s^2}{2\lam}(e^{2\lam k\dt}-1) + (e^{\lam k\dt} - 1)^2s^2\r]\D \hV(s);
    \end{aligned}
\end{equation}
and  BE as,
\begin{equation}\label{be_stoch}
\begin{aligned}
    \tV(s) = & r(s)\dt + e^{-\beta\dt} \E\l[\tV(s_{t+1})|s_t = s\r]\\
    & = r(s)\dt + e^{-\beta\dt} \int_\S \tV(s') \rho_{\dt}(s',s) ds',
\end{aligned}
\end{equation}
where $$\rho_{\dt}(s',s) = \frac{1}{\sqrt{2\pi} \h{\s}}\exp\l(-\frac1{2\h{\s}^2}(s' - se^{\lam\dt})^2\r), \qd \text{with  } \h{\s} = \frac{\s^2}{2\lam}(e^{2\lam \dt} -1).$$

In Figure \ref{fig:eq2}, we compare the exact solution and approximated solution to PhiBE and BE, respectively, for different $\dt, \b, k$. {For the approximated solution, we set the data size $m = 4$, and $J = 10^4$ for Figure~\ref{fig:eq2}/(a),  $J = 10^5$ for Figure~\ref{fig:eq2} /(b), $J = 100$ for Figure~\ref{fig:eq2} /(c).} In Figure \ref{fig: data_er_stoch}/(a), the decay of the error as $\dt \to 0$ for the exact solutions to PhiBE and BE are plotted. In Figure  \ref{fig: data_er_stoch} /(b), the decay of the approximated solution to PhiBE and BE based on Algorithm \ref{algo:galerkin_phibe_stoch} and LSTD are plotted with an increasing amount of data. {We set $J = [10^3, 10^4, 10^5, 10^6, 10^7]$ and $m=4$ for the data size.} In Figure  \ref{fig: data_er_stoch} /(c), the error of the approximated solution to PhiBE and BE based on Algorithm \ref{algo:galerkin_phibe_stoch} and LSTD are plotted with a decreasing $\dt$. {We set $J = 10^6$ and $m=4$ for the data size.}

We observe similar performance in the stochastic dynamics as in the deterministic dynamics, as shown in Figures \ref{fig:eq2} and \ref{fig: data_er_stoch}. In Figure \ref{fig: data_er_stoch}, the variance of the higher order PhiBE is larger than that of the first-order PhiBE because it involves more future steps. However, note that the error is plotted on a logarithmic scale. Therefore, when the error is smaller, although the variance appears to have the same width on the plot, it is actually much smaller. Particularly, when the amount of the data exceeds $10^6$, the variance is smaller than $10^{-1}$. 
{Figure \ref{fig: data_er_stoch}/(c) illustrates a non-monotone dependence of the error on the time discretization $\Delta t$ when the total data budget is fixed. For the data-driven RL algorithm and the first-order PhiBE method, the error initially decreases as $\Delta t$ is reduced from $1$ to $0.1$, reflecting a regime in which discretization error dominates. As $\Delta t$ is further reduced from $0.1$ to $0.01$, the error increases, indicating that finite-sample error becomes dominant. The second-order PhiBE method, which exhibits higher variance, enters the variance-dominated regime earlier: its error already increases as $\Delta t$ decreases from $1$ to $0.1$.}

\begin{figure}
\centering
     \begin{subfigure}[b]{0.32\textwidth}
         \centering
         \includegraphics[width=\textwidth]{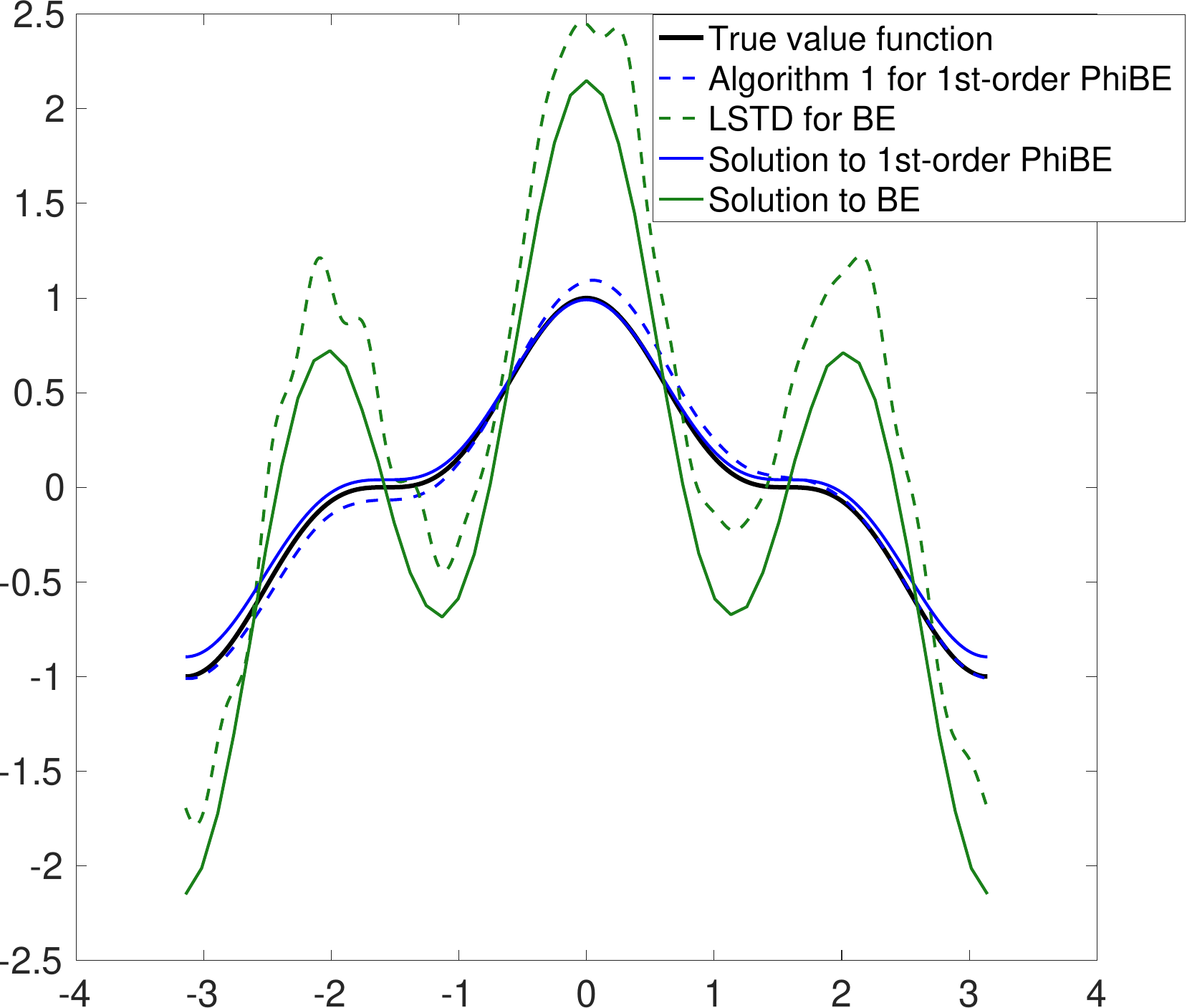}
         \caption{ Stochastic dynamics \eqref{def of ou} with $\dt = 1,  \beta = 0.1, k = 1$.\\\quad\quad}
     \end{subfigure}
     \hfill
     \begin{subfigure}[b]{0.32\textwidth}
         \centering
         \includegraphics[width=\textwidth]{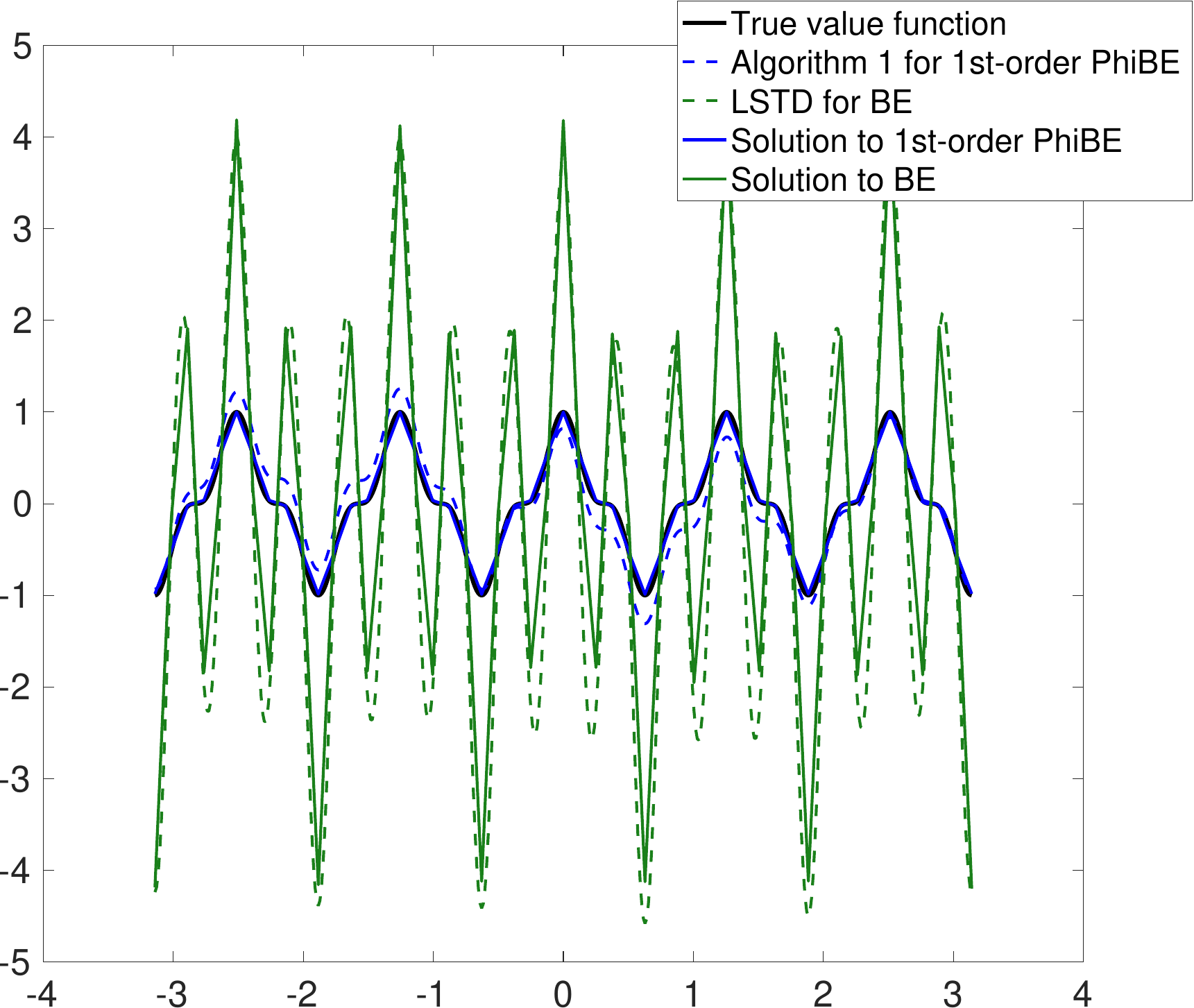}
         \caption{Stochastic dynamics \eqref{def of ou} with $\dt = 0.1,  \beta = 0.1, k = 5$.}
     \end{subfigure}
     \begin{subfigure}[b]{0.32\textwidth}
         \centering
         \includegraphics[width=\textwidth]{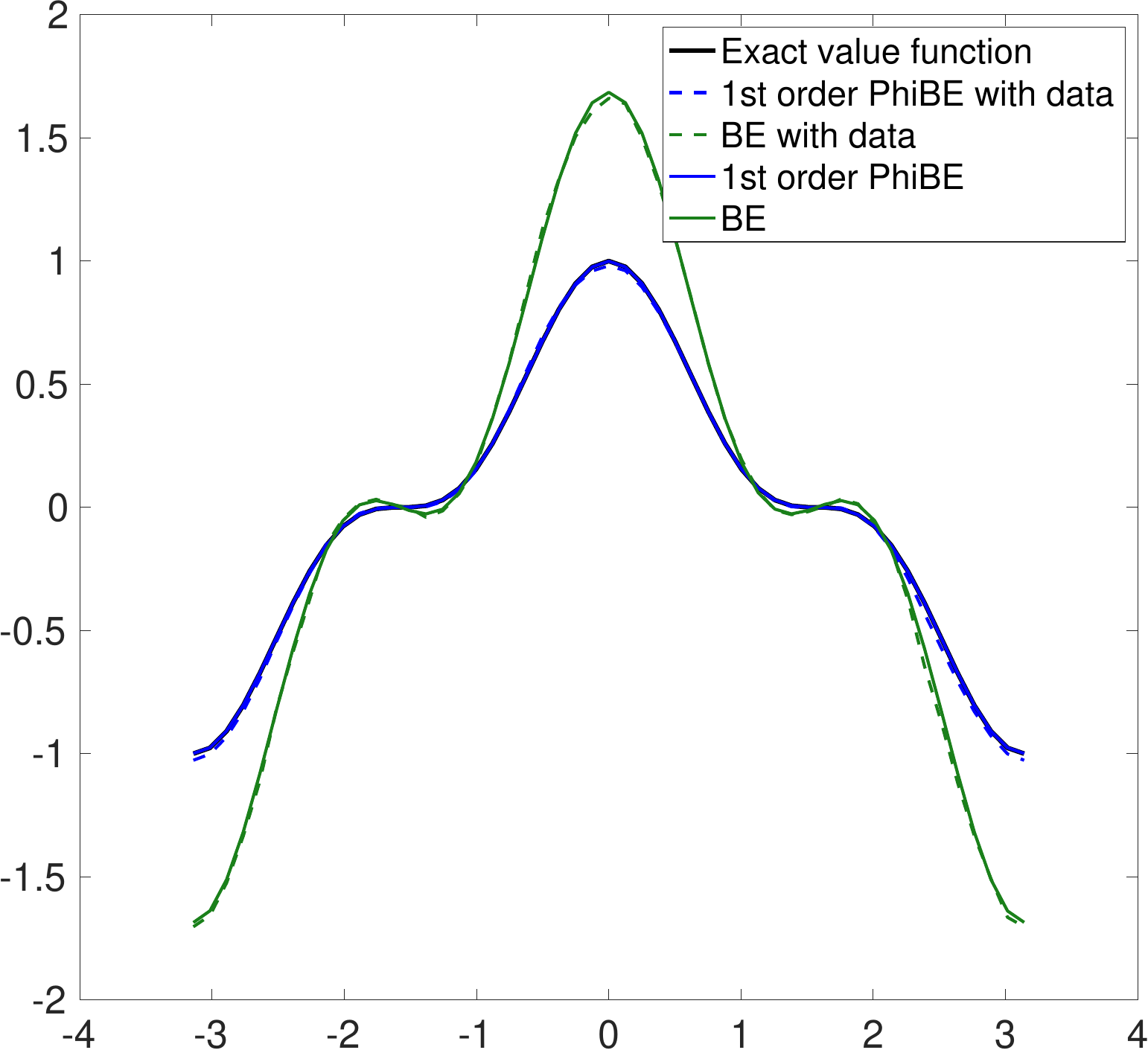}
         \caption{Stochastic dynamics \eqref{def of ou} with $\dt = 0.1,  \beta = 10, k = 1$.}
     \end{subfigure}
     \caption{
     The PhiBE solution and the BE solution, when the discrete-time transition dynamics are given, are plotted in solid lines. The approximated PhiBE solution based on Algorithm \ref{algo:galerkin_phibe_stoch} and the approximated BE solution based on LSTD, when discrete-time data are given, are plotted in dash lines. Both algorithms utilize the same data points.
     }
     \label{fig:eq2} 
\end{figure}
\begin{figure}
\centering
\hfill
     \begin{subfigure}[b]{0.3\textwidth}
         \centering
         \includegraphics[width=\textwidth]{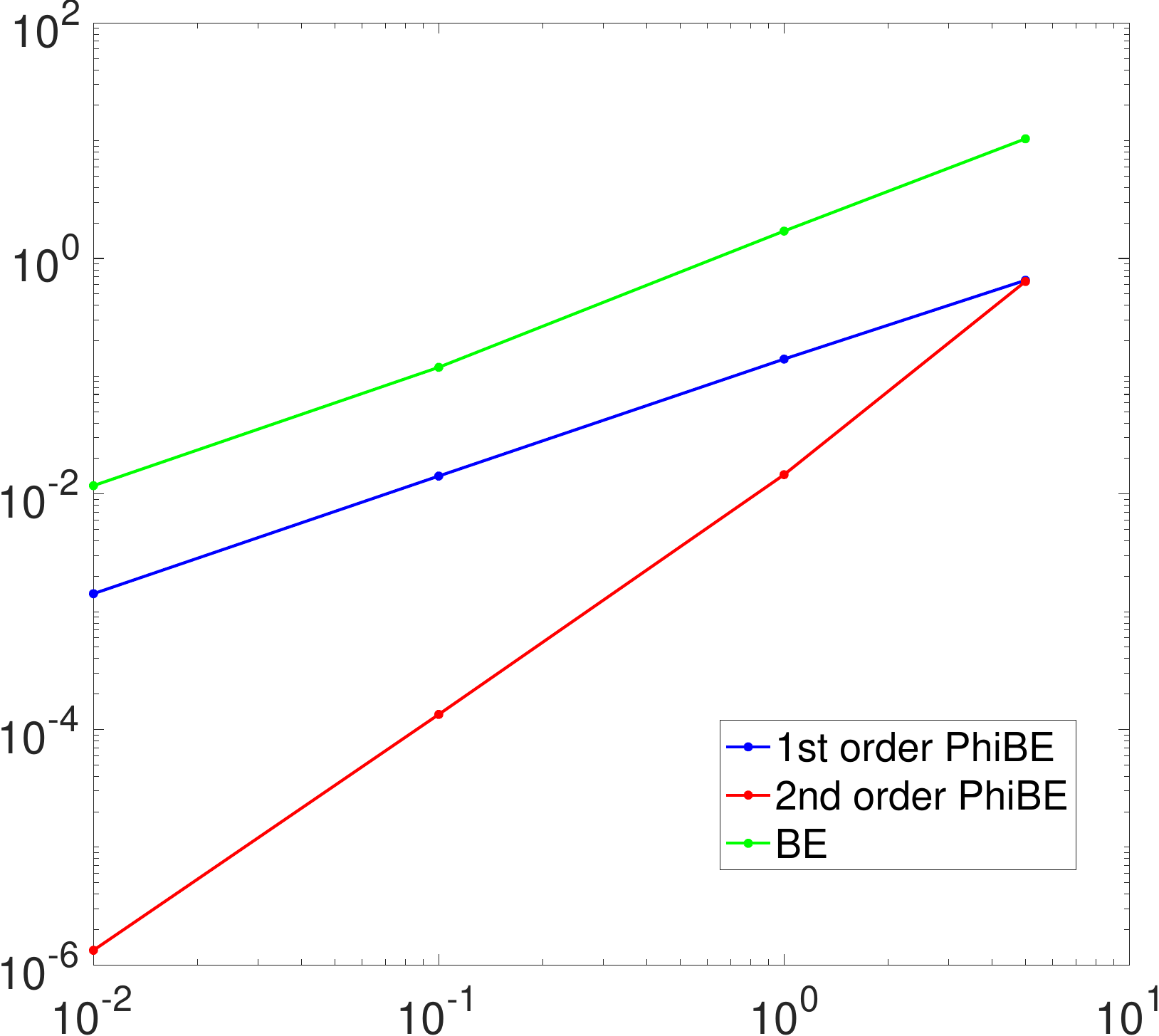}
         \caption{Discretization error w.r.t. data collection frequency $\dt$.}
     \end{subfigure}
     \hfill
     \begin{subfigure}[b]{0.3\textwidth}
         \centering
         \includegraphics[width=\textwidth]{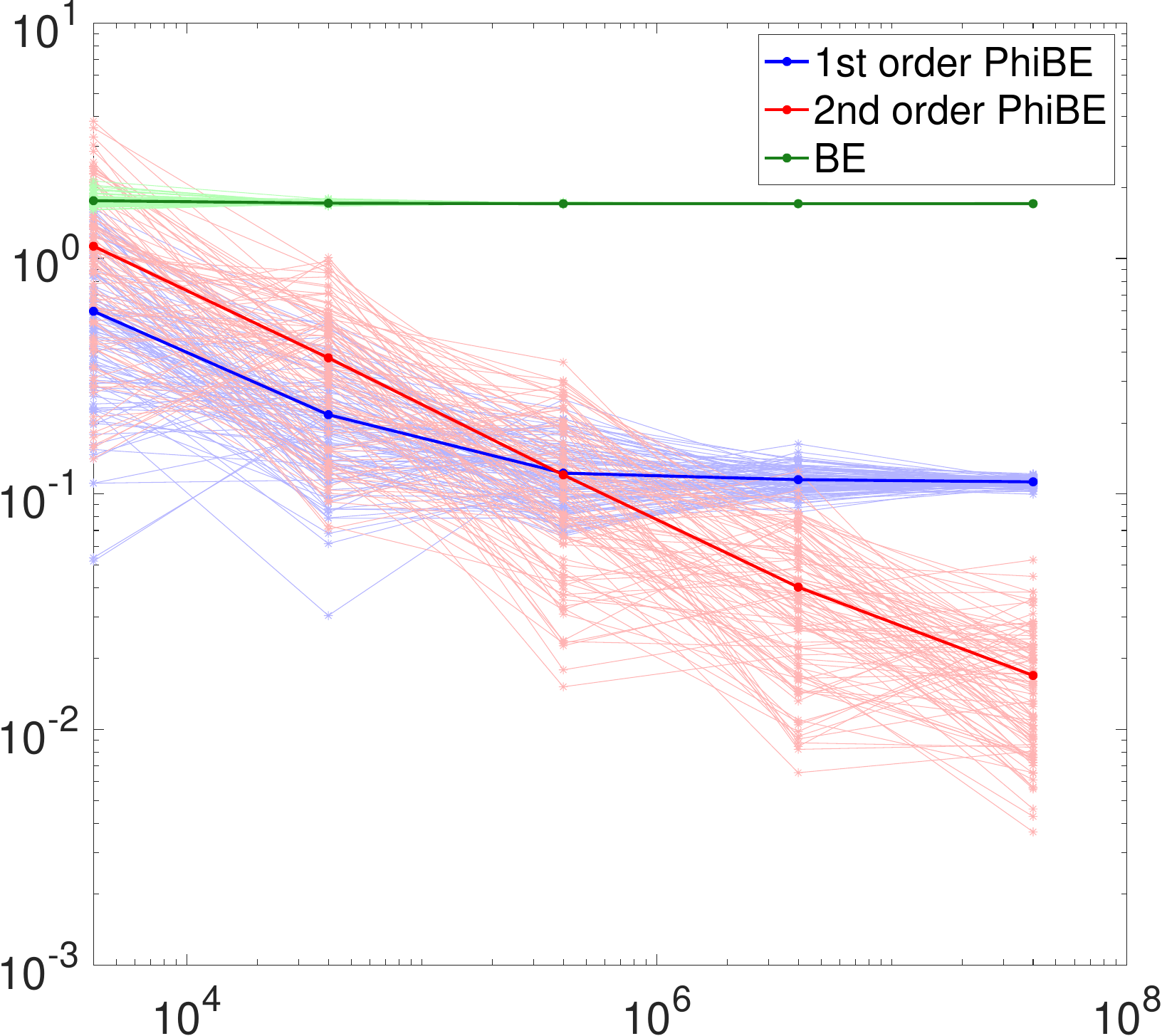}
         \caption{Approximation error w.r.t the amount of data.}
     \end{subfigure}
     \hfill
     \begin{subfigure}[b]{0.36\textwidth}
         \centering
         \includegraphics[width=\textwidth]{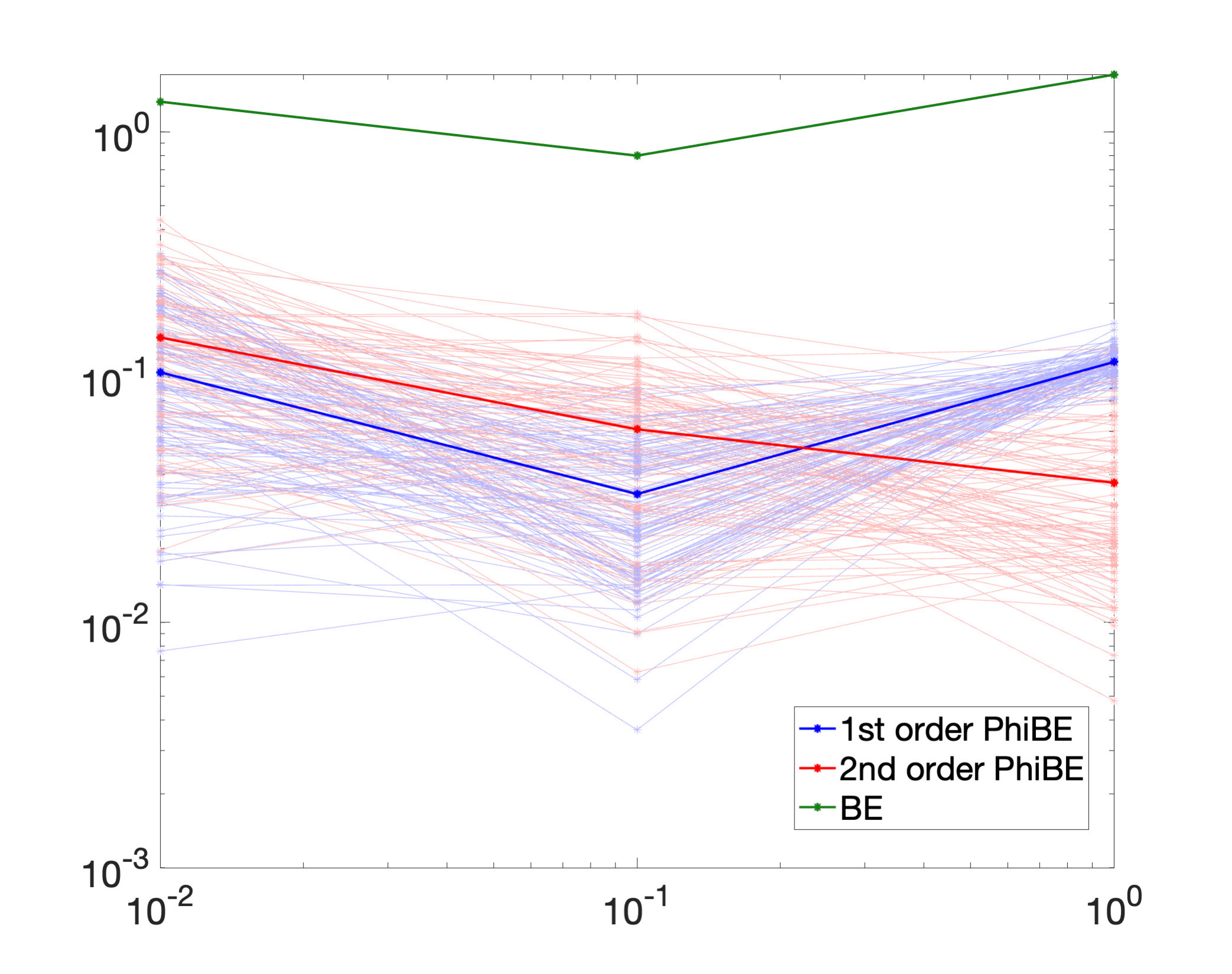}
         \caption{Approximation error w.r.t data collection frequency$\dt$.}
     \end{subfigure}
     \hfill
     \caption{The $L^2$ error \eqref{def of l2 error} of the PhiBE and BE solutions as a function of the time discretization $\dt$ is shown in (a).
Plot (b) reports the $L^2$ error of the data-driven approximations as the amount of data increases, with samples collected every $\dt = 1$ unit of time.
Plot (c) shows the $L^2$ error of the data-driven approximations with a fixed budget of $4\times10^6$ data points as $\dt$ decreases.
{Solid lines denote averages over $100$ independent simulations.}
In all plots, we set $\beta = 0.1$ and $k = 1$.}
     \label{fig: data_er_stoch}
\end{figure}

\subsection{Stabilization problem}
In the previous section, we compared the PhiBE framework with the standard RL formulation and validated our theoretical results using a synthetic example. We now examine their performance on a real control problem.

\paragraph{Problem setup.}
We consider the continuous-time PE problem
\begin{equation}  \label{eq:sde-closed-loop}
\begin{aligned}
    &V(s) = 
    \E\l[\int_0^\infty e^{-\beta t} (q s_t^2+ru(s_t)^2) | s_0 = s\r]\\
    &dS_t=\l(-\kappa s_t^3+(\alpha s_t -b u(s_t)\r)\,dt + \sigma\,dW_t,
\end{aligned}
\end{equation}
with parameters $\kappa>0$, $\alpha\in\mathbb{R}$, $b>0$, and $\sigma>0$.
We focus on the linear feedback policy
\[
u(s)=K s.
\]
Under this policy, the true value function $V$ solves the elliptic equation
\begin{equation} \label{eq:poisson}
\beta V(s) - \mathcal L V(s) = r(s), \qquad
\mathcal L f(s)
= \Big(-\kappa s^{3} + (\alpha - bK)s\Big) f'(s)
+ \frac{\sigma^{2}}{2} f''(s),
\end{equation}
where $r(s)=q s^{2} + r K^{2} s^{2}$.

This model represents a standard stabilization problem, where the quadratic reward balances state deviation and control effort, and the linear policy $u(s)=-Ks$ acts as a proportional controller. When $\kappa=0$, the dynamics reduce to a linear system, corresponding to the classical continuous-time LQR formulation widely used in robotics, aerospace, and process control. In this regime, the value function admits a closed-form solution, providing a baseline against which the approximation accuracy of RL and PhiBE can be directly evaluated. When $\kappa\neq 0$, the cubic drift introduces strong nonlinear restoring forces, modeling phenomena such as soft-spring mechanical systems or the stabilization of an inverted pendulum near its upright position.

\subsubsection{Linear dynamics}
In the linear-dynamics setting, the true value function admits a closed-form solution,
\begin{equation}\label{lqr-true}
    V(s) = a_1s^2 + a_2, \quad a_1 = \frac{R}{\beta - 2\lam}, \quad a_2 = \frac{\sigma^2}{\beta}a_1,
\end{equation}
where $R = q + rK^2$ and $\lambda = \alpha - bK$.
Moreover, for this linear system, the Bellman equation (BE) and PhiBE both admit exact analytical solutions. For the discrete-time BE, we have
\[
\begin{aligned}
    &\text{BE}: \t{V}(s) = Rs^2 \dt + \gamma \int V(s') \rho(s'|s) ds', \quad R = q+rK^2, \quad \gamma = e^{-\beta \dt},
\end{aligned}
\]
where 
\[
\rho(s'|s) \sim \mN (e^{\lam\dt}s, \hat{\sigma}^2\dt), \quad \quad \hat{\sigma}^2 = \frac{\sigma^2}{2\lam\dt}(e^{2\lam\dt} - 1).
\]
This yields the quadratic solution for BE,
\begin{equation} \label{be-anlytic}
\t{V}(s) = a_1^{R}s^2 + a_0^R, \quad a_1^{R} = \frac{R\dt} {1-\gamma e^{2\lam\dt}}, \quad a_0^{R} = \frac{\gamma \dt}{1-\gamma}\h{\sigma}^2 a_1^{R}.
\end{equation}
Similarly, the PhiBE equation can be evaluated exactly:
\[
\begin{aligned}
\text{PhiBE}: \h{V}(s) =& Rs^2 +  \h{V}'(s)\int \frac{s'-s}{\dt} \rho(s'|s) ds' + \frac12\h{V}''(s)\int \frac{(s'-s)^2}{\dt} \rho(s'|s) ds' \\
=&Rs^2 +  \hat{\lambda}s\h{V}'(s) + \frac{\h{\sigma}^2 + \hat{\lam}^2s^2\dt}2\h{V}''(s), \quad \h{\lam} = \frac{e^{\lam\dt} - 1}{\dt}.
\end{aligned}
\]
This yields another quadratic solution,
\begin{equation}\label{phibe-anlytic}
    \h{V}(s) = a_1^{P}s^2 + a_0^P, \quad a_1^{P} =\frac{R}{\beta - 2\h{\lam} - \eta}, \quad a_0^{P} = \frac{\h{\sigma}^2}{\beta}a_1^P.
\end{equation}
where $\eta = \frac{1}{\dt}(e^{2\lam\dt} - 2e^{\lambda\dt} + 1)$.

Figure~\ref{fig:lqr-anlytic} compares the analytical BE and PhiBE solutions, $\widetilde{V}$ and $\widehat{V}$, with the true value function $V(s)$ across four test cases. Throughout all experiments we set $q = 1$, $r = 0.1$, and $K = 2$. The four configurations are:
\begin{equation}\label{lqr-diff case}
\begin{aligned}
    &\text{case 1\ \   Baseline:} \ \alpha = b = \frac14, \sigma = 0.5, \beta = 1, \dt = 0.1.\\
    &\text{case 2\ \  More frequent observations (smaller $\dt$):}\  \alpha = b = \frac14, \sigma = 0.5, \beta = 1, \dt = 0.01.\\
    &\text{case 3\ \   Quicker dynamics:}\  \alpha = b = 1, \sigma = 1, \beta = 1, \dt = 0.1.\\
    &\text{case 4\ \  Less discounted (smaller $\beta$):}\  \alpha = b = \frac14, \sigma = 0.5, \beta = 0.1, \dt = 0.1.
\end{aligned}
\end{equation}

From Figure~\ref{fig:lqr-anlytic}, we observe that reducing $\Delta t$ by one order decreases the approximation error of both formulations by approximately one order, consistent with the theoretical convergence rates. Moreover, as the system dynamics become faster or the discount factor $\beta$ becomes smaller, the advantage of PhiBE over BE diminishes, although PhiBE still performs better overall. This observation aligns with our theoretical analysis: (1) PhiBE achieves greater improvement when the dynamics evolve more slowly, and (2) the approximation error of PhiBE scales as $O(1/\beta^2)$, whereas that of BE scales as $O(1/\beta)$ when $\beta$ is small.

We next evaluate the data-driven algorithms. To solve the PhiBE formulation, we use the first-order method in Algorithm~\ref{algo:galerkin_phibe_stoch}; for the BE formulation, we apply LSTD as in~\eqref{galerkin_be}. We generate $\{s_i\}_{i=1}^n$ from $[-1,1]$ on the uniform mesh, and $s_i'\sim \mN(e^{\lam\dt}s_i, \hat{\sigma}^2\dt)$ for all $1\leq i\leq n$.
consistent with the linear system dynamics. The mean and variance of the approximation error, computed over $100$ Monte Carlo repetitions for all four cases, are shown in Figure~\ref{fig:lqr-data-driven}.

Figure~\ref{fig:lqr-data-driven} demonstrates that the data-driven estimators concentrate around the corresponding analytic solutions in Figure~\ref{fig:lqr-anlytic} when the sample size is sufficiently large. However, the variability differs across scenarios. To examine this more closely, Table~\ref{table:lqr} reports the mean and variance of the approximation error. Two observations stand out from Table~\ref{table:lqr}. First, across all test cases, PhiBE consistently attains smaller mean error and smaller variance compared with BE given the same amount of data. Second, for a fixed sample size, the variance increases for both methods as $\Delta t$ decreases, $\beta$ decreases, or the dynamics become faster. %A precise characterization of how the variance and the associated sample complexity depend on these quantities is an interesting direction for future work.

\begin{table}[h]
\centering
\begin{tabular}{|c|c|c|c|c|}
\hline
& mean (PhiBE) & mean (BE) & variance (PhiBE) & variance (BE) \\ \hline
Baseline ($n = 10^6$)& $4.76\times 10^{-3}$& $4.54 \times 10^{-2}$ &$2.7\times 10^{-6}$ & $1.02\times 10^{-5}$\\ \hline 
Smaller $\dt$ ($n = 10^6$) & $9.64\times 10^{-3}$& $1.04 \times 10^{-2}$ &$3.84\times 10^{-5}$ & $4.3\times 10^{-5}$\\ \hline 
Smaller $\dt$ ($n = 10^7$) & $2.91\times 10^{-3}$& $4.7 \times 10^{-3}$ &$3.6\times 10^{-6}$ & $8.1\times 10^{-6}$ \\ \hline 
Quick dynamics ($n = 10^6$)  & $3.45\times 10^{-2}$& $5.5 \times 10^{-2}$ &$3.87\times 10^{-4}$ & $8.07\times 10^{-4}$ \\ \hline 
Smaller $\beta$ ($n = 10^6$) & $7.48\times 10^{-2}$& $7.81 \times 10^{-2}$ &$2.11\times 10^{-3}$ & $2.93\times 10^{-3}$ \\ \hline
Smaller $\beta$ ($n = 10^7$) & $3.41\times 10^{-2}$& $4.97 \times 10^{-2}$ &$1.97\times 10^{-4}$ & $3.89\times 10^{-4}$\\ \hline 
\end{tabular}
\caption{Comparison of the approximation error for the model-free PhiBE solver and the LSTD method under linear dynamics.}\label{table:lqr}
\end{table}

\begin{figure}
\centering
     \begin{subfigure}[b]{0.24\textwidth}
         \centering
         \includegraphics[width=\textwidth]{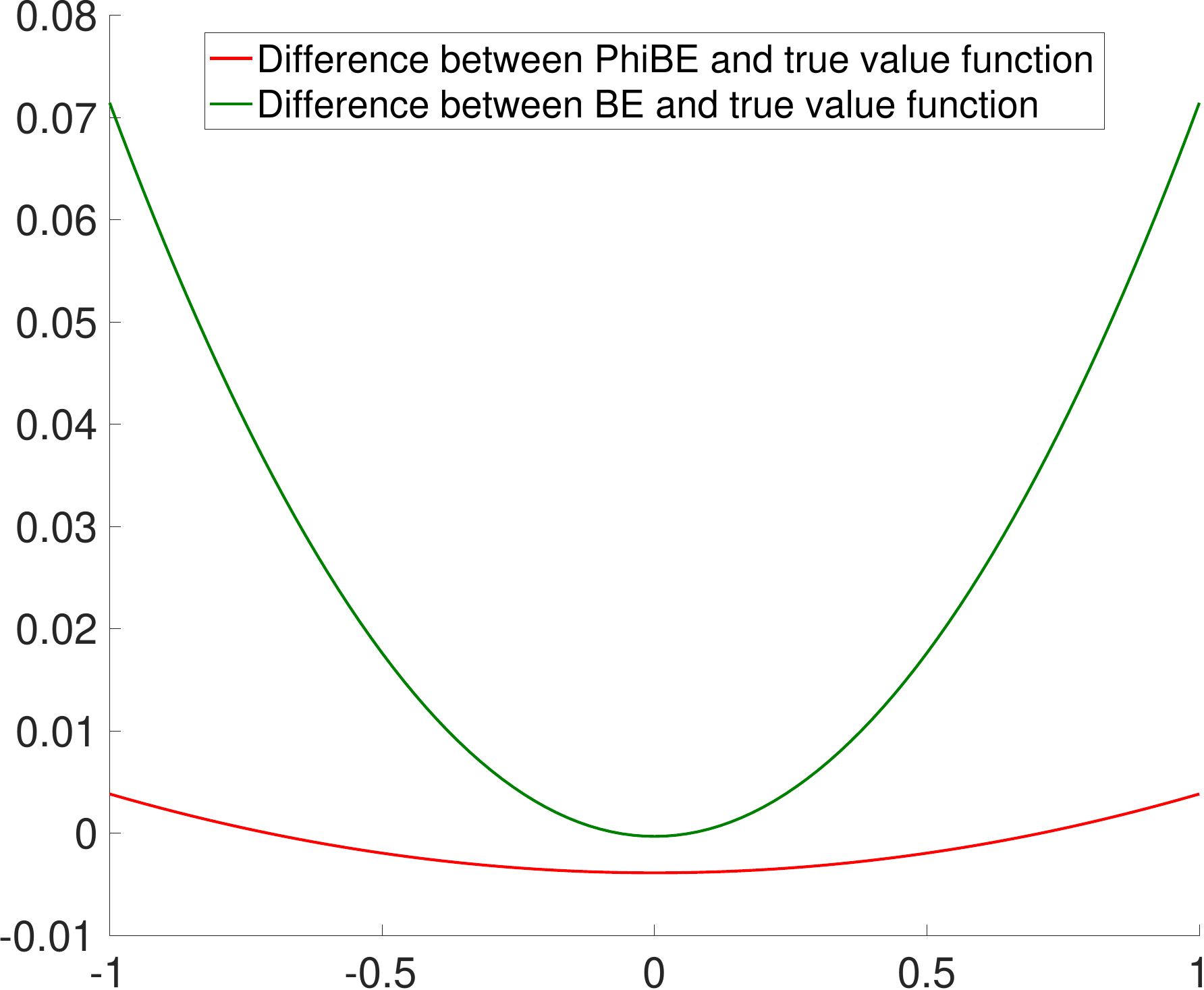}
         \caption{Baseline}
     \end{subfigure}
     \begin{subfigure}[b]{0.24\textwidth}
         \centering
         \includegraphics[width=\textwidth]{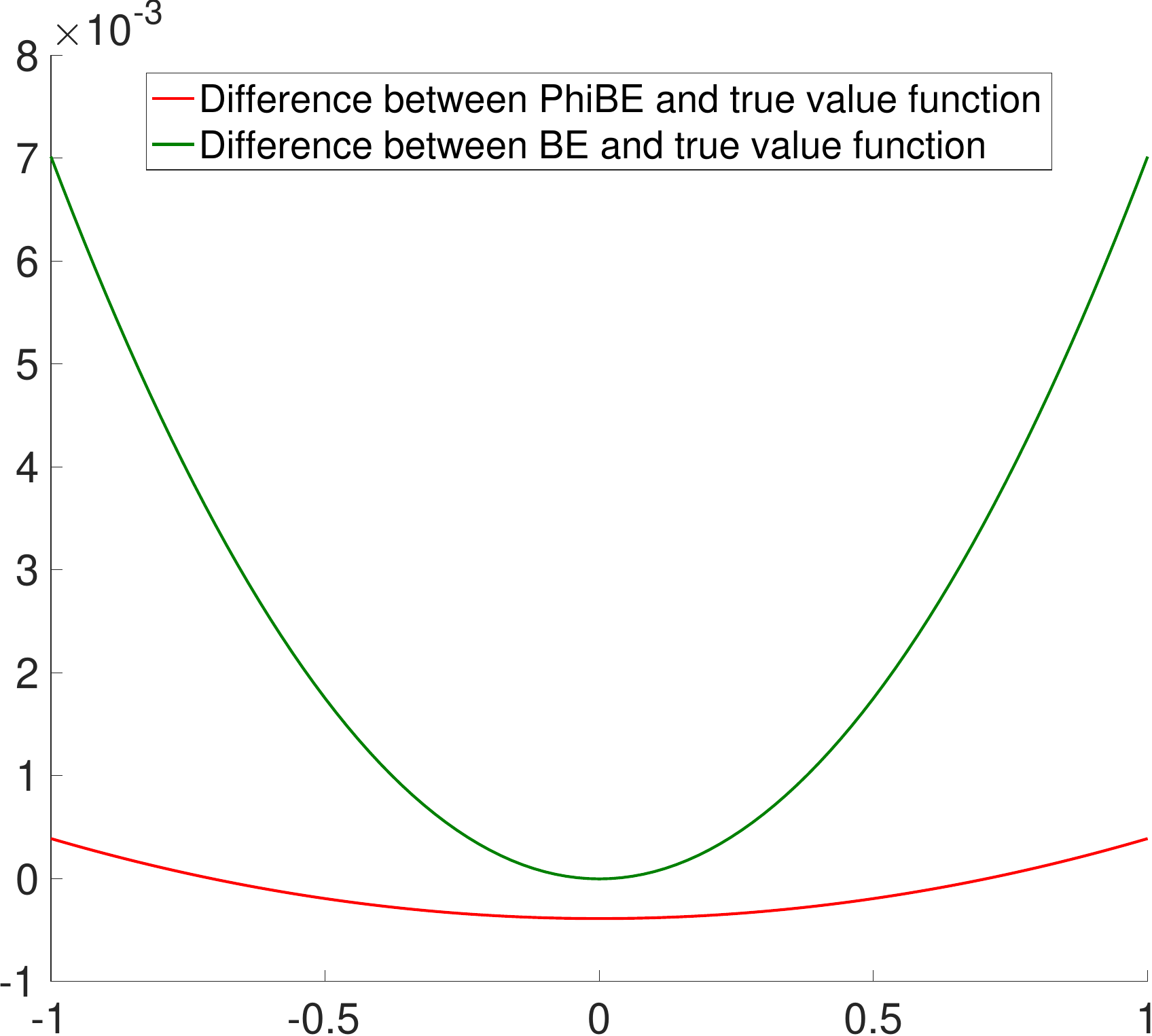}
         \caption{Smaller $\dt$}
     \end{subfigure}
     \begin{subfigure}[b]{0.24\textwidth}
         \centering
         \includegraphics[width=\textwidth]{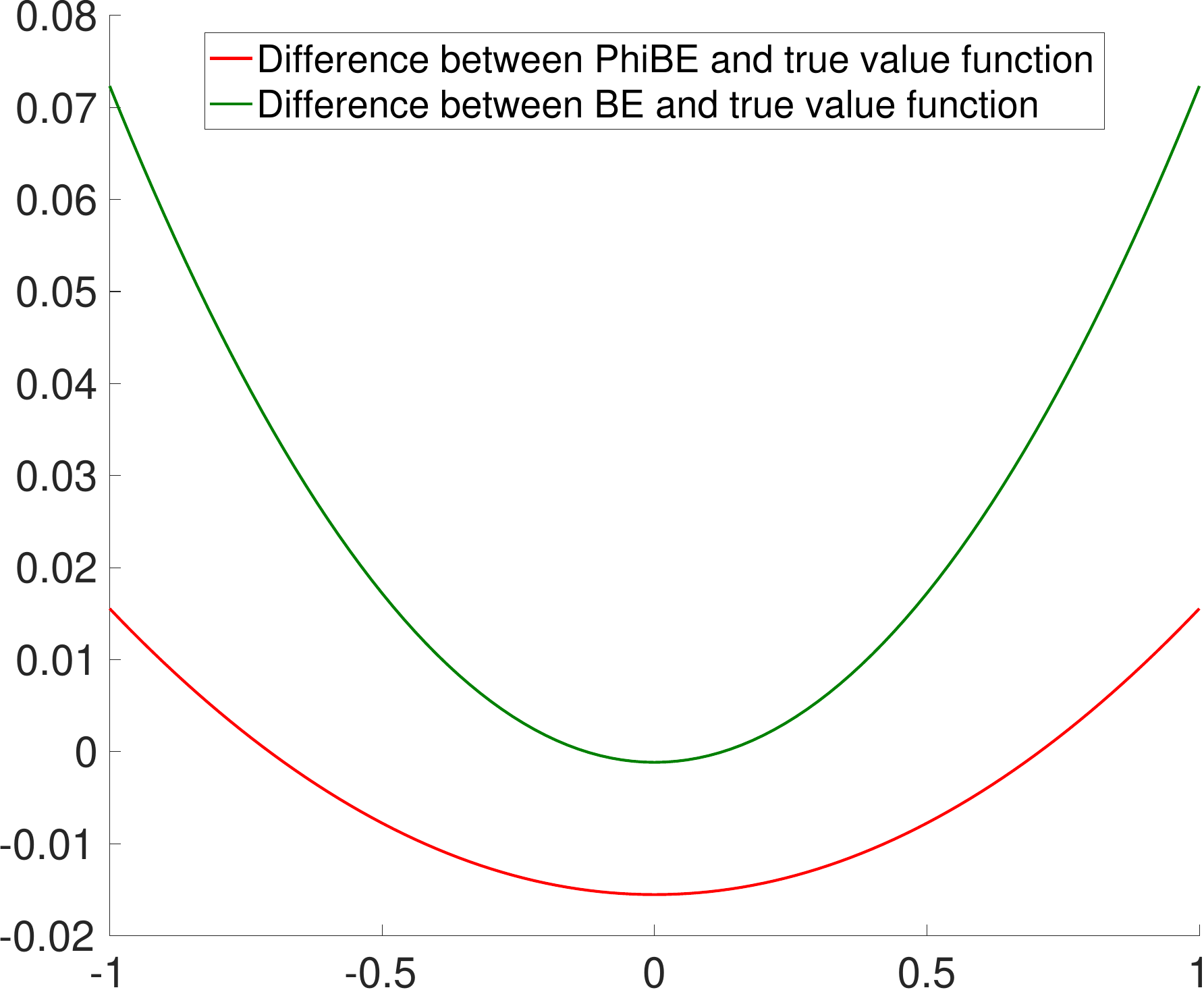}
         \caption{Quicker dynamics}
     \end{subfigure}
     \begin{subfigure}[b]{0.24\textwidth}
         \centering
         \includegraphics[width=\textwidth]{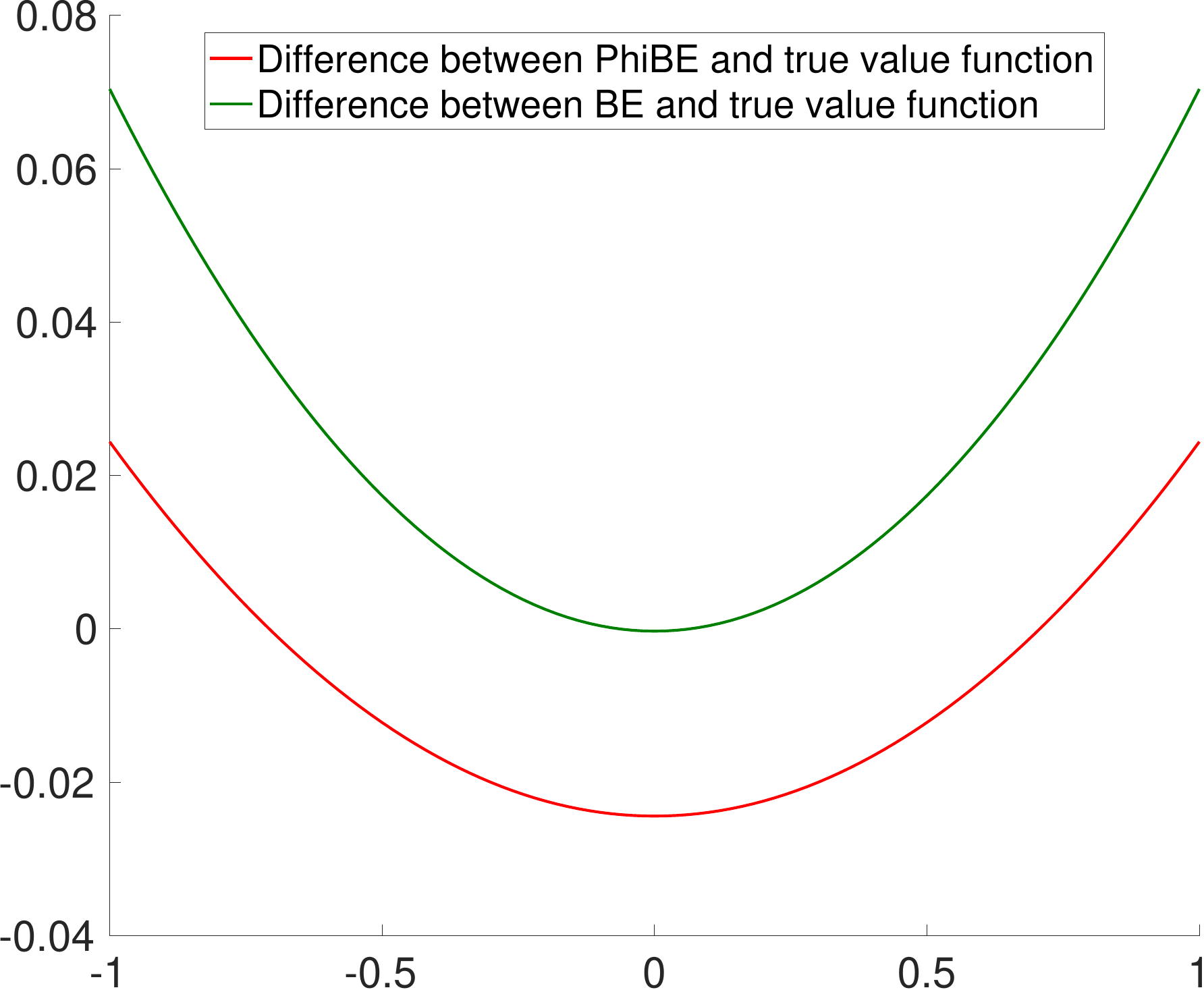}
         \caption{Smaller $\beta$}
     \end{subfigure}
     \caption{
     Stabilization control under linear dynamics. 
The error $\hat{V}(s) - V(s)$ from the PhiBE solution (red) and 
$\t{V}(s) - V(s)$ from the BE solution (green) are plotted, where 
$V(s)$, $\widehat{V}(s)$, and $\widetilde{V}(s)$ are given in 
\eqref{lqr-true}, \eqref{phibe-anlytic}, and \eqref{be-anlytic}. 
Each subplot corresponds to one scenario in \eqref{lqr-diff case}.
}
     \label{fig:lqr-anlytic}
\end{figure}

\begin{figure}
\centering
     \begin{subfigure}[b]{0.24\textwidth}
         \centering
         \includegraphics[width=\textwidth]{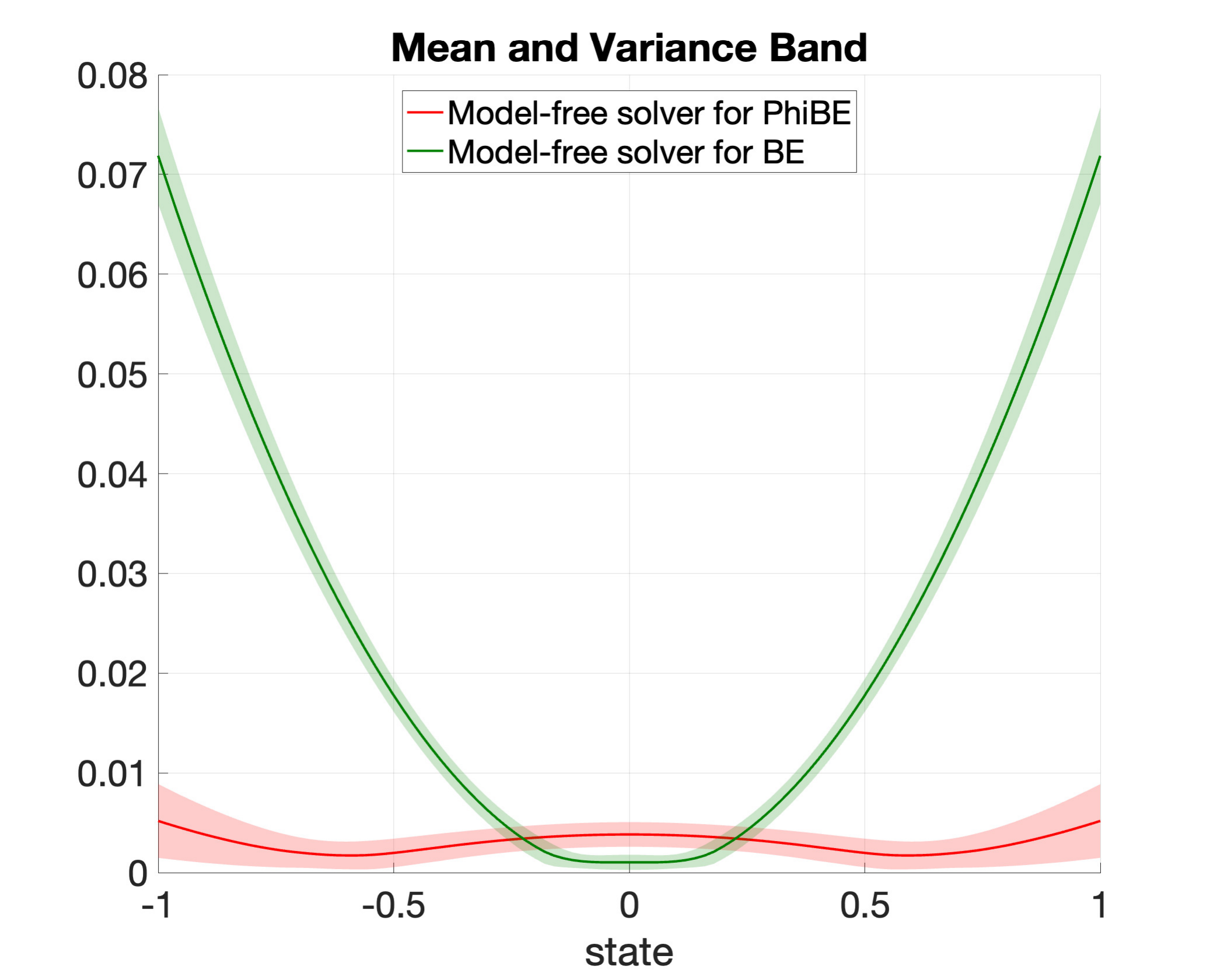}
         \caption{Baseline: $n = 10^6$}
     \end{subfigure}
     \begin{subfigure}[b]{0.24\textwidth}
         \centering
         \includegraphics[width=\textwidth]{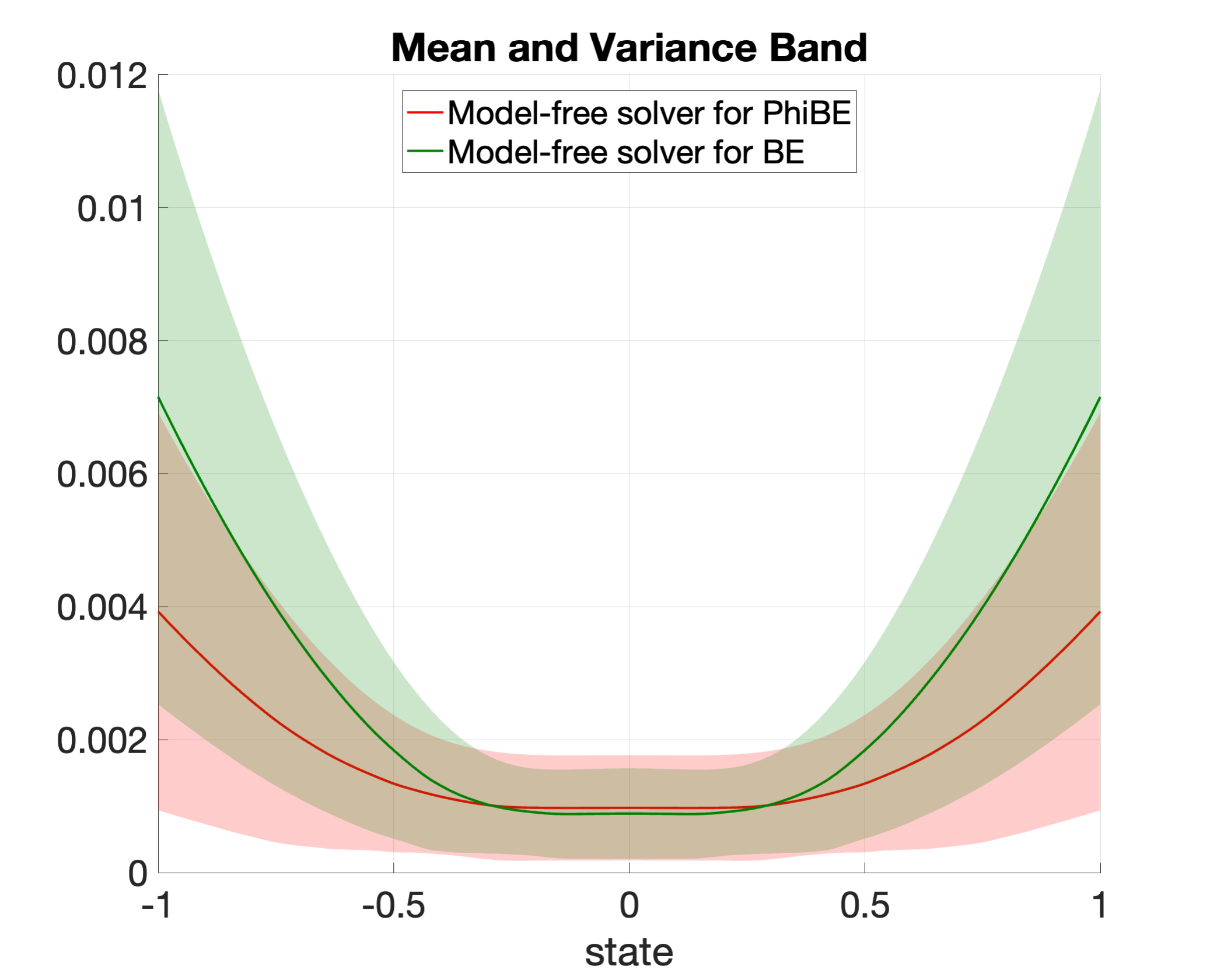}
         \caption{Smaller $\dt$: $n = 10^7$}
     \end{subfigure}
     \begin{subfigure}[b]{0.24\textwidth}
         \centering
         \includegraphics[width=\textwidth]{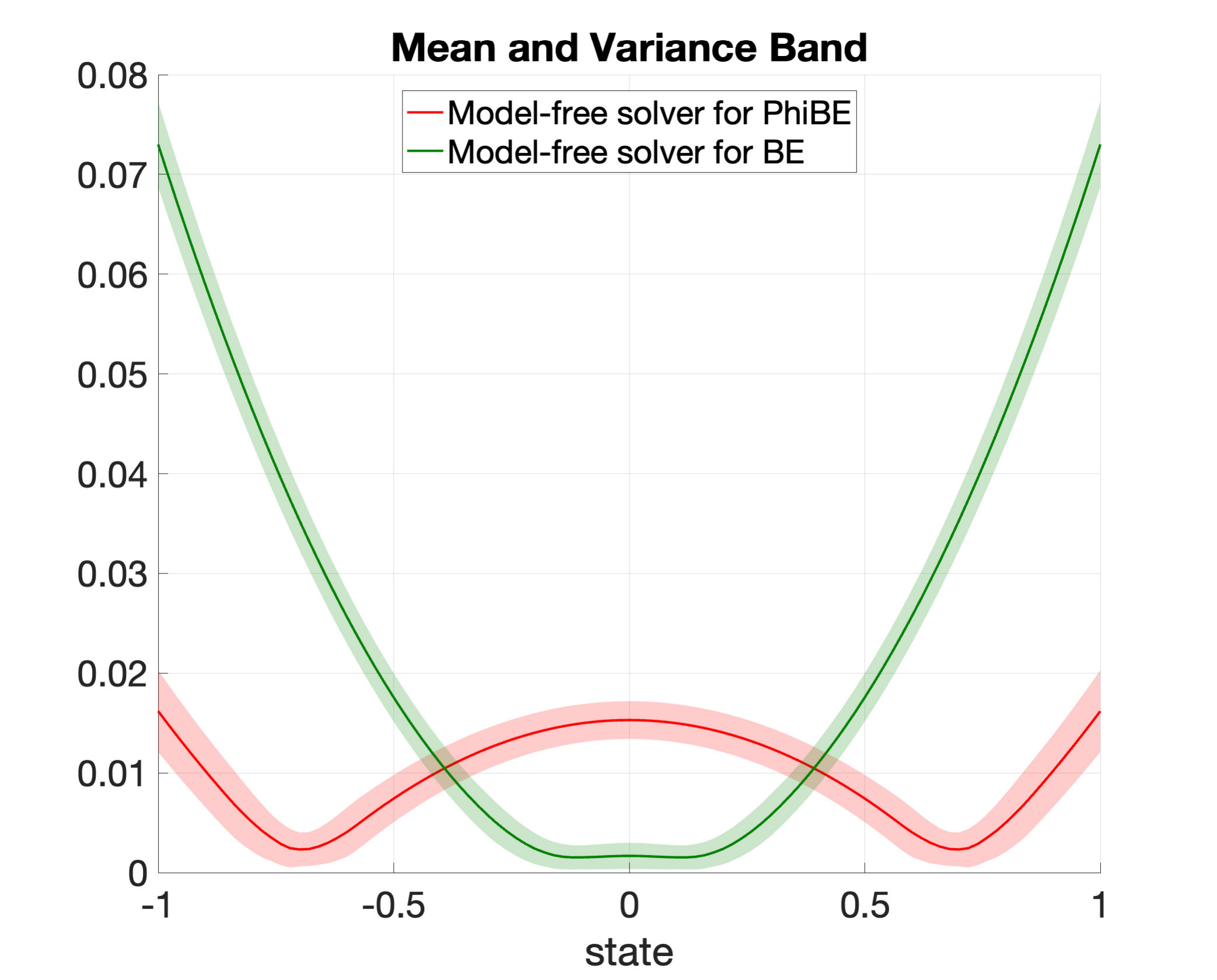}
         \caption{{\tiny Quicker dynamics: $n=10^6$ }}
     \end{subfigure}
     \begin{subfigure}[b]{0.24\textwidth}
         \centering
         \includegraphics[width=\textwidth]{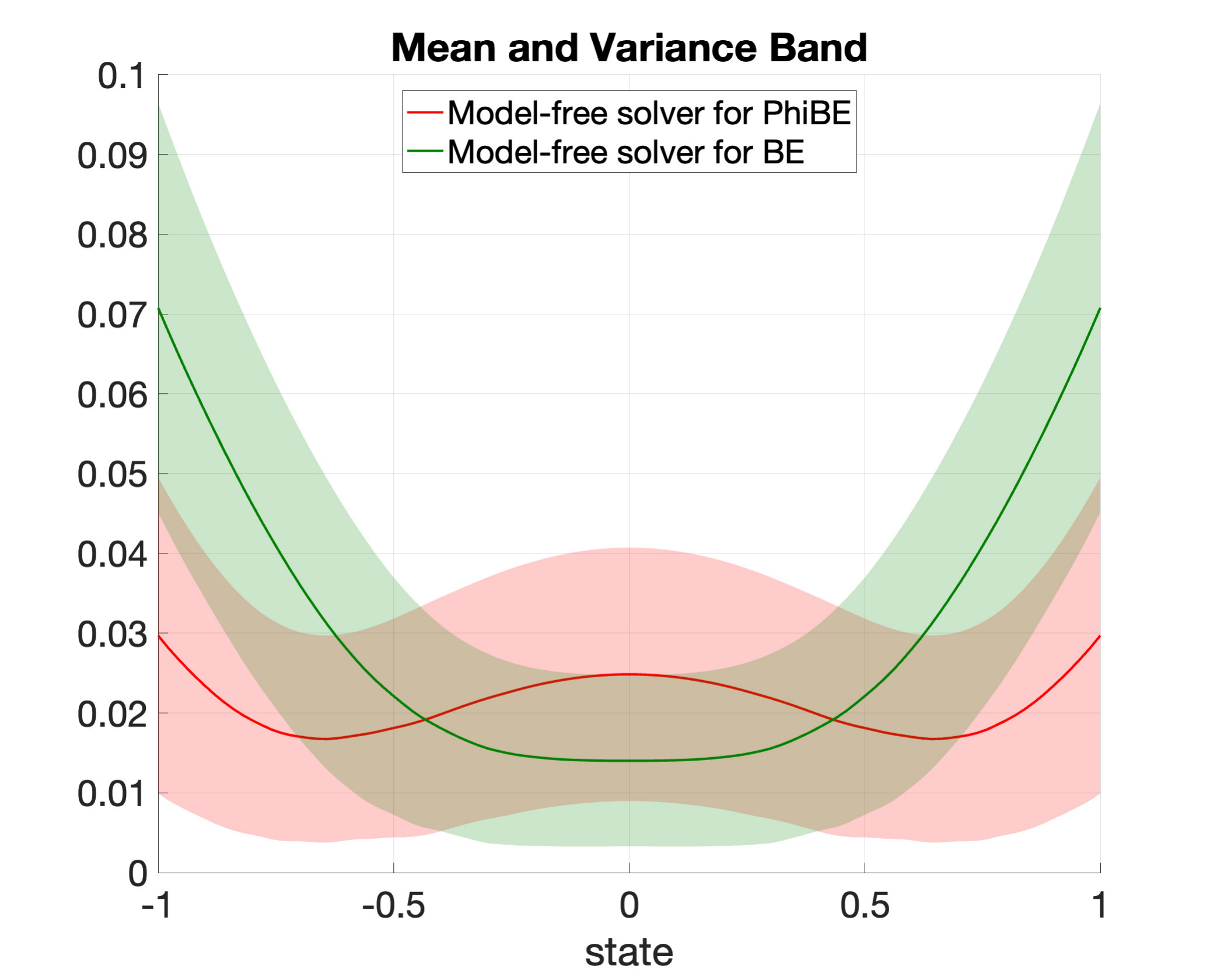}
         \caption{Smaller $\beta$: $n = 10^7$}
     \end{subfigure}
     \caption{
     Stabilization control under linear dynamics. 
The mean and variance of the approximation error 
$|\hat{V}(s) - V(s)|$ from the PhiBE algorithm 
(Algorithm~\ref{algo:galerkin_phibe_stoch}) are shown in red, and those of 
$\t{V}(s) - V(s)$ from the LSTD method 
\eqref{galerkin_be} are shown in green. 
Each subplot corresponds to one scenario in \eqref{lqr-diff case}, 
illustrating how the error behaves under different values of 
$\dt$, $\beta$, and the system dynamics.}
     \label{fig:lqr-data-driven}
\end{figure}

\subsubsection{Nonlinear Dynamics}
In this section, we demonstrate that our approach also applies to nonlinear dynamics. In this setting, the true value function is no longer available in closed form, so we approximate the ground-truth solution using a high-accuracy Galerkin method with sufficiently many polynomial basis functions. The data collection procedure follows the linear case, except that we generate each $s_i'$ by simulating the SDE using the Euler-Maruyama scheme with step size $\delta_t = 10^{-3}$.

We consider the following four cases, setting $q = 1$, $r = 0.1$, $K = 2$, and $n = 10^6$ for all experiments:
\begin{equation}\label{nlqr-diff case}
\begin{aligned}
    &\text{case 1\ \   Baseline:} \ \alpha = b = \kappa = 0.1, \sigma = 0.05, \beta = 1, \dt = 0.1;\\
    &\text{case 2\ \  More frequent observations (smaller $\dt$):}\  \alpha = b = \kappa = 0.1, \sigma = 0.05, \beta = 1, \dt = 0.01;\\
    &\text{case 3\ \   Quicker dynamics:}\  \alpha = b = \kappa = 1/2, \sigma = 1/4, \beta = 1, \dt = 0.1;\\
    &\text{case 4\ \  Less discounted (smaller $\beta$):}\   \ \alpha = b = \kappa = 0.1, \sigma = 0.05, \beta = 0.1, \dt = 0.1.
\end{aligned}
\end{equation}

The resulting approximation errors are shown in Figure~\ref{fig:nonlinear-qr-data-driven}. The qualitative behavior is similar to the linear-dynamics experiments: PhiBE exhibits smaller mean error and more stable performance across all cases.

One caveat appears in Case~3, where both methods exhibit a few significant outliers. These outliers inflate the empirical variance, especially for LSTD, which in one run produced an error nearly two orders of magnitude larger than the average. This behavior is sensitive to the random sample, and a systematic analysis of such instability phenomena is left for future work.

\begin{figure}
\centering
     \begin{subfigure}[b]{0.24\textwidth}
         \centering
         \includegraphics[width=\textwidth]{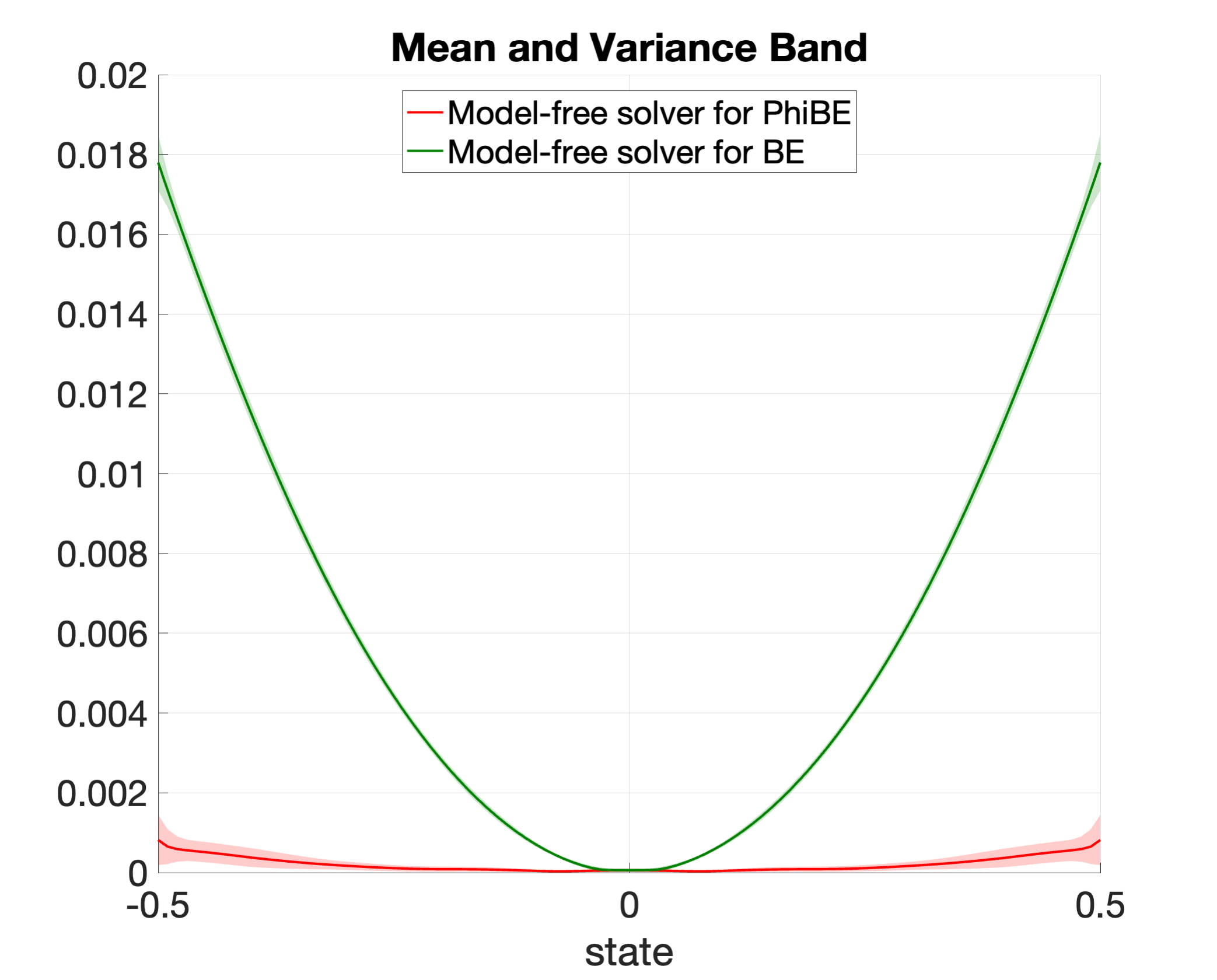}
         \caption{Baseline}
     \end{subfigure}
     \begin{subfigure}[b]{0.24\textwidth}
         \centering
         \includegraphics[width=\textwidth]{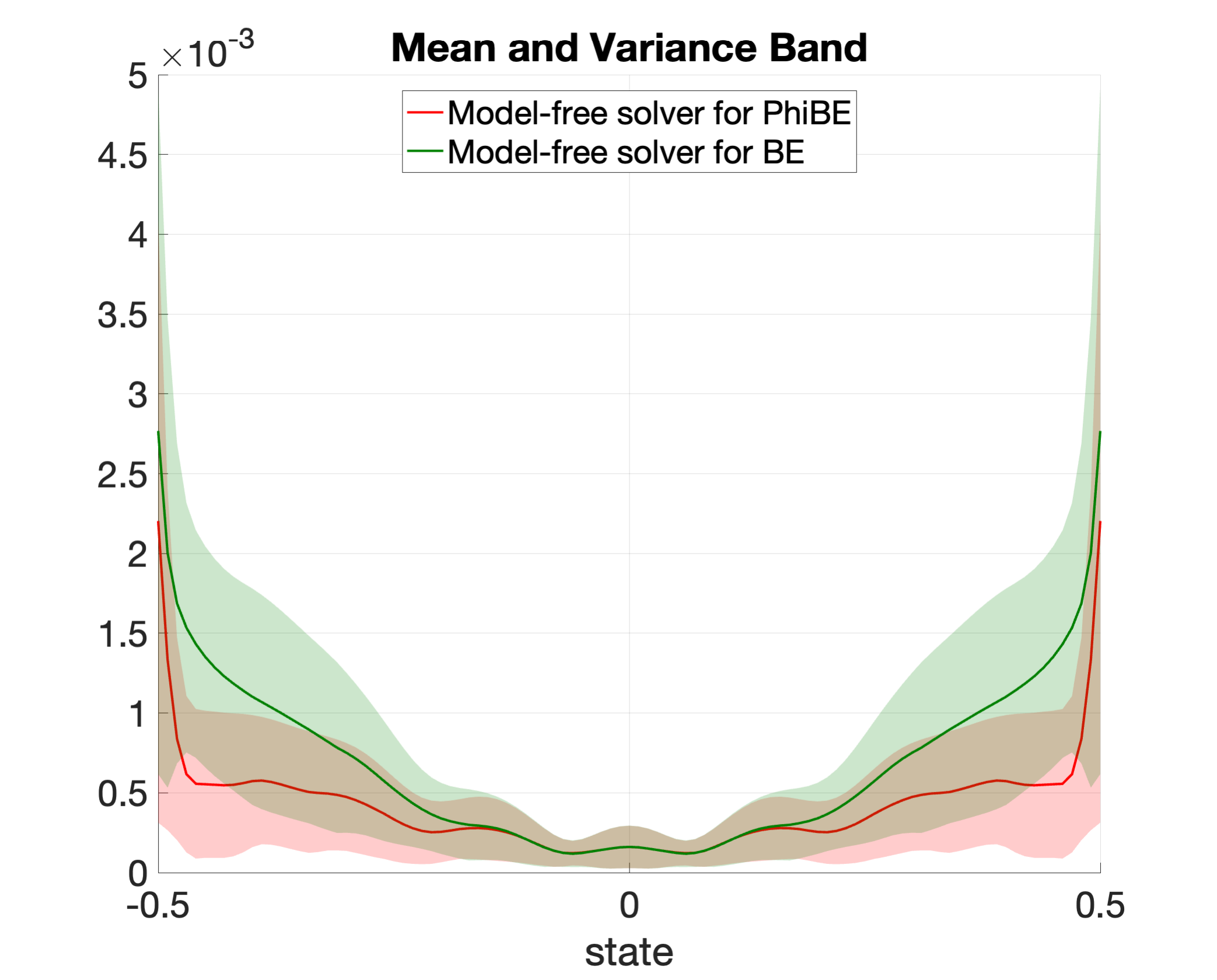}
         \caption{Smaller $\dt$}
     \end{subfigure}
     \begin{subfigure}[b]{0.24\textwidth}
         \centering
         \includegraphics[width=\textwidth]{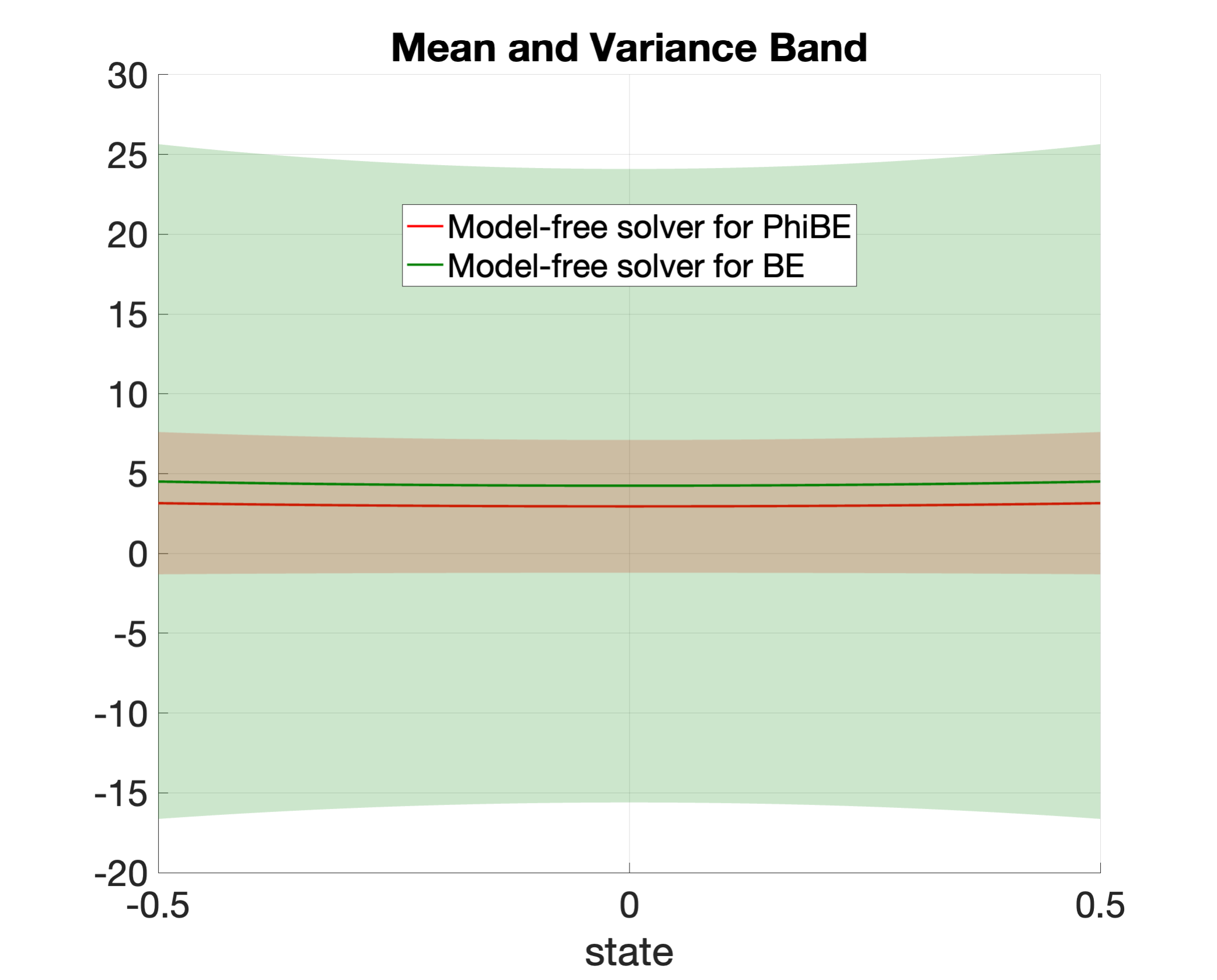}
         \caption{{Quicker dynamics }}
     \end{subfigure}
     \begin{subfigure}[b]{0.24\textwidth}
         \centering
         \includegraphics[width=\textwidth]{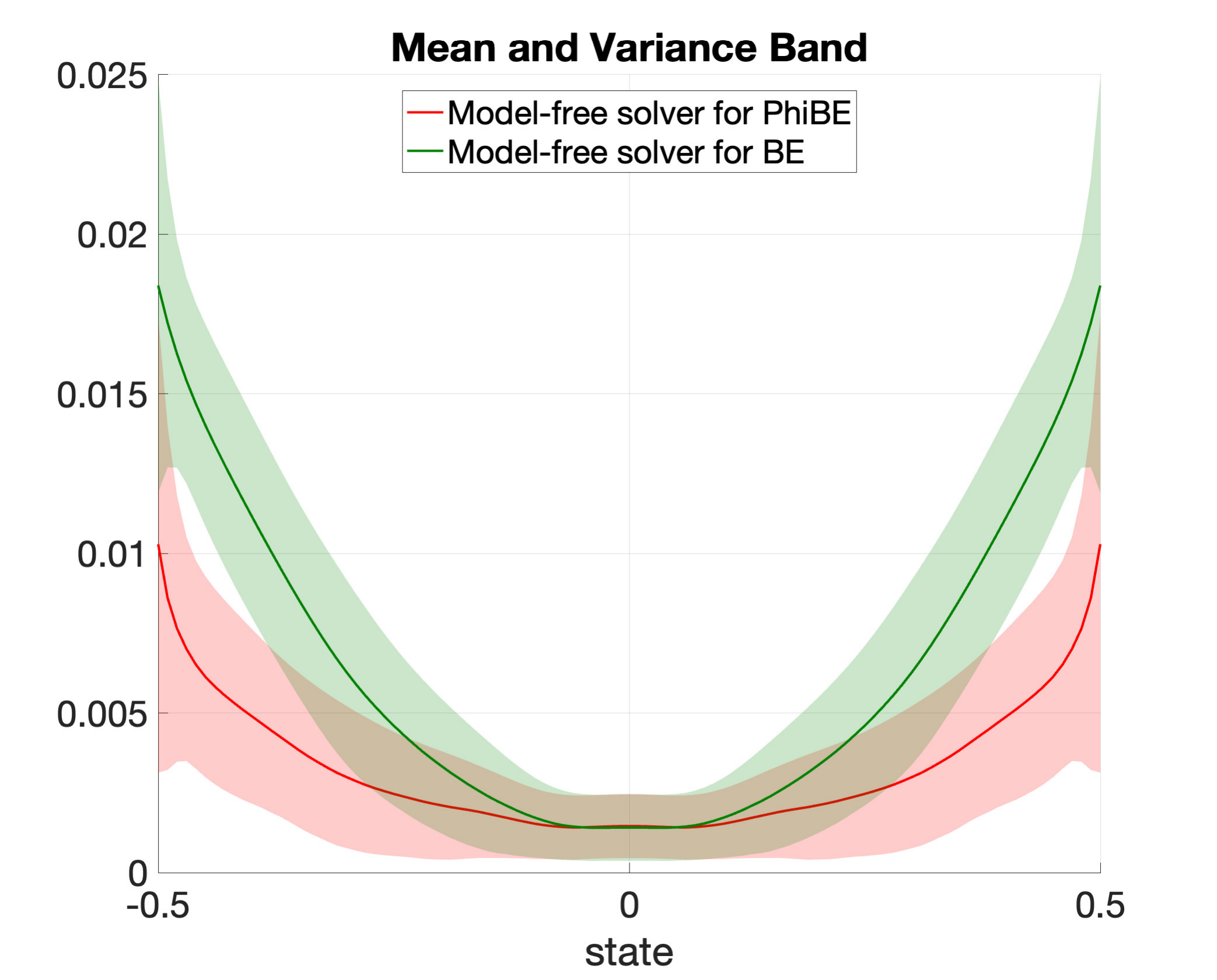}
         \caption{Smaller $\beta$}
     \end{subfigure}
     \caption{Stabilization control under nonlinear dynamics. 
The mean and variance of the approximation error 
$|\hat{V}(s) - V(s)|$ from the PhiBE algorithm 
(Algorithm~\ref{algo:galerkin_phibe_stoch}) are shown in red, and those of 
$\t{V}(s) - V(s)$ from the LSTD method 
\eqref{galerkin_be} are shown in green. 
Each subplot corresponds to one scenario in \eqref{nlqr-diff case}, 
illustrating how the error behaves under different values of 
$\dt$, $\beta$, and the system dynamics.}
     \label{fig:nonlinear-qr-data-driven}
\end{figure}

\subsection{High-Dimensional Case}
We next evaluate the algorithms on a $10$-dimensional linear-quadratic PE problem:
\[
\begin{aligned}
    &V(s) = 
    \E\l[\int_0^\infty e^{-\beta t} (s_t^\top Q s_t)dt | s_0 = s\r]\\
    &ds_t=As_t \,dt + \sigma\,dW_t,
\end{aligned}
\]
where $Q\in\R^{10\times 10}$ is constructed to have eigenvalues $(1,\ldots,10)$ by generating a random orthonormal matrix $O$ and setting $Q = O^\top \text{diag}(1,\ldots,10)\, O$.
The drift matrix $A\in\R^{10\times 10}$ is generated analogously with eigenvalues $-(1,\ldots,10)\times 0.1$.  
The diffusion matrix $\Sigma\in\R^{10\times 10}$ is diagonal matrix with all diagonal entries equal to $0.3$, and we set $\beta = 1$.

We use polynomial basis functions up to second order. The dataset consists of samples $(s_i, s_i')_{i=1}^n$, where  $s_i \sim \mathrm{Unif}([-1,1]^{10})$, and $s_i'$ is obtained by simulating the SDE for one step of size $\dt$ starting from $s_i$. Since achieving a smaller discretization error at smaller $\Delta t$ requires more samples, we choose
$n = 10^5,\; 10^6,\; 10^8 \qquad \text{for} \qquad \Delta t = 1,\; 0.1,\; 0.01,$ respectively.

Figure~\ref{fig:10-dim} reports the results for LSTD and the model-free PhiBE algorithm. The findings closely parallel the one-dimensional case: both error curves decay approximately at first order in $\Delta t$, confirming that the approach scales effectively to higher dimensions and that the theoretical predictions remain valid in the high-dimensional setting.

\begin{figure}
\centering
         \includegraphics[width=0.4\textwidth]{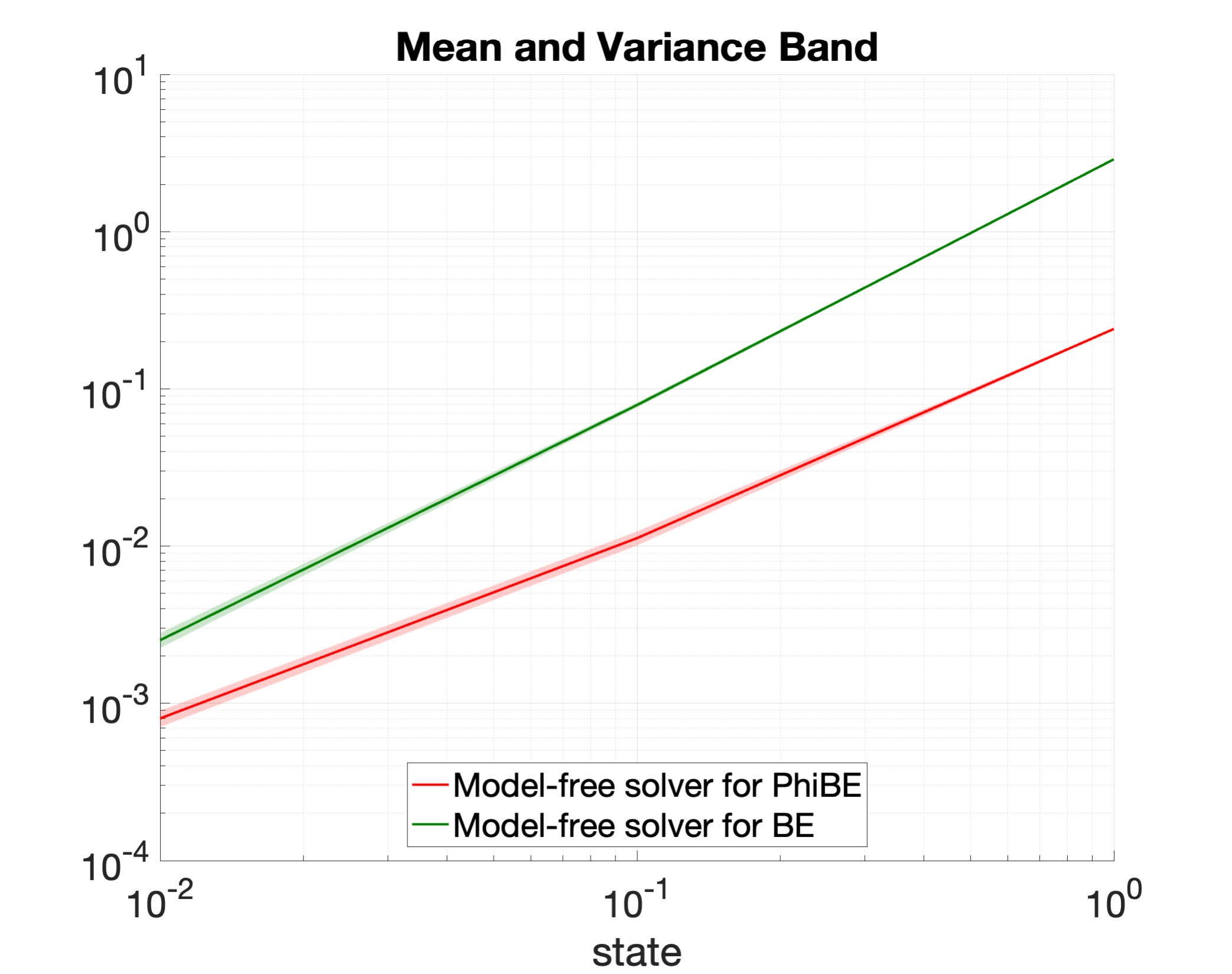}
     \caption{Stabilization control under linear dynamics with high-dimensional state space $s\in \R^{10}$. 
The mean and variance of the approximation error 
$\ll \hat{V}(s) - V(s)\rl_\rho$ from the PhiBE algorithm 
(Algorithm~\ref{algo:galerkin_phibe_stoch}) are shown in red, and those of 
$\ll \t{V}(s) - V(s)\rl_\rho$ from the LSTD method 
\eqref{galerkin_be} are shown in green. }
     \label{fig:10-dim}
\end{figure}

\color{black}

\section{Proofs}
\subsection{Proof of Theorem \ref{thm:rl}.}\label{sec:proof of thm rl}
Let $\rho(s',t|s)$ be the probability density  function of $s_t$ that starts from $s_0 = s$, then  it satisfies the following PDE
\begin{equation}\label{def of pt_t rho}
    \pt_t\rho(s',t|s) = \nb\cdot[\mu(s')\rho(s',t|s)] + \frac12\sum_{i,j}\pt_{s_i}\pt_{s_j}[\Sig_{ij}(s') \rho(s',t|s)]. 
\end{equation}
with initial data $\rho(s',0|s) = \delta_{s}(s')$.
Let $f(t,s) = e^{-\b t}r(s)$, then
\begin{equation}\label{ineq_10}
    \begin{aligned}
    &V(s) - \tV(s) = \E\l[\sum_{i=0}^\infty \int_{\dt i }^{\dt (i+1)}f(t,s_t) - f(\dt i,s_{\dt i}) dt|s_0 = s\r]\\
    =& \sum_{i=0}^\infty \int_{\dt i }^{\dt (i+1)}\l(\int_\S f(t,s') \rho(s',t|s)-  f(\dt i,s')\rho(\dt i, s')ds'\r)dt.
\end{aligned}
\end{equation}
Since
\begin{equation}\label{ineq_6}
\begin{aligned}
    &\int_\S f(t,s') \rho(s',t|s)-  f(\dt i,s')\rho(\dt i, s')ds'\\
    =&\int_\S f(t,s') (\rho(s',t|s)-\rho(s',\dt i|s))+(f(t,s') -  f(\dt i,s'))\rho(s',\dt i|s)ds'\\
    = & \int_\S f(t,s')\pt_t\rho(s',\xi_1|s)(t - \dt i) + \pt_tf(\xi_2,s')(t - \dt i)\rho(s',\dt i |s)ds'\qd \\
    &\text{where }\xi_1, \xi_2 \in (\dt i, \dt(i+1))\\
    = & \int_\S \mL_{\mu,\Sig} f(t,s'),
    \rho(s',\xi_1|s) (t - \dt i) ds' - \int_\S\b e^{-\b \xi_2}r(s')\rho(s',\dt i |s)(t - \dt i) ds'\\
    = & \l(e^{-\b t}\int_\S \mL_{\mu,\Sig} r(s') 
    \rho(s',\xi_1|s)  ds' - \b e^{-\b \xi_2}\int_\S r(s')\rho(s',\dt i |s) ds' \r)(t - \dt i),
    \end{aligned}   
\end{equation}
where the second equality is due to the mean value theorem, and the third equality is obtained by inserting the equation \eqref{def of pt_t rho} for $\rho(s',t|s)$ and integrating by parts. Therefore, for $t\in[\dt i, \dt(i+1)],$
\[
\begin{aligned}
    &\lv \int_\S f(t,s') \rho(s',t|s)-  f(\dt i,s')\rho(\dt i, s')ds'\rv \\
    \leq&  \ll \mL_{\mu,\Sig} r  \rl_\Linf e^{-\b \dt i}(t - \dt i)  + \b e^{-\b \dt i}\ll r \rl_\Linf (t - \dt i) .
    \end{aligned}
\]
Therefore, one has
\[
\begin{aligned}
    &\ll V(s) - \tV(s)\rl_\Linf\\
    \leq&\sum_{i=0}^\infty \int_{\dt i }^{\dt (i+1)}\ll \mL_{\mu,\Sig} r  \rl_\Linf e^{-\b \dt i}(t - \dt i)  + \b e^{-\b \dt i}\ll r \rl_\Linf (t - \dt i)  dt\\
    \leq &\l(\ll \mL_{\mu,\Sig} r  \rl_\Linf + \b\ll r \rl_\Linf \r)\sum_{i=0}^\infty e^{-\b \dt i} \int_{\dt i }^{\dt (i+1)} (t - \dt i) dt\\
    \leq &\frac{1}{2}\l(\ll \mL_{\mu,\Sig} r  \rl_\Linf + \b\ll r \rl_\Linf \r) \sum_{i=0}^\infty e^{-\b \dt i} \dt^2  = \frac{C}{1-e^{-\b \dt} }\dt^2  =\frac{C}{\b}\dt + C\l(\frac{1}{1-e^{-\b\dt}}\dt^2 -\frac\dt\b \r),
\end{aligned}
\]
where $C =\frac{1}{2}\l(\ll \mL_{\mu,\Sig} r  \rl_\Linf + \b\ll r \rl_\Linf \r) $.
Since 
\[
\lim_{\dt\to0}C\l(\frac{1}{1-e^{-\b\dt}}\dt^2 - \frac\dt\b\r)\frac{1}{\dt} = 0,
\]
one has, 
\[
\begin{aligned}
    &\ll V(s) - \tV(s) \rl_\Linf = \frac{L\dt}{\b} + o(\dt).
\end{aligned}
\]

\subsection{Proof of Theorem \ref{thm:v-vhat_deter}}\label{sec:proof of them Phibe deter}
Note that the true value function $V$ and the $i$-th order PhiBE solution $\hV$ satisfies 
\[
\b V(s) =  r(s)+ \mu(s)\cdot \nb V(s), \quad  \b \hV_i(s) =  r(s)+ \hmu_i(s)\cdot \nb \hV_i(s). 
\]
First, by the following lemma, one can bound $\ll V - \hV_i \rl_\Linf $ with $\ll \mu - \hmu_i \rl_\Linf $.
\begin{lemma}\label{lemma:Lhat - L}
For function $V$ and $\hV$ satisfies, 
\[
\b V(s) =  r(s)+ \mu(s)\cdot \nb V(s), \quad  \b \hV(s) =  r(s)+ \hmu(s)\cdot \nb \hV(s), 
\]
the distance between $V$ and $\hV$ can be bounded by 
\[
\ll V - \hV \rl_\Linf \leq \frac{2\ll \mu - \hmu \rl_\Linf \ll \nb r \rl_\Linf}{(\b - \ll \nb \mu \rl_\Linf)^2}.
\]
\end{lemma}
(See Section \ref{proof of lemma Lhat - L} for the proof of the above lemma.)
Therefore, one has
\begin{equation}\label{ineq_13}
    \ll V - \hV_i \rl_\Linf \leq  \frac{\ll \mu - \hmu_i \rl_\Linf \ll \nb r \rl_\Linf}{(\b - \ll \nb \mu \rl_\Linf)^2}.
\end{equation}
Then by the following lemma, one can further bound $\ll \mu - \hmu_i \rl_\Linf$. 
\begin{lemma}\label{lemma:muhat - mu}
The distance between $\hmu_i(s)$ defined in \eqref{def of higher order deter mu} and the true dynamics can be bounded by
\[
\ll \hmu_i(s) - \mu(s) \rl_\Linf \leq C_i\ll \mL^i_{\mu}\mu(s)\rl_\Linf\dt^i,
\] 
where $\mL_\mu$ is defined in \eqref{def of mLmu}, and
\begin{equation}\label{def of C_i}
C_i = \frac{\sum_{j=0}^i|\coef{i}_j|  j^{i+1} }{(i+1)!}.
\end{equation} 
\end{lemma}
(See Section \ref{proof of lemma muhat - mu} for the proof of the above lemma.) Hence, one completes the proof by applying the above lemma to \eqref{ineq_13}
\[
\ll V - \hV_i \rl_\Linf \leq \frac{C_i\ll \mL_\mu^i\mu \rl_\Linf\ll \nb r\rl_\Linf }{(\b - \ll \nb \mu \rl_\Linf)^2} \dt^i.
\]

\subsection{Proof of Theorem \ref{thm: linear deter}}\label{proof of them linear}
\def\da{\Delta_i}
\def\hci{{C}_i}

For the linear dynamics $ds_t = As_t dt$, one has $s_t = e^{A t}s_0$, therefore, 
\[
\hat{\mu}(s) = \frac1\dt\sum_{j=1}^i \coef{i}_j(s_{j
dt} - s_0) = \frac1\dt\sum_{j=1}^i \coef{i}_j(e^{Aj\dt} - I)s.
\]
Let $\hat{A}_i = \coef{i}_j(e^{Aj\dt} - I)$, then the i-th order PhiBE reads,
\[
\begin{aligned}
    \b\hV_i(s) = r(s) + \hat{A}_is\cdot \nb \hV_i(s),
\end{aligned}
\]
which can equivalently written as 
\[
\begin{aligned}
    \b\hV_i(s) = \int_0^\infty e^{-\beta t} r(e^{\hat{A}_i t}s) dt.
\end{aligned}
\]
Now define $\phi(\th) = e^{At}e^{\th \da t}s$, where $\da = \hat{A} - A$, then $e^{At}s = \phi(0), \quad e^{\hat{A}t}s=\phi(1)$,
which implies
\[
\begin{aligned}
    &r(e^{At}s) - r(e^{\hat{A}t}s) = r(\phi(0)) - r(\phi(1)) = \int_0^1 \frac{d}{d\th} r(\phi(\th)) d\th = \int_0^1 \nb (r(\phi(\th)))^\top \phi'(\th) d\th\\
    =&  \l[\int_0^1 (\nb r(\phi(\th)))^\top \phi(\th)   d\th\r]  \da t .
\end{aligned}
\]
Therefore, one has
\[
\begin{aligned}
    &\lv V(s) - \hV_1(s) \rv = \lv \int_0^\infty e^{-\beta t} \l(r(e^{At}s) - r(e^{\hat{A}t}s)\r) dt \rv \\
    = &  \lv \int_0^1 \l(\int_0^\infty (\nb r(\phi(\th)))^\top \phi(\th) e^{-\beta t} t dt\r) d\th  \da  \rv \leq  \ll u\cdot r(u) \rl_\Linf  \int_0^1 \l(\int_0^\infty e^{-\beta t} t dt\r) d\th  \lv \Delta_A \rv  \\
    =& \frac{\ll u\cdot r(u) \rl_\Linf }{\beta^2} \ll \da \rl.
\end{aligned}
\]
Next we estimate $\ll \da\rl$. 
First note that 
    \[
    \h{A}_i =  \frac1\dt\sum_{j=1}^i \coef{i}_j(e^{Aj\dt} - I) =  \frac1\dt\sum_{j=1}^i \coef{i}_j\l(\sum_{k=1}^i \frac{1}{k!}(Aj\dt)^k + R_{ij}\r)
    \]
    where 
    \[
    R_{ij} = e^{Aj\dt} - \sum_{k=0}^i \frac{1}{k!}(Aj\dt)^k = \frac{A^{i+1}(j\dt)^{i+1}}{(i+1)!}e^{A\xi}, \quad \text{for }\xi\in[0,j\dt),\quad 
    \]
    and therefore
    \[
    \ll R_{ij}  \rl \leq \frac{\ll A \rl^{i+1} (j\dt)^{i+1} }{(i+1)!}e^{\ll A\rl j\dt}\leq   \frac{\ll A \rl^{i+1} (j\dt)^{i+1} }{(i+1)!} e^{\ll A \rl i\dt}.
    \]
    By the definition of $\coef{i}_j$, one has
    \[
    \h{A}_i  =  \sum_{k=1}^i\frac{1}{k!}A^k\dt^{k-1} \l(\sum_{j=1}^i \coef{i}_j j^k\r)+  \frac1\dt\sum_{j=1}^i \coef{i}_jR_{ij} = A + \frac1\dt\sum_{j=1}^i \coef{i}_jR_{ij},
    \]
    which leads to
    \begin{equation}\label{diff of A Ahat}
        \ll \da \rl = \ll \h{A}_i - A \rl \leq \frac1\dt\sum_{j=1}^i \lv \coef{i}_j \rv \ll R_{ij} \rl \leq  \hci\ll A \rl D_A^{i} \dt^i , \quad \text{with }D_A = e^{\ll A \rl \dt}\ll A \rl
    \end{equation}
    where $\hci$ is defined in \eqref{def of C_i}.

\color{black}

\subsection{Proof of Theorem \ref{thm:rl_rho}}\label{sec:proof of them rl stoch}
We first present the property of the operator $\mL_{\mu,\Sig}$ that will be frequently used later in the following Proposition. 
\begin{proposition}\label{lemma:l2 operator}
    For the operator $\mL_{\mu,\Sig}$ defined in \eqref{def of mL}, under Assumption \ref{ass_2}/(a), one has
    \[
    {\la \mL_{\mu,\Sig} V(s), V(s) \ra_\rho \leq -\frac{\lammin}{2} \ll \nb V \rl_{\rho}^2;}
    \]
    \[
    \sum_i\la \pt_{s_i}\mL_{\mu,\Sig} V(s), \pt_{s_i}V(s) \ra_\rho \leq C_{\nb\mu,\nb\Sig} \ll \nb V \rl^2_\rho;
    \]
    \[
    \begin{aligned}
        \la \mL_{\mu,\Sig} f(s), g(s)\ra_\rho
        \leq &\l[\l(\ll \mu\rl_\Linf+ \frac12\ll \nb\cdot\Sig \rl_\Linf \r)\ll g\rl_\rho +  \frac12\ll \Sig\rl_\Linf \ll \nb g\rl_\rho \r]\ll \nb f\rl_\rho\\
        &+ \frac12\ll \Sig\rl_\Linf \ll \nb\log\rho \rl_\rho\ll g\rl_\rho\ll \nb f \rl_\Linf;\\
        \la \mL_{\mu,\Sig} f(s), g(s)\ra_\rho \leq &\l[\l(\ll \mu\rl_\Linf+ \frac12\ll \nb\cdot\Sig \rl_\Linf + \frac12\ll \Sig\rl_\Linf \ll \nb\log\rho \rl_\Linf\r)\ll g\rl_\rho \r.\\
        &\l.+ \frac12\ll \Sig\rl_\Linf \ll \nb g\rl_\rho \r]\ll \nb f\rl_\rho,
    \end{aligned}
    \]
    where $C_{\nb\mu,\nb\Sig}$ is defined in \eqref{def of c123} depending on the first derivatives of $\mu, \Sig$. \\
\end{proposition}
\begin{proof}
Inserting the operator $\mL_{\mu,\Sig}$, and applying integral by parts gives,
\begin{equation}\label{important_ineq}
    \begin{aligned}
 &\la \mL_{\mu,\Sig} V(s), V(s) \ra_\rho = \la\mu \cdot \nb V, V\ra_\rho - \frac12 \sum_{i,j} \la \pt_{s_j}(\Sig_{ij} V\rho), \pt_{s_i}V \ra \\
 =& \sum_i\la\mu_i \rho, \pt_{s_i}\l(\frac12V^2\r)\ra - \frac12 \sum_{i,j} \la \pt_{s_j}(\Sig_{ij} \rho), \pt_{s_i}\l(\frac12V^2\r) \ra - \frac12 \sum_{i,j} \la (\pt_{s_j}V)\Sig_{ij}, \pt_{s_i}V \ra_\rho\\
 =& -\sum_i\la\pt_{s_i}(\mu_i \rho), \frac12V^2\ra + \frac12 \sum_{i,j} \la \pt_{s_i}\pt_{s_j}(\Sig_{ij} \rho), \frac12V^2\ra - \frac12 \int (\nb V)^\top \Sig (\nb V) \rho \,ds\\
  =& \la \nb\cdot\l(- \mu \rho + \frac12\nb\cdot(\Sig\rho)\r), \frac12V^2\ra  - \frac12 \int (\nb V)^\top \Sig (\nb V) \rho \,ds\\
  \leq& \la \nb\cdot\l(- \mu \rho + \frac12\nb\cdot(\Sig\rho)\r), \frac12V^2\ra -\frac{\lammin}{2}\ll \nb V \rl^2_\rho,
\end{aligned}
\end{equation}
where the last inequality is because of the positivity of the matrix $\Sig(s)$ in Assumption \ref{ass_2}. 

If $\rho$ is the stationary solution, then
\[
\la \nb\cdot\l(- \mu \rho + \frac12\nb\cdot(\Sig\rho)\r), \frac12V^2\ra = 0.
\]
If $\rho$ s.t. $\frac{1}{\lammin}\ll - \mu  + \frac12\nb\cdot\Sig + \frac12\Sig\nb
\log\rho\rl_\Linf^2 \leq \frac\beta2$, then
\[
\begin{aligned}
    &\la \nb\cdot\l(- \mu \rho + \frac12\nb\cdot(\Sig\rho)\r), \frac12V^2\ra = \la \l( \mu  - \frac12\nb\cdot\Sig - \frac12\Sig\nb
\log\rho\r)\rho, \nb\l(\frac12V^2\r)\ra\\
\leq & \ll \mu  - \frac12\nb\cdot\Sig - \frac12\Sig\nb
\log\rho\rl_\Linf \ll V \rl_\rho \ll \nb V \rl_\rho \\
\leq& \frac{1}{\lammin}\ll \mu  - \frac12\nb\cdot\Sig - \frac12\Sig\nb
\log\rho\rl_\Linf^2 \ll V \rl_\rho^2 +  \frac\lammin4\ll \nb V \rl_\rho^2
\leq \frac\beta2\ll V \rl_\rho^2 +  \frac\lammin4\ll \nb V \rl_\rho^2.
\end{aligned}
\]
Inserting it back to \eqref{important_ineq}, this completes the proof of the first inequality.
\color{black}

For the second part of the Lemma, first note that 
\[
\pt_{s_i}\mL_{\mu,\Sig} V = \pt_{s_i}\mu\cdot\nb V + \frac{1}{2}\pt_{s_i}\Sig:\nb^2V + \mL_{\mu,\Sig}\pt_{s_i} V.
\]
Therefore, applying the first part of the Lemma gives
\begin{equation*}
\begin{aligned}
    &\sum_i\la \pt_{s_i}\mL_{\mu,\Sig} V(s), V(s) \ra_\rho \\
    \leq  &\sum_i\l(\la \pt_{s_i}\mu\cdot\nb V, \pt_{s_i}V\ra_\rho + \frac{1}{2}\la \pt_{s_i}\Sig:\nb^2V, \pt_{s_i}V\ra_\rho\r) - \frac{\lammin}{4}\sum_{i}\ll \nb\pt_{s_i}V \rl^2_\rho + \frac\beta2\ll \nb V \rl_\rho^2\\
    \leq & \ll \nb \mu \rl_\Linf \ll \nb V \rl_\rho^2 + \frac{1}{2}\ll \nb\Sig \rl_\Linf \ll \nb^2V \rl_\rho \ll \nb V \rl_\rho - \frac{\lammin}{4}\ll \nb^2V \rl^2_\rho+ \frac\beta2\ll \nb V \rl_\rho^2\\
    = &- \frac{\lammin}{4}\l(\ll \nb^2V \rl_\rho - \frac{\ll \nb \Sig \rl_\Linf}{\lammin}\ll \nb V \rl_\rho\r)^2 + \l(\frac{\ll \nb \Sig\rl_\Linf^2}{4\lammin} + \ll \nb \mu \rl_\Linf + \frac\beta2 \r)\ll \nb V \rl_\rho^2\\
    \leq & \frac{C_{\nb\mu,\nb\Sig}}{2}\ll \nb V \rl^2_\rho ,
\end{aligned}
\end{equation*}
where
\begin{equation}\label{def of c123}
\begin{aligned}
{C_{\nb\mu,\nb\Sig} = \frac{\ll \nb \Sig\rl_\Linf^2}{2\lammin} + 2\ll \nb \mu \rl_\Linf + \beta.}
\end{aligned}
\end{equation}

For the last two inequalities, one notes
\[
\begin{aligned}
    &\la \mL_{\mu,\Sig} f, g\ra_\rho \\
    =& \la \mu\cdot\nb f , g\ra_\rho -\frac12\l[\la \nb f \cdot\nb\cdot\Sig, g\ra_\rho + \la \nb f \Sig, \nb g\ra_\rho + \la  \nb f  \Sig, \frac{\nb\rho}{\rho} g\ra_\rho\r]\\
    \leq& \l[\l(\ll \mu\rl_\Linf+ \frac12\ll \nb\cdot\Sig \rl_\Linf \r)\ll g\rl_\rho +  \frac{\ll \Sig\rl_\Linf}2 \ll \nb g\rl_\rho \r]\ll \nb f\rl_\rho + \la  \nb f  \frac{\Sig}{2}, g \nb\log\rho \ra_\rho .\\
\end{aligned}
\]
By bounding the last term differently,
\[
\frac12\ll \nb\log \rho \rl_\Linf \ll \Sig\rl_\Linf\ll g\rl_\rho\ll \nb f\rl_\rho\quad \text{or, }\quad\frac12\ll \nb\log \rho \rl_\rho \ll \Sig\rl_\Linf\ll g\rl_\rho\ll \nb f\rl_\Linf,
\]
one ends up with the last two inequalities of the Lemma. 

\end{proof}

\paragraph{Proof of Theorem \ref{thm:rl_rho}}

Now we are ready to prove Theorem \ref{thm:rl_rho}. 
\begin{proof}
By \eqref{ineq_10}, one has
\begin{equation}\label{ineq_11}
   \begin{aligned}
    &\ll V(s) - \tV(s)\rl_\rho \\
    \leq&  \sum_{i=0}^\infty \sqrt{\dt \int_{\dt i }^{\dt (i+1)}\ll\int_\S f(t,s') \rho(s',t|s)-  f(\dt i,s')\rho(\dt i, s')ds'\rl_\rho^2dt} ,
\end{aligned} 
\end{equation}
where the Jensen's inequality is used.
By \eqref{ineq_6}, one has for $t\in[\dt i, \dt(i+1)]$
\[
\begin{aligned}
    &\ll \int_\S f(t,s') \rho(s',t|s)-  f(\dt i,s')\rho(\dt i, s')ds'\rl_\rho \\
    \leq&  e^{-\b\dt i}(t-\dt i) \l(\ll p_1(\xi_1,s)\rl_\rho +  \b\ll p_2(\dt i, s) \rl_\rho\r),
    \end{aligned}
\]
where
\[
p_1(s,t) = \int_\S \mL_{\mu,\Sig} r(s') 
    \rho(s',t|s)  ds', \quad p_2(s,t) = \int_\S  r(s') 
    \rho(s',t|s)  ds'.
\]
Note that both $p_1(s,t)$ and $p_2(s,t)$ satisfies
\[
\pt_t p_i(s,t) = \mL_{\mu,\Sig} p_i(s,t), \quad \text{with initial data }p_1(0,s) = \mL_{\mu,\Sig} r(s), \quad p_2(0,s) = r(s). 
\]
By Proposition \ref{lemma:l2 operator}, one has
\[
\pt_t\l(\frac1{2}\ll p_i(t) \rl^2_\rho\r) \leq -\frac{\lammin}{4}\ll \nb p_i(t) \rl^2_\rho  + \frac{\beta}{2}\ll p_i(t) \rl^2_\rho,
\]
which implies, 
\[
\ll p_i(t) \rl_\rho\leq e^{\frac{\beta}2 t}\ll p_i(0)\rl_\rho.
\]
Therefore, one has
\[
\begin{aligned}
    &\ll \int_\S f(t,s') \rho(s',t|s)-  f(\dt i,s')\rho(\dt i, s')ds'\rl_\rho \\
    \leq&  e^{-\b\dt i}(t-\dt i) e^{\frac\beta2\dt i }\l(e^{\frac\beta2(\xi_1 - \dt i)}\ll\mL_{\mu,\Sig} r(s)\rl_\rho +  \b \ll r(s)\rl_\rho\r), \quad \xi_1\in[\dt i, \dt(i+1)].
    \end{aligned}
\]
Inserting it back to \eqref{ineq_11} yields, 
\[
\begin{aligned}
    &\ll V(s) - \tV(s)\rl_\rho \\
    \leq & \l(\ll\mL_{\mu,\Sig} r(s)\rl_\rho +  \b\ll r(s)\rl_\rho\r)  e^{\frac\beta2\dt}\sum_{i=0}^\infty \sqrt{\dt e^{-\b\dt i} \int_{\dt i }^{\dt (i+1)} (t-\dt i)^2  dt}\\
    = & \l(\ll\mL_{\mu,\Sig} r(s)\rl_\rho +  \b\ll r(s)\rl_\rho\r)e^{\frac\beta2\dt} \frac{1}{\sqrt{3}}\dt^2 \sum_{i=0}^\infty  e^{-\frac\b2\dt i} \\
    = & \frac{2}{\b}\l(\ll\mL_{\mu,\Sig} r(s)\rl_\rho +  \b\ll r(s)\rl_\rho\r)\dt + o(\dt),
\end{aligned}
\]
where the last equality comes from 
\[
e^{\frac\beta2\dt} \dt^2 \sum_{i=0}^\infty  e^{-\frac\b2\dt i}  = \frac{2\dt}{\beta} + o(\dt),
\]
which completes the proof. 
\end{proof}

\subsection{Proof of Theorem \ref{thm:v-vhat_stoch}}\label{sec:proof of them Phibe stoch}

First note that $V,\hV_i$ satisfies,
\[
\mL_{\mu,\Sig} V = \b V - r, \quad \mL_{\hmu_i,\hs_i} \hV_i = \b \hV_i - r.
\]
By the following Lemma, one can bound $\ll V - \hV_i \rl_\rho$ by the distance between $\mu, \Sig$ and $\hmu_i, \hs_i$. 
\begin{lemma}\label{lemma:hv-v-rho}
For $V,\hV$ satisfying
\[
\mL_{\mu,\Sig} V = \b V - r, \quad \mL_{\hmu,\hs} \hV = \b \hV - r,
\]
under Assumption \ref{ass_2}/(a), if $\ll \hmu - \mu \rl_\Linf\leq C_\mu, \ll \hs - \Sig \rl_\Linf \leq C_\Sig, \ll \nb\cdot (\hs - \Sig)\rl_\Linf \leq C_{\nb\cdot\Sig}$, $\ll\nb\log\rho\rl_\rho\leq L_\rho$ , and $C_\mu + \frac12C_\Sig\leq \sqrt{\frac{\b\lammin}{8}}, C_\Sig\leq \frac\lammin4$, one has
\[
\ll V - \hV \rl_\rho \leq \l[\frac{4C_\mu + 2C_{\nb\cdot\Sig}}{\b}\l(1+\frac{2C_\Sig}{\lammin}\r) + \frac{2C_\Sig}{\sqrt{\b\lammin}}\r]\ll \nb V \rl_\rho + \frac{2C_\Sig L_\rho}{\b} \ll \nb\hV\rl_\Linf.
\]
\end{lemma} 
(See Section \ref{proof of lemma hv-v-rho} for the proof of the above lemma)
Then we further apply the following lemma regarding the distance between $\mu, \Sig$ and $\hmu_i, \hs_i$.
\begin{lemma}\label{lemma:stoch_dynamics_order}
Under Assumption \ref{ass_2}, for $\hmu(s), \hs(s)$ defined in \eqref{def of higher order stoch mu}, one has 
\[
\ll \hmu_i (s)-  \mu (s)\rl_\Linf \leq L_\mu \dt^i, \quad \ll \hs_i (s)_{kl}-  \Sig(s)_{kl}\rl_\Linf \leq  L_\Sig \dt^i + o(\dt^i),
\] 
and
\begin{equation}
    \ll \nb\cdot (\hs - \Sig) \rl^2_\Linf \leq L_{\nb\cdot\Sig} \dt^i +o(\dt^i),
\end{equation}
where $ L_\mu, L_\Sig, L_{\nb\cdot\Sig}$ are constants depending on $\mu, \Sig, i$ defined in \eqref{def of Lmu}, \eqref{def of Lsig}, \eqref{def of LSigrho}, respectively.
\end{lemma}
(See Section \ref{proof of lemma stoch_dynamics_order} for the proof of the above lemma) Combine the above two lemmas, one can bound
\[
\ll V - \hV_i \rl\leq \l[ \l(\frac{3L_\mu + 2L_{\nb\cdot\Sig}}{\b}\l(1+\frac{L_\Sig\dt^i}{\lammin}\r) + \frac{L_\Sig}{\sqrt{\b\lammin}}\r)\ll \nb V \rl_\rho + \frac{2L_\Sig L_\rho}{\b}\ll \nb \hV \rl_\Linf \r]\dt^i + o(\dt^i)
\]
for 
\begin{equation}\label{def of D}
    \dt^i\leq D_{\mu, \Sig, \b}, \quad D_{\mu, \Sig, \b} = \min\l\{\frac{\lammin}{2L_\Sig}, \frac{\sqrt{2\b\lammin}}{2L_\mu + L_{\nb\cdot\Sig}} \r\},
\end{equation} 
with $L_\mu, L_\Sig, L_{\nb\cdot\Sig}$ defined in \eqref{def of Lmu}, \eqref{def of Lsig}, \eqref{def of LnbSigrho}. 
Furthermore,  by the following lemma on the upper bound for $\ll \nb V \rl_\rho, \ll \nb \hV \rl_\Linf$, one completes the proof for Theorem \ref{thm:v-vhat_stoch}.
\begin{lemma}\label{lemma:pt_s l2 est}
    Under Assumption \ref{ass_2}/(a), for $V(s)$ satisfying \eqref{stoch_PDE}, one has
    \[
    \ll \nb V(s) \rl_\rho \leq \frac2{\b}\l(\sqrt{\frac{C_{\nb\mu,\nb\Sig}}{\lammin}}\ll r \rl_\rho + \ll \nb r\rl_\rho\r)
    \]
    \[\ll \nb V(s) \rl_\Linf 
     \leq \frac1\b\l(\sqrt{\frac{2C_{\nb\mu,\nb\Sig}}{\lammin}}  \ll r \rl_\Linf + \ll \nb r \rl_\Linf \r) + o(\dt^{i/2})
     \]
    where $C_{\nb\mu,\nb\Sig}$ is a constant defined in \eqref{def of c123} that depends on $\nb\mu(s), \nb\Sig(s)$.
\end{lemma}
(See Section \ref{proof of lemma pt_s l2 est} for the proof of the above lemma)
one has,
\[
\begin{aligned}
    \ll V - \hV_i \rl_\rho\leq &\l[ \l(\frac{3L_\mu + 2L_{\nb\cdot\Sig}}{\b^2} + \frac{L_\Sig}{\b\sqrt{\b\lammin}}\r)\l(\sqrt{\frac{C_{\nb\mu,\nb\Sig}}{\lammin}}\ll r \rl_\rho + \ll \nb r\rl_\rho\r) \r.\\
    &\l.+ \frac{2L_\Sig L_\rho}{\b^2}\l(\sqrt{\frac{2C_{\nb\mu,\nb\Sig}}{\lammin}}  \ll r \rl_\Linf + \ll \nb r \rl_\Linf \r) \r]\dt^i + o(\dt^i)\\
    \leq & \l[\frac{C_{r,\mu,\Sig}}{\b^2} +  \frac{\h{C}_{r,\mu,\Sig}}{\b^{3/2}}\r]\dt^i + o(\dt^i) ,
\end{aligned}
\]
where 
\begin{equation}\label{def of Crmusig}
    \begin{aligned}
    C_{r,\mu,\Sig} = &\l(3L_\mu + 2L_{\nb\cdot\Sig}\r)\l(\sqrt{\frac{C_{\nb\mu,\nb\Sig}}{\lammin}}\ll r \rl_\rho + \ll \nb r\rl_\rho\r) \\
    &+ 2L_\Sig L_\rho\l(\sqrt{\frac{2C_{\nb\mu,\nb\Sig}}{\lammin}}  \ll r \rl_\Linf + \ll \nb r \rl_\Linf \r),\\
    \h{C}_{r,\mu,\Sig} =& \frac{L_\Sig}{\sqrt{\lammin}}\l(\sqrt{\frac{C_{\nb\mu,\nb\Sig}}{\lammin}}\ll r \rl_\rho + \ll \nb r\rl_\rho\r),
\end{aligned}
\end{equation}
with $L_\mu, L_\Sig, L_{\nb\cdot\Sig}, L_\rho, C_{\nb\mu, \nb\Sig}$ defined in \eqref{def of Lmu}, \eqref{def of Lsig}, \eqref{def of LnbSigrho}, \eqref{def of Lrho}, \eqref{def of c123}.

\subsection{Proof of Theorem \ref{thm:galerkin err} and Galerkin error with unbounded $\ll \nb\log\rho \rl_\Linf$}\label{proof of galerkin error}
    The $i$-th approximation $\hV_i(s)$ can be divided into two parts,
    \[
    \hV_i(s) = \hV_i^P(s) + e_i^P(s) \quad \text{with}\quad \hV_i^P(s) = \sum_{k=1}^p\hv_k\phi_k(s), e_i^P(s) = \hV_i(s) - \hV_i^P(s),
    \]
    where $\hV_i^P(s)$ could be any functions in the linear space spanned by $\Phi(s)$. 
    Note that $\hV_i(s)$ satisfies
    \[
    \la (\b-  \mL_{\hmu_i,\hs_i}) \hV_i(s), \Phi\ra_\rho = \la r(s), \Phi(s)\ra_\rho,
    \]
    which can be divided into two parts,
    \[
    \la (\b-  \mL_{\hmu_i,\hs_i}) \hV^P_i(s), \Phi\ra_\rho = \la r(s), \Phi(s)\ra_\rho -  \la (\b-  \mL_{\hmu_i,\hs_i}) e_i^P(s), \Phi\ra_\rho,
    \]
    subtract the above equation from the Galerkin equation \eqref{galerkin} gives
    \[
    \la (\b-  \mL_{\hmu_i,\hs_i})(\hV^G_i(s)  - \hV^P_i(s)), \Phi\ra_\rho =  \la (\b-  \mL_{\hmu_i,\hs_i}) e_i^P(s), \Phi\ra_\rho.
    \]
    Let $e^G_i (s)=  \hV^G_i(s)  - \hV^P_i(s) = \sum_{k=1}^pe_k\phi_k(s)$, then multiplying $(e_1, \cdots, e_p)$ to the above equation yields,
    \begin{equation}\label{ineq_16}
        \begin{aligned}
        \la (\b-  \mL_{\hmu_i,\hs_i})e^G_i, e^G_i\ra_\rho = & \la (\b-  \mL_{\hmu_i,\hs_i}) e_i^P, e^G_i\ra_\rho\\
        \la (\b-  \mL_{\mu,\Sig})e^G_i, e^G_i\ra_\rho =& \la \mL_{\hmu_i - \mu,\hs_i - \Sig} e^G_i, e^G_i\ra_\rho +\la (\b-  \mL_{\hmu_i,\hs_i}) e_i^P, e^G_i\ra_\rho.
        \end{aligned}
    \end{equation}
    
    When $L^\infty_\rho = \ll \nb \log \rho \rl_\Linf$ is bounded, by applying the last inequality of Proposition \ref{lemma:l2 operator} and Lemma \ref{lemma:stoch_dynamics_order}, one has
    \[
    \begin{aligned}
        \frac\b2\ll e^G_i \rl^2_\rho +\frac\lammin4 \ll \nb e^G_i \rl_\rho^2 \leq& c_1 \ll \nb e^G_i\rl^2_\rho + c_2 \ll e^G_i\rl_\rho\ll \nb e^G_i\rl_\rho + \b\ll e_i^P\rl_\rho\ll  e^G_i\rl_\rho \\
        &+ c_3\ll  \nb e_i^P\rl_\rho \ll\nb  e^G_i\rl_\rho + c_4\ll  \nb e_i^P \rl_\rho \ll e^G_i\rl_\rho,
    \end{aligned}
    \]
    where
    \begin{equation}\label{def of c34}
    \begin{aligned}
        &c_1 = \frac{L_\Sig}{2}\dt^i,\quad  c_2 = \l(L_\mu + \frac{L_{\nb\cdot\Sig}}{2} + \frac{L_\Sig L^\infty_\rho}{2} \r)\dt^i,\\
        &c_3 = \frac{1}{2}\l(\ll \Sig \rl_\Linf + L_\Sig\dt^i\r), \\
        &c_4 = \l(\ll \mu\rl_\Linf + \frac{\ll \nb\cdot \Sig \rl_\Linf }{2} + \frac{\ll \Sig \rl_\Linf L^\infty_\rho}{2} \r) + \l(L_\mu + \frac{L_{\nb\cdot\Sig}}{2} + \frac{L_\Sig L^\infty_\rho}{2} \r)\dt^i
    \end{aligned}
    \end{equation}
    with $L_\mu,L_\Sig, L_{\nb\cdot\Sig}$ defined in \eqref{def of Lmu}, \eqref{def of Lsig}, \eqref{def of LnbSigrho}.
    Under the assumption that $c_1\leq \frac{\lammin}{16}, \frac{c_2}{2}\leq\min\l\{\frac\lammin{16},\frac{\b}{4}\r\} $, i.e., 
    \begin{equation}\label{def of eta}
        \dt^i \leq \eta_{\mu,\Sig, \b}, \quad \eta_{\mu,\Sig, \b} = \frac{\min\l\{\frac{\lammin}{4} , \b\r\}}{2L_\mu + L_{\nb\cdot\Sig} + L_\Sig \max\{L^\infty_\rho, 4\}}
    \end{equation}
    with $L_\mu,L_\Sig, L_{\nb\cdot\Sig}$ defined in \eqref{def of Lmu}, \eqref{def of Lsig}, \eqref{def of LnbSigrho},    one has
    \[
    \begin{aligned}
        \frac\b4\ll e^G_i \rl^2_\rho +\frac\lammin8 \ll \nb e^G_i \rl_\rho^2 \leq&  \b\ll e_i^P \rl_\rho\ll  e^G_i\rl_\rho + c_3\ll  \nb e_i^P \rl_\rho \ll\nb  e^G_i\rl_\rho,\\
         &+ c_4\ll  \nb e_i^P \rl_\rho \ll e^G_i\rl_\rho\\
        \ll e^G_i \rl_\rho\leq & 4\ll e_i^P\rl_\rho + \sqrt{\frac{8}{\b}\l(\frac{2c_3^2}{\lammin}  +\frac{4c_4^2}{\b}\r)}\ll  \nb e_i^P \rl_\rho ,
        \end{aligned}
    \]
    which implies that
    \[
    \begin{aligned}
        \ll \hV_i - \hV_i^G \rl_\rho\leq\ll e_i^P \rl_\rho +\ll e_i^G \rl_\rho \leq & 5\ll e_i^P\rl_\rho + \sqrt{\frac{8}{\b}\l(\frac{2c_3^2}{\lammin}  +\frac{4c_4^2}{\b}\r)}\ll  \nb e_i^P \rl_\rho.
    \end{aligned}
    \]
    Since the above inequality holds for all $V_i^P$ in the linear space spanned by $\{\Phi\}$, therefore,
    \[
    \ll \hV_i - \hV_i^G \rl_\rho\leq  \frac{1}{\beta}C_G\inf_{V = \th^\top\Phi}\ll \hV_i - V \rl_{H^1_\rho}
    \]
    where
    \begin{equation}\label{def of c galerkin}
        C_G = \max\l\{5\beta, \frac{4c_3\sqrt{\beta}}{\sqrt{\lammin}} +6c_4\r\}, \quad \ll f \rl_{H^1_\rho} = \ll f \rl_\rho + \ll \nb f \rl_\rho.
    \end{equation}
with $c_3, c_4$ defined in \eqref{def of c34}.

\begin{theorem}[Galerkin Error with unbounded $\ll\nb\log\rho\rl_\Linf$]\label{thm:unbounded rho galerkin error}
The Galerkin solution \\$\hV^G_i(s) = \th^\top\Phi(s)$ satisfies
    \begin{equation}\label{galerkin-2}
        \la (\b - \mL_{\hmu_i,\hs_i}) \hV^G_i(s), \Phi\ra_\rho = \la r(s), \Phi(s)\ra_\rho.
    \end{equation}
    When $\ll \nb \log \rho \rl_\Linf$ is unbounded, assume that the bases $L^\infty_\Phi = \ll \Phi \rl_\Linf$ is bounded,  then as long as $\dt^i \leq \min\l\{\h{\eta}_{\mu,\Sig,\b}, D_{\mu,\Sig,\b}\r\}$, the Galerkin solution $\hV^G_i(s)$ approximates the solution to the $i$-th order PhiBE defined in \eqref{def of high order stoch}  with an error 
    \[
    \ll \hV^G_i(s)  - \hat{V}_i(s) \rl_\rho \leq \frac{\h{C}_G}\b\inf_{V^P = \th^\top\Phi}\ll \hV_i - V^P \rl_{H^1_{\rho,\infty}}. 
    \]
     where $ \h{\eta}_{\mu,\Sig,\b}, \h{C}_G, D_{\mu,\Sig,\b}$ are constants depending on $\mu, \Sig, \b, L^\infty_\rho, L^\infty_\Phi$ defined in \eqref{def of hateta}, \eqref{def of hatc galerkin}, \eqref{def of D} respectively, and  $\ll f \rl_{H^1_{\rho,\infty}} = \ll f \rl_{H^1_\rho} + \ll \nb f \rl_\Linf$. 
\end{theorem}
\begin{proof}
When $L^\infty_\rho = \ll \nb \log \rho \rl_\Linf$ is not bounded, by applying the last second inequality of Proposition \ref{lemma:l2 operator} and Lemma \ref{lemma:stoch_dynamics_order} to \eqref{ineq_16}, one has
\begin{equation}\label{ineq_17}
    \begin{aligned}
        &\frac\b2\ll e^G_i \rl^2_\rho +\frac\lammin4 \ll \nb e^G_i \rl_\rho^2 \leq c_1 \ll \nb e^G_i\rl^2_\rho + c_5 \ll e^G_i\rl_\rho\ll \nb e^G_i\rl_\rho + c_6 \ll e^G_i\rl_\Linf\ll \nb e^G_i\rl_\rho\\
        &+ \b\ll e_i^P\rl_\rho\ll  e^G_i\rl_\rho 
        + c_3\ll  \nb e_i^P\rl_\rho \ll\nb  e^G_i\rl_\rho + c_7\ll  \nb e_i^P \rl_\rho \ll e^G_i\rl_\rho+ c_8\ll  \nb e_i^P \rl_\Linf \ll e^G_i\rl_\rho,
    \end{aligned}
\end{equation}    
where
\begin{equation}\label{def of c78}
 \begin{aligned}
    &c_5 = \l(L_\mu + \frac{L_{\nb\cdot\Sig}}{2} \r)\dt^i, \quad c_8 = \frac{\ll \Sig \rl_\Linf L_\rho}{2}+ \frac{L_\Sig L_\rho}{2} \dt^i,\\
    &c_6 =  \frac{L_\Sig L_\rho}{2}\dt^i, \quad c_7 = \l(\ll \mu\rl_\Linf + \frac{\ll \nb\cdot \Sig \rl_\Linf }{2} \r) + \l(L_\mu + \frac{L_{\nb\cdot\Sig}}{2} \r)\dt^i,
\end{aligned}    
\end{equation}
with $L_\mu,L_\Sig, L_{\nb\cdot\Sig}, L_\rho$ defined in \eqref{def of Lmu}, \eqref{def of Lsig}, \eqref{def of LnbSigrho}, \eqref{def of Lrho}. 
Since 
\[
\ll e^G_i \rl_\Linf =  \ll e^\top \Phi  \rl_\Linf \leq \ll e \rl_2\ll \Phi \rl_\Linf \leq \frac{1}{\sqrt{\hat{\lam}}} \ll e^G_i \rl_\rho \ll \Phi \rl_\Linf,
\]
where $\hat{\lam}$ is the smallest eigenvalue of the matrix $B$, where $B_{ij} = \int \phi_i(s)\phi_i(s) \rho(s) ds$. Note that the matrix $B$ is always positive definite when $\{\phi_i(s)\}$ are linear independent bases w.r.t. the weighted $L^2$ norm. let $L^\infty_\Phi = \ll \Phi \rl_\Linf$, then \eqref{ineq_17} can be rewritten as
\[
\begin{aligned}
    \frac\b2\ll e^G_i \rl^2_\rho &+\frac\lammin4 \ll \nb e^G_i \rl_\rho^2 \leq c_1 \ll \nb e^G_i\rl^2_\rho + c_9 \ll e^G_i\rl_\rho\ll \nb e^G_i\rl_\rho+ \b\ll e_i^P\rl_\rho\ll  e^G_i\rl_\rho \\
    & + c_3\ll  \nb e_i^P\rl_\rho \ll\nb  e^G_i\rl_\rho + c_7\ll  \nb e_i^P \rl_\rho \ll e^G_i\rl_\rho+ c_8\ll  \nb e_i^P \rl_\Linf \ll e^G_i\rl_\rho,
\end{aligned}
\]
where
\[
    \begin{aligned}
    &c_9 = \l(L_\mu+ \frac{L_{\nb\cdot\Sig}}2 + \frac{L_\Sig L_\rho L^\infty_\Phi}{2\sqrt{\hat{\lam}}}\r) \dt^i.
\end{aligned}
\]
When $c_1 \leq \frac{\lammin}{16}, \frac{c_9}{2}\leq \min\{\frac{\lammin}{16}, \frac\b4\}$, i.e., 
    \begin{equation}\label{def of hateta}
        \dt^i \leq \h{\eta}_{\mu,\Sig, \b}, \quad \hat{\eta}_{\mu,\Sig, \b} = \frac{\min\l\{\frac{\lammin}{4} , \b\r\}}{2L_\mu + L_{\nb\cdot\Sig} + L_\Sig \max\l\{\frac{L_\rho L^\infty_\Phi}{\sqrt{\hat{\lam}}}, 4\r\}}, 
    \end{equation} 
    with $L_\mu,L_\Sig, L_{\nb\cdot\Sig}$ defined in \eqref{def of Lmu}, \eqref{def of Lsig}, \eqref{def of LnbSigrho}, 
then one has
\[
\begin{aligned}
    \frac{\b}{4}\ll e^G_i \rl^2_\rho +\frac\lammin8 \ll \nb e^G_i \rl_\rho^2 \leq  &\b\ll e_i^P\rl_\rho\ll  e^G_i\rl_\rho + c_3\ll  \nb e_i^P\rl_\rho \ll\nb  e^G_i\rl_\rho \\
    &  + c_7\ll  \nb e_i^P \rl_\rho \ll e^G_i\rl_\rho+ c_8\ll  \nb e_i^P \rl_\Linf \ll e^G_i\rl_\rho,\\
    %\frac{\b}{4}\ll e^G_i \rl^2_\rho \leq  3\b\ll e_i^P\rl^2_\rho +& \frac{c_3^2}{\lammin}\ll  \nb e_i^P\rl_\rho^2 + \frac{3c_7^2}\b\ll  \nb e_i^P \rl_\rho^2+ \frac{3c_8^2}{\b}\ll  \nb e_i^P \rl^2_\Linf,\\
    \ll e^G_i \rl_\rho \leq  7\ll e_i^P\rl_\rho +&  \sqrt{\frac{16c_3^2}{\b\lammin}+ \frac{48c_7^2}{\b^2}}\ll  \nb e_i^P \rl_\rho + \frac{7c_8}{\b}\ll  \nb e_i^P \rl_\Linf,
\end{aligned}
\]
which implies that
\[
\begin{aligned}
    \ll \hV_i - \hV_i^G \rl_\rho \leq & 8\ll e_i^P\rl_\rho +  \sqrt{\frac{16c_3^2}{\b\lammin}+ \frac{48c_7^2}{\b^2}}\ll  \nb e_i^P \rl_\rho + \frac{7c_8}{\b}\ll  \nb e_i^P \rl_\Linf.
\end{aligned}
\]
Since the above inequality holds for all $V_i^P$ in the linear space spanned by $\{\Phi\}$, therefore,
\[
\ll \hV_i - \hV_i^G \rl_\rho\leq  \frac{1}{\beta}\h{C}_G\inf_{V = \th^\top\Phi}\ll \hV_i - V \rl_{H^1_{\rho,\infty}},
\]
where
\begin{equation}\label{def of hatc galerkin}
    \h{C}_G = \max\l\{8\beta, \frac{4c_3\sqrt{\beta}}{\sqrt{\lammin}}+ 7c_7, 7c_8\r\}, \quad \ll f \rl_{H^1_{\rho,\infty}} = \ll f \rl_{H^1_\rho} + \ll \nb f \rl_\Linf.
\end{equation}
with $c_3, c_7, c_8$ defined in \eqref{def of c34}, \eqref{def of c78}.
\end{proof}

\subsection{Proof of Theorem \ref{thm:sample-complexity}}\label{proof of sample-complexity}

Note that in the proof, we refer the vector norm $\ll b \rl$ as the Euclidean norm, and the matrix norm $\ll A \rl$ as the operator norm. 

We first prove that the matrix $A$ defined in \eqref{1st model-free galerkin} is positive definite. By applying the first and last inequalities of Proposition \ref{lemma:l2 operator}, one has for $\forall \th\in\R^p$, 
\[
\begin{aligned}
    &\th^\top A \th = \la \beta V(s) - \mL_{\hat{\mu}_1, \hat{\Sig}_1} V(s), V(s) \ra_\rho, \quad \text{with }V = \Phi(s)^\top \th\\
    = &\beta \ll  V(s) \rl_\rho^2 - \la \mL_{\mu, \Sig} V(s), V(s) \ra_\rho - \la \mL_{\hat{\mu}_1 - \mu, \hat{\Sig}_1 - \Sig} V(s), V(s) \ra_\rho\\
    \geq& \beta \ll  V \rl_\rho^2 + \frac{\lammin}{2} \ll \nb V\rl_\rho^2 \\
    &- \l[\l(\ll \e_\mu \rl_\infty + \frac12\ll \nb\cdot \e_\Sig \rl_\infty + \frac12 \ll \e_\Sig \rl_\infty \ll \nb\log\rho \rl_\infty \r) \ll  V(s) \rl_\rho\ll \nb V \rl_\rho + \frac12\ll\e_\Sig \rl_\infty \ll \nb V \rl_\rho^2 \r]
\end{aligned}
\]
where $\e_\mu = \hat{\mu}_1 - \mu, \e_\Sig = \hat{\Sig}_1 - \Sig$. 
By Lemma \ref{lemma:stoch_dynamics_order}, one has
\[
\ll \e_\mu \rl_\infty + \frac12\ll \nb\cdot \e_\Sig \rl_\infty + \frac12 \ll \e_\Sig \rl_\infty \ll \nb\log\rho \rl_\infty \lesssim (L_\mu + L_\Sig L_\rho^\infty + L_{\nb\cdot\Sig} )\dt, \quad \ll\e_\Sig \rl_\infty \leq L_\Sig \dt.
\]
Under the assumption that $\dt \leq \eta_{\mu, \Sig, \beta}$ defined in \eqref{def of eta} is small enough, one can further bound
\[
\begin{aligned}
    &\th^\top A \th 
    \geq& \frac\beta2 \ll  V \rl_\rho^2 = \frac\beta2 \th^\top G_\Phi \th \geq \frac{\beta \lam_\Phi}2 \ll \th \rl^2, 
\end{aligned}
\]
where the last inequality comes from Assumption \eqref{ass on bases} on $\Phi$.

In addition, one also has boundedness on $\ll b \rl$,
\[
\ll b \rl = \ll \int r(s)\Phi(s) \rho(s) ds \rl \leq R L_\Phi, \quad R = \ll r \rl_\Linf.
\]

Next, one can divided the error $\ll \th_n - \th \rl$ by
\[
\begin{aligned}
    &\ll \th_n - \th \rl = \ll A_n^{-1}b_n - A^{-1} b \rl = \ll  A_n^{-1}(b_n - b) - (A_n^{-1} - A^{-1}) b  \rl \\
    \leq& \ll A_n^{-1} \rl \ll b_n - b \rl + \ll A_n^{-1} \rl \ll A_n - A \rl \ll A^{-1} \rl \ll b \rl.
\end{aligned}
\]
If one has, 
\[
\ll A_n - A \rl  \lesssim \min\l\{\frac{\beta \lam_\Phi}4, \frac{(\beta \lam_\Phi)^2\e}{16 RL_\Phi^2}\r\}, \quad \ll b_n - b \rl \leq \frac{\e\beta \lam_\Phi}{8L_\Phi}
\]
then by the positive definite of $A$ and boundedness of $b$, one has
\[
\ll V_n - V \rl \leq \ll \th_n - \th \rl L_\Phi \leq \l(\frac4{\beta \lam_\Phi}\ll b_n - b \rl  + \frac{8L_\Phi R}{(\beta \lam_\Phi)^2}\ll A_n - A \rl  \r) L_\Phi \leq \e
\]
Since we assume $\e$ small enough, so the goal becomes the sample complexity to bound 
\begin{equation}\label{target error}
   \ll b_n - b \rl \leq \frac{\e\beta \lam_\Phi}{8L_\Phi},\quad \ll A_n - A \rl  \lesssim \frac{(\beta \lam_\Phi)^2\e}{16RL_\Phi^2}
\end{equation}

Next, we bound the first term. Let $b_i = r(s_i)\Phi(s_i)$ with $s_i \sim \rho$. 
By the boundedness of $\|r\|_\infty$ and the almost-sure boundedness of $\Phi(s)$, 
the random vectors $b_i$ are uniformly bounded. 
Applying Bernstein’s inequality for sums of bounded random vectors 
\cite{vershynin2018high}, we obtain that for $n \geq \log(1/\delta)$, with probability at least $1-\delta$,
\[
\ll b_n - b \rl = \ll \frac{1}{n} \sum_{i=1}^n b_i  - E[b_i] \rl \lesssim L_\Phi R\sqrt{\frac{\log(1/\delta)}{n}}
\]
Therefore, the sample complexity to achieve the first error in \eqref{target error} is
\[
n \gtrsim \frac{\log(1/\delta)L_\Phi^2 R^2}{ \e^2\beta^2 \lam_\Phi^2}
\]

\iffalse
Now we bound the second term. Let $w_i = s_i' - s_i$. By the classical Fokker-Planck equation theorem \cite{}, one has the conditional distribution of $w_i$ given $s_i$ is
\[
\frac{c_1}{\dt^{d/2}}\exp\l(-\frac{|w|^2}{c_1\dt}\r) \leq \rho(\dt,w|s_i) \leq \frac{c_2}{\dt^{d/2}}\exp\l(-\frac{|w|^2}{c_2\dt}\r).
\]
In addition, when $\dt$ is small enough, one can further bound
\[
c_2\lesssim \sigma_{\max}^2 : = \sigma_{\max} = \ll \sigma \rl_\Linf.
\]
\fi

Now we bound the second term. Let $w_i = s_i' - s_i$.
Under Assumption~\ref{ass_2},
the transition density $\rho(\dt,w\mid s_i)$ admits two-sided Gaussian bounds of Aronson type \cite{aronson1967bounds}. Specifically, there exist positive constants
$c_1,c_2,c_3,c_4$, depending only on the dimension $d$, the ellipticity constants
$\lambda,\Lambda$ of $\sigma\sigma^\top$, and $\|b\|_\infty$, such that for all sufficiently
small $\dt$,
\[
c_1\,\dt^{-d/2}\exp\!\Big(-\frac{|w|^2}{c_2\,\dt}\Big)
\;\le\;
\rho(\dt,w\mid s_i)
\;\le\;
c_3\,\dt^{-d/2}\exp\!\Big(-\frac{|w|^2}{c_4\,\dt}\Big).
\]

Define the norm $\ll \cdot \rl_{\psi_1}$ for scalar random variable $X$ as
\[
\ll X \rl_{\psi_1} = \inf\l\{C>0: \E\l[e^{\frac{|X|}{C}}\r]\leq 2\r\},
\]
and for vector or matrix $X$, it refers to
\[
\ll X \rl_{\psi_1} = \sup_{\ll u \rl_2 = 1} \ll u^\top X \rl_{\psi_1}, \quad \ll X \rl_{\psi_1} = \sup_{\ll u \rl_2 = \ll v \rl_2 = 1} \ll u^\top X v \rl_{\psi_1} .
\]
then one has, 
\[
 \ll w_i \rl_{\psi_1|s_i} \lesssim \sigma_{\max}\sqrt{\dt}.
\]
It is well known that sub-Gaussian random variables are sub-exponential and that
\[
\|X\|_{\psi_1}\lesssim \|X\|_{\psi_2},
\]
where the implicit constant is universal \cite{vershynin2018high}.

Let $M_i(t) := \int_0^tu^\top \sigma(s_\tau) dB_\tau$ with some $\ll u\rl= 1$, %then $u^\top w_i = y_i(\dt)+ M_i(\dt)$.  First,\[y_i(\dt) := \ll \int_0^t u^\top \mu(s_\tau) d\tau \rl\]
Since the quadratic variant of $M_i(t)$ can be bounded by \cite{karatzas2014brownian} 
\[
\la M_i \ra_t = \int_0^t u^\top \sigma(s_\tau)\sigma(s_\tau)^\top u \, d\tau
\le \sigma_{\max}^2 t,
\]
$\lambda M_i(t)$ is a continuous local martingale with $M_i(0) = 0$ and $\E[\exp(\frac{\lam^2}2\la M \ra_t)] < \infty$ for any finite $t$. By Doléans--Dade exponential martingale theorem \cite[Chapter~3]{karatzas1998brownian}
\[
Z_t := \exp\!\l( \lambda M_t - \frac{\lambda^2}{2} \langle M \rangle_t \r)
\]
is a martingale, and therefore
\[
\mathbb{E}[Z_t \mid s_i(0) = s_i] = Z_0 = 1.
\]
It follows that
\[
\mathbb{E}\!\left[ \exp(\lambda M_\dt) \mid s_i(0) = s_i\right]
\le \exp\!\left( \frac{\lambda^2}{2} \langle M \rangle_\dt \right)
\le \exp\!\left( \frac{\lambda^2}{2} \sigma_{\max}^2 \dt \right).
\]
This implies \cite{vershynin2018high}
\begin{equation}\label{bound on M}
    \| M_i(\dt) \|_{\psi_2 \mid s_i}
\le C \sigma_{\max}\sqrt{\dt},
\end{equation}
for universal constants $C$. In addition, because the upper bound is independent of $s_i$, and the relation between $\ll \cdot \rl_{\psi_1}$ and $\ll \cdot \rl_{\psi_2}$  \cite{vershynin2018high}, one has,
\begin{equation}\label{bound on M1}
\| M_i(\dt) \|_{\psi_1}
\lesssim \sigma_{\max}\sqrt{\dt}.
\end{equation}
Let $Y_i = \int_0^\dt \mu(s_i(t))dt$, then $|Y_i|\leq \ll \mu \rl_\Linf\dt$. Because of \eqref{bound on M1} and $u^\top w_i  = M_i(\dt) + u^\top Y_i$, 
one has
\[
\| w_i - \E[ w_i ] \|_{\psi_1 } \leq 2\| w_i\|_{\psi_1 } \lesssim \sigma_{\max}\sqrt{\dt}  + \ll \mu \rl_\Linf \dt.
\]
Therefore, 
\[
 \ll \frac{w_i}{\dt}\cdot \nb \Phi(s_i)  - \E\l[ \frac{w_i}{\dt}\cdot \nb \Phi(s_i)  \r] \rl_{\psi_1} \lesssim L_\Phi\l(\frac{\sigma_{\max} }{\sqrt{\dt}} + \ll \mu \rl_\Linf\r).
\]
%Since the uppper bound is independent of $s_i$, %and for $\dt$ small enough, we only focus on the leading order term, one has \[  \ll \frac{w_i}{\dt}\cdot \nb \Phi(s_i)  \rl_{\psi_1} \lesssim \frac{\sigma_{\max} L_\Phi}{\sqrt{\dt}}. \]
By \eqref{bound on M} and the boundedness of $\ll \nb^2 \Phi(s)\rl \leq L_\Phi$, one has $M_i^\top \nb^2\Phi(s_i) M_i$ is sub-exponential, therefore, one has
\[
\begin{aligned}
&\ll w_i^\top \nb^2\Phi(s_i) w_i\rl_{\psi_1|s_i} \leq \ll M_i^\top\nb^2\Phi(s_i) M_i \rl_{\psi_1|s_i} +  2\ll M_i^\top\nb^2\Phi(s_i) Y_i \rl_{\psi_1|s_i} + \ll Y_i^\top\nb^2\Phi(s_i) Y_i \rl_{\psi_1|s_i}\\
\lesssim & L_\Phi\l(\sigma_{\max}^2\dt  + \ll \mu \rl_\Linf \sigma_{\max}\dt\sqrt{\dt} +  \ll \mu \rl_\Linf^2\dt^2\r)
\end{aligned}
\]
Hence, one has
\[
\begin{aligned}
&\ll \frac1\dt w_i^\top \nb^2\Phi(s_i) w_i - \E\l[\frac1\dt w_i^\top \nb^2\Phi(s_i) w_i\r] \rl_{\psi_1} \leq 2\ll \frac1\dt w_i^\top \nb^2\Phi(s_i) w_i \rl_{\psi_1} \\
\lesssim & L_\Phi (\sigma_{\max}^2 + \sigma_{\max} \ll \mu \rl_\Linf \sqrt{\dt} + \ll \mu \rl_\Linf^2\dt ).
\end{aligned}
\]
Therefore, for $A_i = \Phi(s_i)\l(\beta \Phi(s_i) - \frac{w_i}\dt\cdot\nb\Phi(s_i) - \frac1{2\dt} w_i^\top \nb^2\Phi(s_i) w_i \r)^\top$
\[
\ll A_i - \E[A_i]\rl_{\psi_1} \lesssim L_\Phi^2 \l(\beta +  \frac{\sigma_{\max} }{\sqrt{\dt}} + \ll \mu \rl_\Linf + \sigma_{\max}^2 + \sigma_{\max} \ll \mu \rl_\Linf \sqrt{\dt} + \ll \mu \rl_\Linf^2\dt \r).
\]
Finally, by applying Bernstein inequality \cite{tropp2012user}, one obtains that with probability $1- \delta$, one has
\[
\ll A_n - A \rl  \lesssim  \ll A_i - \E[A_i]\rl_{\psi_1}\l(\sqrt{\frac{\log(2p/\delta)}{n}} + \frac{\log(2p/\delta)}{n}\r) %\lesssim  L_\Phi^2\sigma_{\max} \sqrt{\frac{\log(2p/\delta)}{n\dt}} 
\]
In addition, for $n \geq\log(2p/\delta)$ large enough, and $\dt$ small enough, one can further simplify the above inequality
\[
\ll A_n - A \rl  \lesssim  \frac{L_\Phi^2\sigma_{\max}}{\sqrt{\dt}}\sqrt{\frac{\log(2p/\delta)}{n}} 
\]
%where the last inequality is obtained by the upper bound for $\dt$. 
Therefore, the sample complexity to achieve the error in \eqref{target error} is
\[
n \gtrsim \frac{\log(2p/\delta)R^2L_\Phi^8\sigma_{\max}^2}{ \beta^4 \e^2\dt\lam_\Phi^4}.
\]
To summarize, with probability $1-\delta$, if 
\[
n \gtrsim \max\l\{\frac{\log(2p/\delta)R^2L_\Phi^8\sigma_{\max}^2}{ \beta^4 \e^2\dt\lam_\Phi^4}, \frac{\log(1/\delta)L_\Phi^2 R^2}{ \e^2\beta^2 \lam_\Phi^2}\r\}
\]
one has
\[
\ll V_n - \hat{V}_1^G \rl_\rho \leq \e
\]
This gives the first inequality.

On the other hand, if 
\[
n \gtrsim n_1 = \frac{16L_\Phi^4 \log(2p/\delta)\l( \frac{\sigma_{\max} }{\sqrt{\dt}}  \r)^2 }{\beta^2\lam_\Phi^2}  \approx \frac{L_\Phi^4 \log(2p/\delta)\sigma_{\max}^2\lam_\Phi^{-2}}{\beta^2\dt}  
\]
then w.p. $1-\delta$, one has
\[
\ll A_n - A\rl \leq \frac{\beta\lam_\Phi}{4}.
\]
Therefore, for $\forall n \geq n_1$, w.p. $1-\delta$, one has,
\[
\begin{aligned}
&\ll V_n - V \rl \leq \l(\frac4{\beta \lam_\Phi}\ll b_n - b \rl  + \frac{8L_\Phi R}{(\beta \lam_\Phi)^2}\ll A_n - A \rl  \r) L_\Phi \\
\lesssim& \l(\frac4{\beta \lam_\Phi}L_\Phi R\sqrt{\frac{\log(p/\delta)}{n}}  + \frac{8L_\Phi R}{(\beta \lam_\Phi)^2}L_\Phi^2\sigma_{\max}  \sqrt{\frac{\log(2p/\delta)}{n\dt}}  \r) L_\Phi\\
\lesssim &\l(L_\Phi^4 R\sigma_{\max}\lam_\Phi^{-2}\sqrt{\log(2p/\delta)} \r)\frac{1}{\beta^2\sqrt{{n\dt}}}   
\end{aligned}
\]

\section{Conclusion}
In this paper, we introduce PhiBE, a PDE-based Bellman equation that integrates discrete-time information into a continuous-time PDE. PhiBE outperforms the classical Bellman equation in approximating the continuous-time PE problem, particularly in scenarios where underlying dynamics evolve slowly. Importantly, the approximation error of PhiBE depends on the dynamics, making it more robust against changes in reward structures. This property allows greater flexibility in designing reward functions to effectively achieve RL objectives.
We further proposed higher-order PhiBE, which yields improved approximations of the true value function and achieves the same accuracy with sparser data, thereby improving sample efficiency. In addition, we developed a model-free algorithm for solving PhiBE under linear function approximation and established theoretical guarantees that sharpen classical RL convergence results, particularly in their dependence on the discretization step size~$\dt$.

This work serves as the first step toward a systematic PDE-based approach to continuous-time reinforcement learning. We expect the PhiBE framework to extend naturally to more general RL settings, including richer function classes and control problems beyond policy evaluation.

\appendix

\section*{Appendix}
\section{Derivation of the HJB equation under deterministic dynamics}\label{appendix:derivation}
By the definition of $V(s)$ in \eqref{def of value}, one has
\[
V(s_t) = \int_t^\infty e^{-\b (\t{t} - t)} r(s_{\t{t}}) d\t{t},
\]
which implies that,
\[
\begin{aligned}
    \frac{d}{dt}V(s_t) =& \b\int_t^\infty e^{-\b (\t{t} - t)} r(s_{\t{t}}) d\t{t} - r(s_t). 
\end{aligned}
\]
Using the chain rule on the LHS of the above equation yields $\frac{d}{dt}V(s_t) = \frac{d}{dt}s_t\cdot \nb V(s_t)$, and the RHS can be written as $\b V(s_t) - r(s_t)$, resulting in a PDE for the true value function 
\begin{equation}\label{deter_Vst}
\b V(s_t) = r(s_t) + \frac{d}{dt}s_t\cdot \nb V(s_t).
\end{equation}
or equivalently, 
\begin{equation}\label{deter_V2}
\b V(s) = r(s) + \mu(s)\cdot \nb V(s).
\end{equation}

\color{black}
\section{Proof of Lemmas}
\subsection{Proof of Lemma \ref{lemma:Lhat - L}}\label{proof of lemma Lhat - L}
By Feynman–Kac theorem, it is equivalently to write $\hV$ as, 
\[
\hV = \int_0^\infty e^{-\b t} r(\hst) dt \quad \text{with}\quad \frac{d}{dt}\hst = \hmu(s_t).
\] 
Hence, 
\begin{equation}\label{diff_deter}
\begin{aligned}
\lv V(s) - \hV(s)\rv  =& \lv \int_0^\infty e^{-\b t} (r(s_t) - r(\hst) ) dt \rv = \lv \int_0^\infty e^{-\b t} \l(\int_{s_t}^{\h{s}_t}\nb r(s) ds\r)dt \rv\\
\leq &\ll \nb r \rl_\Linf  \int_0^\infty e^{-\b t}  \lv \hst - \st \rv dt,
\end{aligned}
\end{equation}
where 
\begin{equation}\label{deter_dynam}
\frac{d}{dt}s_t = \mu(s_t), \quad \frac{d}{dt}\hst = \hmu(\hst), \quad s_0 = \h{s}_0 = s.
\end{equation}
Subtracting the two equations in \eqref{deter_dynam} and multiplying it with  $(\hst - \st)^\top$ gives
\[
\begin{aligned}
\frac12\frac{d}{dt} \ll\h{s}_t-s_t\rl^2  =& (\hmu (\h{s}_t)) - \mu (s_t) )(\h{s}_t-s_t) \\
=& \l((\hmu (\hst)-  \mu (\hst) + (\mu (\hst) - \mu (s_t)) \r)(\hst - \st)\\
\leq& \ll \mu - \hmu\rl_\Linf \ll \hst - \st\rl + \ll \nb\mu(s) \rl_\Linf \ll \hst - \st\rl^2\\
\leq & \frac{1}{2\e}\ll \mu - \hmu\rl^2_\Linf + \l(\frac{\e}{2}+\ll \nb\mu(s) \rl_\Linf \r)\ll \hst - \st\rl^2, \quad \text{for }\forall \e>0,
\end{aligned}
\]
where the mean value theorem is used in the first inequality. This implies
\[
\begin{aligned}
    &\ll \hst - \st \rl_2 \leq \frac{1}{\sqrt{\e}}\ll \mu - \hmu\rl_\Linf t^{1/2} e^{(\e/2+\ll\nb\mu\rl_\Linf)t}.
\end{aligned}
\]
Inserting the above inequality back to \eqref{diff_deter} gives
\[
\begin{aligned}
\ll V(s) - \hV(s) \rl_\Linf \leq&\frac{1}{\sqrt{\e}} \ll \nb r \rl_\Linf \ll \mu - \hmu\rl_\Linf  \int_0^\infty e^{-(\b - \e/2 -\ll \nb \mu \rl_\Linf) t}  t^{1/2} dt\\
 = &\frac{\sqrt{\pi}\ll \mu - \hmu\rl_\Linf \ll \nb r \rl_\Linf }{2\sqrt{\e}(\b - \e/2 - \ll \nb \mu(s) \rl_\Linf)^{3/2} }.
\end{aligned}
\]
Assigning $\e = \frac{1}{2}(\b - \ll \nb \mu(s) \rl_\Linf)$ to the above inequality completes the proof.

\subsection{Proof of Lemma \ref{lemma:muhat - mu}}\label{proof of lemma muhat - mu}
By Taylor expansion, one has
\[
s_{j\dt} = \sum_{k=0}^i\frac{(j\dt )^k}{k!}\l(\l.\frac{d^k}{dt^k}s_t\r|_{t=0}\r) + \frac{(j\dt )^{i+1}}{(i+1)!}\l(\l.\frac{d^{i+1}}{dt^{i+1}}s_t\r|_{t=\xi_j}\r)
\]
with $\xi_j \in (0, j\dt)$. Inserting it into $\hmu_i(s)$ gives, 
\[
\begin{aligned}
\hmu_i(s) =& \frac1\dt\sum_{j=0}^i\coef{i}_j[s_{j\dt}|s_0 = s] \\
=& \frac1\dt\sum_{j=0}^i\coef{i}_j \l[ \sum_{k=0}^i \l(\l.\frac{d^k}{dt^k}s_t\r|_{t=0}\r)\frac{(\dt j)^k}{k!} +  \l(\l.\frac{d^{i+1}}{dt^{i+1}}s_t\r|_{t=\xi_j}\r)\frac{(\dt j)^{i+1}}{(i+1)!}   \r]\\
= & \frac1\dt\sum_{k=0}^i \l(\l.\frac{d^k}{dt^k}s_t\r|_{t=0}\r) \frac{(\dt)^k}{k!} \sum_{j=0}^i\coef{i}_jj^k +  \frac1\dt\sum_{j=0}^i\coef{i}_j \l(\l.\frac{d^{i+1}}{dt^{i+1}}s_t\r|_{t=\xi_j}\r)\frac{(\dt j)^{i+1}}{(i+1)!}  \\
=&\l(\l.\frac{d}{dt}s_t\r|_{t=0}\r)  + \frac{ \dt^{i}}{(i+1)!} \sum_{j=0}^i\coef{i}_j j^{i+1}\l(\l.\frac{d^{i+1}}{dt^{i+1}}s_t\r|_{t=\xi_j}\r) ,
\end{aligned}
\]
where the last equality is due to the definition of $\coef{i}$ in \eqref{def of a}. Since \[\l. \frac{d}{dt}s_t\r|_{t=0} = \mu(s_0) = \mu(s),\] one has 
\[
\lv \hmu_i(s) - \mu(s) \rv = \frac{\dt^i}{(i+1)!} \lv \sum_{j=0}^i\coef{i}_j j^{i+1} \l(\l.\frac{d^{i+1}}{dt^{i+1}}s_t\r|_{t=\xi_j}\r)\rv  . 
\]
Since
\[
\begin{aligned}
 \frac{d^{i+1}}{dt^{i+1}}s_t =&  \frac{d^i}{dt^{i}}(\mu(s_t)) = \mL_{\mu}^i \mu(s_t),
\end{aligned}
\]
then as long as  $\ll \mL_{\mu}^i \mu(s_t) \rl_\Linf$ is bounded, one has
\[
\ll \hmu_i(s) - \mu(s) \rl_\Linf \leq\frac{\ll \mL_{\mu}^i \mu(s) \rl_\Linf }{(i+1)!} \sum_{j=0}^i|\coef{i}_j|  j^{i+1} \dt^i = C_i\ll \mL_{\mu}^i \mu(s) \rl_\Linf\dt^i .
\]

\subsection{Proof of Lemma \ref{lemma:hv-v-rho}}\label{proof of lemma hv-v-rho}
Substracting the second equation from the first one and let $e(s) = V(s) -  \hV(s)$ gives,
\[
\begin{aligned}
    &\b e = \mL_{\mu,\Sig}e +  (\mL_{\mu,\Sig} - \mL_{\hmu,\hs})\hV.
\end{aligned}
\]
Multiply the above equation with $e(s)\rho(s)$ and integrate it over $s\in\S$, one has, 
\begin{equation}\label{ineq_7}
\begin{aligned}
    &\frac\b2\ll e \rl_\rho^2 = \la \mL_{\mu,\Sig} e, e\ra_\rho + \la  \mL_{\hmu - \mu,\hs - \Sig} \hV, e\ra_\rho\\
\leq& -\frac{\lammin}{4} \ll \nb e \rl_\rho^2 + \l(C_\mu+ \frac12C_{\nb\cdot\Sig} \r)\ll e\rl_\rho\ll \nb \hV \rl_\rho + \frac12C_\Sig\ll \nb e\rl_\rho\ll \nb \hV \rl_\rho\\
&+ \frac12C_\Sig L_\rho\ll e\rl_\rho\ll \nb \hV \rl_\Linf\\
\leq& -\l(\frac{\lammin}{4} -c_2\r)\ll \nb e \rl_\rho^2 + \l( c_1\ll e \rl_\rho + c_2\ll \nb V \rl_\rho \r)\ll \nb e\rl_\rho + \l(c_1\ll \nb V\rl_\rho + c_3\ll \nb \hV \rl_\Linf \r)\ll e\rl_\rho ,
\end{aligned}
\end{equation}
where the first and third equations in Proposition \ref{lemma:l2 operator} are used for the first inequality, $\ll \nb \hV \rl_\rho \leq \ll \nb V \rl_\rho + \ll \nb e \rl_\rho$ are used for the second inequality, and $c_1 = C_\mu+ \frac12C_{\nb\cdot\Sig} $, $c_2 = \frac12C_\Sig$, $c_3 = \frac{1}{2}C_\Sig L_\rho$. 
Under the assumption that $c_2\leq \frac\lammin4$, one has
\begin{equation*}
\begin{aligned}
\frac\b2\ll e \rl_\rho^2 \leq  & -\frac\lammin8\l(\ll \nb e\rl_\rho - \frac{4}{\lammin}(c_1\ll e \rl_\rho + c_2\ll \nb V \rl_\rho) \r)^2 \\
&+ \frac{2}{\lammin}(c_1\ll e \rl_\rho + c_2\ll \nb V \rl_\rho)^2 + \l(c_1\ll \nb V\rl_\rho + c_3\ll \nb \hV \rl_\Linf \r)\ll e\rl_\rho \\
\leq & \frac{2c_1^2}{\lammin}\ll e\rl_\rho^2 + \l[\l(\frac{4c_1c_2}{\lammin} + c_1\r)\ll \nb V \rl_\rho+ c_3\ll \nb \hV \rl_\Linf\r]\ll e\rl_\rho + \frac{2c_2^2}{\lammin}\ll \nb V\rl_\rho^2.
\end{aligned}
\end{equation*}
Under the assumption that  $2c_1^2\leq \frac18\b\lammin$, one has
\[
\begin{aligned}
   \frac{\b}{4}\ll e \rl^2_\rho \leq \l[\frac4\b \l(\frac{4c_1c_2}{\lammin} + c_1\r)^2+ \frac{2c_2^2}{\lammin} \r]\ll \nb V \rl_\rho^2 + \frac{4c_3^2}\b \ll \nb \hV \rl_\Linf^2  + \frac{\b}{8}\ll e \rl^2_\rho,
\end{aligned}
\]
which yields, 
\[
\begin{aligned}
   \ll e \rl_\rho \leq \l[\frac{4c_1}{\b} \l(\frac{4c_2}{\lammin} + 1\r)+ \frac{3c_2}{\sqrt{\b\lammin}} \r]\ll \nb V \rl_\rho + \frac{4c_3}{\b}\ll \nb\hV \rl_\Linf.
\end{aligned}
\]

\subsection{Proof of Lemma \ref{lemma:stoch_dynamics_order}}\label{proof of lemma stoch_dynamics_order}
The proof of Lemma \ref{lemma:stoch_dynamics_order} replies the following two lemmas, which we will prove later. 
\begin{lemma}\label{i-th operator}
    Define operator $\Pi_{i,\dt} f(s) = \frac1\dt \E[\sum_{j=1}^i\coef{i}_jf(s_{j\dt} - s_0)|s_0 = s]$ with $\coef{i}_j$ defined in \eqref{def of a} and $f(0) = 0$, then 
    \[
    \Pi_{i,\dt}f(s) = [\mL_{\mu,\Sig}f](0) + \frac{1}{\dt i!}\sum_{j=1}^i\coef{i}_j\int_0^{j\dt}\E[\mL^{i+1}_{\mu,\Sig}f(s_t - s_0)|s_0 = s] t^i dt.
    \]
\end{lemma}
\begin{lemma}\label{lemma:plinf}
    For $p(s,t) = \E[f(s_t)|s_0 = s]$ with $s_t$ driven by the SDE \eqref{def of dynamics}, then under Assumption \ref{ass_2}/(a), one has 
    \[
    \ll \nb p(s, t)\rl_\Linf \leq \sqrt{\frac{C_{\nb\mu, \nb\Sig}}{\lammin}}\ll f \rl_\Linf + \ll \nb f \rl_\Linf, \quad \text{with }C_{\nb\mu, \nb\Sig} \text{ defined in \eqref{def of c123}}.
    \]
    For $p(s,t) = \E[f(s_t)(s_t - s_0)|s_0 = s]$ with $s_t$ driven by the SDE \eqref{def of dynamics}, then under Assumption \ref{ass_2}/(a), one has 
    \[
    \ll p(s,t) \rl_\Linf \leq  C_1(\ll \mu \rl_{C^1},\ll \Sigma \rl_{C^1}, \ll f \rl_{C^1})\sqrt{e^t-1}.
    \]
    \[
    \begin{aligned}
            \ll \nb p(s, t)\rl_\Linf \leq& C_2(\ll \mu \rl_{C^2},\ll \Sigma \rl_{C^2}, \ll f \rl_{C^2})\sqrt{e^t-1} .
    \end{aligned}
    \]
    where $C_1(\ll \mu \rl_{C^1},\ll \Sigma \rl_{C^1}, \ll f \rl_{C^1})$ and $C_2(\ll \mu \rl_{C^2},\ll \Sigma \rl_{C^2}, \ll f \rl_{C^2})$ are defined in \eqref{def of C1} and \eqref{def of C2}.
\end{lemma}

Now we are ready to prove Lemma \ref{lemma:stoch_dynamics_order}.
By Lemma \ref{i-th operator}, one has
\begin{equation}\label{ineq_55}
    \hmu_i(s) = \mu(s) + \frac1{\dt i!}\sum_{j=1}^i\coef{i}_j\l( \int_0^{j\dt}\E[\mL^{i}_{\mu,\Sig}\mu(s_t)|s_0 = s]t^idt  \r),
\end{equation}
which implies that 
\[
\begin{aligned}
    &\ll \hmu_i(s) - \mu(s) \rl_\Linf 
    \leq& \frac{ \ll\mL_{\mu,\Sig}^{i}\mu(s)\rl_\Linf}{\dt i!}\sum_{j=1}^i|\coef{i}_j|\int_0^{j\dt} t^idt
    \leq L_\mu \dt^i,
\end{aligned}
\]
where 
\begin{equation}\label{def of Lmu}
    L_\mu = \h{C}_i\ll\mL_{\mu,\Sig}^{i}\mu(s)\rl_\Linf \quad \text{with } \h{C}_i = \sum_{j=1}^i\frac{|\coef{i}_j|j^{i+1}}{(i+1)!}\text{ defined in \eqref{def of C_i}}.
\end{equation}

To prove the second inequality in the lemma, first apply Lemma \ref{i-th operator}, one has
\begin{equation}\label{ineq_4}
    \begin{aligned}
    \hs_i(s) 
    = & \Sig(s) \\
    &+ \frac{1}{\dt i!}\sum_{j=1}^i\coef{i}_j\int_0^{j\dt}\E[\mL_{\mu,\Sig}^i\l(\mu(s_t) (s_t-s_0)^\top + (s_t - s_0)\mu^\top(s_t) + \Sig(s_t)\r)|s_0 = s]t^idt.
\end{aligned}
\end{equation}
Note that 
\begin{equation}\label{def of h f g}
    \begin{aligned}
 h(s_t): =& \mL_{\mu,\Sig}^i\l(\mu(s_t) (s_t-s_0)^\top + (s_t - s_0)\mu^\top(s_t) + \Sig(s_t)\r) \\
 &-\l[ \mL_{\mu,\Sig}^{i}\mu(s_t) (s_t-s_0)^\top + (s_t - s_0)(\mL_{\mu,\Sig}^{i}\mu(s_t))^\top\r]
\end{aligned}
\end{equation}
is a function that only depends on the derivative $\nb^j\Sig, \nb^j\mu$ up to $2i$-th order, which can be bounded under Assumption \ref{ass_2}/(b). 
Thus applying the second inequality of Lemma \ref{lemma:plinf} yields
\[
\begin{aligned}
&\ll \E[\mL_{\mu,\Sig}^i\l(\mu(s_t) (s_t-s_0)^\top + (s_t - s_0)\mu^\top(s_t) + \Sig(s_t)\r)|s_0 = s] \rl_\Linf\\
\leq& \ll h(s)\rl_\Linf + 2C_1 e^{t/2} ,
\end{aligned}
\]
where $C_1 = C_1(\ll \mu \rl_{C^1},\ll \Sigma \rl_{C^1}, \ll \mL_{\mu,\Sig}^{i}\mu \rl_{C^1})$ defined in \eqref{def of C1}.
Hence, one has,
\[
\begin{aligned}
&\ll \int_0^{j\dt}\E[\mL_{\mu,\Sig}^i\l(\mu(s_t) (s_t-s_0)^\top + (s_t - s_0)\mu^\top(s_t) + \Sig(s_t)\r)|s_0 = s]t^idt \rl_\Linf \\
\leq& \frac{1}{i+1}\ll h \rl_\Linf(j\dt)^{i+1} + 2C_1 \int_{0}^{j\dt} (2 + t)t^i dt, \quad \text{for }j\dt \leq 3\\
=&\frac{1}{i+1}\l(\ll h \rl_\Linf + 4C_1 \r)(j\dt)^{i+1} + 2C_1 \frac{1}{i+2} (j\dt)^{i+2},
\end{aligned}
\]
where $e^{t/2} \leq 2+t$ for $t\leq 3$ are used in the first inequality.
Plugging the above inequality back to \eqref{ineq_4} implies
\[
\begin{aligned}
    &\ll\hs_i(s) - \Sig(s)\rl_\Linf
    \leq L_\Sig \dt^{i} + o(\dt^i),
\end{aligned}
\]
where 
\begin{equation}\label{def of Lsig}
    L_\Sig =\h{C}_i(\ll h(s)\rl_\Linf+4C_1(\ll \mu \rl_{C^1},\ll \Sigma \rl_{C^1}, \ll\mL_{\mu,\Sig}^{i}\mu \rl_{C^1})) 
\end{equation}
with $\h{C}_i, h(s), C_1(p,q,g)$ defined in \eqref{def of Lmu}, \eqref{def of h f g}, \eqref{def of C1}.

To prove the third inequality, one first takes $\nb\cdot$ to \eqref{ineq_4}, 
\[
\begin{aligned}
    &\nb\cdot \hs_i(s)  = \nb\cdot \Sig(s) + \frac{1}{\dt i!}\sum_{j=1}^i\coef{i}_j \int_0^{j\dt}\nb\cdot p(s,t) t^idt,
\end{aligned}
\]
where \[p(s,t) = \E[h(s_t)|s_0 = s] + \E[\mL_{\mu,\Sig}^{i}\mu(s_t) (s_t-s_0)^\top + (s_t - s_0)(\mL_{\mu,\Sig}^{i}\mu(s_t))^\top |s_0 = s],\]
with $h$ defined in \eqref{def of h f g}.
Therefore, by the first and third inequalities in Lemma \ref{lemma:plinf}, and denoting $c_1 = \sqrt{C_{\nb\mu, \nb\Sig}/\lammin} \ll h \rl_\Linf + \ll \nb\cdot h \rl_\Linf, c_2 = C_2(\ll \mu \rl_{C^2},\ll \Sigma \rl_{C^2}, \ll \mL_{\mu,\Sig}^{i}\mu \rl_{C^2})$, one has 
\begin{equation}\label{eq:pts sig}
    \begin{aligned}
&\ll  \nb\cdot \hs_i(s)  - \nb\cdot \Sig(s)  \rl_\Linf
\leq \frac{1}{\dt i!}\sum_{j=1}^i|\coef{i}_j| \l(\int_0^{j\dt} c_1t^i + 2c_2e^{t/2} t^i dt\r)\\
\leq &  \frac{1}{\dt i!}\sum_{j=1}^i|\coef{i}_j| \l(\int_0^{j\dt} (c_1 + 4c_2) t^i + 2c_2t^{i+1} dt\r) = L_{\nb\cdot\Sig} \dt^i  +o(\dt^i),
\end{aligned}
\end{equation}
where 
\begin{equation}\label{def of LSigrho}
\begin{aligned}
    L_{\nb\cdot\Sig} =&\h{C}_i \l[\sqrt{C_{\nb\mu, \nb\Sig}/\lammin} \ll h \rl_\Linf + \ll \nb\cdot h \rl_\Linf\r. \\
    &\l.+ 4C_2(\ll \mu \rl_{C^2},\ll \Sigma \rl_{C^2}, \ll \mL_{\mu,\Sig}^{i}\mu \rl_{C^2}) \r],  
\end{aligned}
\end{equation}
with $\h{C}_i, h(s), C_{\nb\mu, \nb\Sig}, C_2(p,q,g)$ defined in \eqref{def of Lmu}, \eqref{def of h f g}, \eqref{def of c123}, \eqref{def of C2}.

\paragraph{Proof of Lemma \ref{i-th operator}}
\begin{proof}
First note that
\begin{equation}\label{ineq_5}
    \Pi_{i,\dt} f(s) = \frac1\dt\sum_{j=1}^i\coef{i}_j\int_\S f(s' - s) \rho(s',j\dt|s) ds',
\end{equation}
where $\rho(s',t|s)$ is defined in \eqref{def of pt_t rho}.
By Taylor's expansion, one has
\[
\rho(s',j\dt|s) = \sum_{k=0}^i\pt^k_t\rho(s',0|s)\frac{(j\dt)^k}{k!} + \frac{1}{i!}\int_0^{j\dt}\pt_t^{i+1}\rho(s',t|s)t^idt.
\]
Inserting the above equation into \eqref{ineq_5} yields,
\[
\begin{aligned}
    \Pi_{i,\dt} f(s) =& \underbrace{\frac1\dt\sum_{k=0}^i\l(\sum_{j=1}^i\coef{i}_jj^k\r) \frac{(\dt)^k}{k!}\int_\S f(s' - s)  \pt^k_t\rho(s',0|s)ds'}_{I}\\
    &+ \underbrace{\frac1{\dt i!}\sum_{j=1}^i\coef{i}_j\l(\int_\S \int_0^{j\dt}f(s' - s)\pt_t^{i+1}\rho(s',t|s)t^idt ds' \r)}_{II}.
    \end{aligned}
\]
By the definition of $\coef{i}_j$, the first part can simplified to
\[
\begin{aligned}
    I
    =&\frac1\dt\l(\sum_{j=1}^i\coef{i}_j\r) \int_\S f(s' - s) \rho(s',0|s)ds' + \int_\S f(s' - s)  \pt_t\rho(s',0|s)ds'  \\
    =&\frac{\sum_{j=1}^i\coef{i}_j}{\dt}f(0) + \int_\S\mL_{\mu, \Sig}f(s' - s)\rho(s',0|s) ds'   = \mL_{\mu, \Sig}f(0).
\end{aligned}
\]
Apply integration by parts, the second part can be written as
\[
II = \frac{1}{\dt i!}\sum_{j=1}^i\coef{i}_j\int_0^{j\dt}\E[\mL^{i+1}_{\mu,\Sig}f(s_t - s_0)|s_0 = s] t^i dt,
\]
which completes the proof. 
\end{proof}

\paragraph{Proof of Lemma \ref{lemma:plinf}}
\begin{proof}
Note that $p(s,t)$ satisfies the following forward Kolmogorov equation \cite{pavliotis2016stochastic},
\[
    \pt_tp(s,t) = \mL_{\mu,\Sig} p(s,t), \quad \text{with }p(s,0) = f(s).
\]
Multiplying $p$ to the above equation and let $q_l = \pt_{s_l}p$, with Assumption \ref{ass_2}/(a), one has
\begin{equation}\label{ineq_12}
\begin{aligned}
&\pt_t (\frac12p^2) = \mL_{\mu,\Sig}(\frac12 p^2) - \frac12q^\top \Sig q, \quad \pt_t (\frac12p^2) \leq \mL_{\mu,\Sig}(\frac12 p^2) - \frac\lammin2\ll q\rl^2_2.
\end{aligned}
\end{equation}
On the other hand, $q_l$ satisfies
\begin{equation}\label{ineq18}
    \pt_t q_l = \mL_{\mu,\Sig}q_l + \mL_{\pt_{s_l}\mu,\pt_{s_l}\Sig} p, \quad \text{with }q(s,0) = \pt_{s_l}f(s).
\end{equation}
Multiplying $q_l$ to the above equation and then summing it over $l$ gives,
\[
\begin{aligned}
\pt_t (\frac12 \ll q \rl_2^2) =& \mL_{\mu,\Sig} (\frac12  \ll q \rl_2^2) - \frac12\sum_l (\nb q_l)^\top\Sig (\nb q_l) + q^\top \nb \mu \cdot q + \sum_l\frac12(\pt_{s_l}\Sig:\nb q) q_l,\\
\pt_t (\frac12 \ll q \rl_2^2) \leq& \mL_{\mu,\Sig} (\frac12  \ll q \rl_2^2) - \frac\lammin2\ll\nb q\rl_2^2  + \ll \nb\mu \rl_2\ll q \rl_2^2  + \frac12\ll \nb\Sig \rl_2\ll\nb q\rl_2\ll q \rl_2\\
\leq& \mL_{\mu,\Sig} (\frac12  \ll q \rl_2^2)   + \underbrace{\l(\ll \nb\mu \rl_\Linf  + \frac{\ll \nb\Sig \rl^2_\Linf}{8\lammin}\r)}_{\frac\lammin2c}\ll q \rl_2^2.
\end{aligned}
\]
Adding the above inequality to $c\times$\eqref{ineq_12} with $c = \frac{C_{\nb\mu, \nb\Sig}}{\lammin}$ and $C_{\nb\mu, \nb\Sig}$ defined in \eqref{def of c123}, one has
\[
\pt_t \l(\frac{c}2 p^2 + \frac12 \ll q \rl_2^2\r) \leq \mL_{\mu,\Sig} (\frac{c}2 p^2 + \frac12  \ll q \rl_2^2) . \\
\]
Let $g(s,t) = \frac{c}2 p^2 + \frac12 \ll q \rl_2^2$, then $\pt_t g \leq \mL_{\mu,\Sig} g$. Let $g_1(s,t) = \E[cf(s_t)^2/2+ \ll\nb f(s_t)\rl_2^2/2|s_0 = s]$, then $g_1(s,t)$ satisfying $\pt_t g_1 = \mL_{\mu,\Sig} g_1$, with $g_1(0,s) = g(0,s)$. Since $\ll g_1(t,s)  \rl_\Linf \leq \frac{c}2\ll f \rl_\Linf^2 + \frac12\ll \nb f \rl_\Linf$, by comparison theorem, one has for $\forall s\in \S$
\[
g(t,s) \leq g_1(t,s)\leq  \frac12(c\ll f \rl^2_\Linf + \ll \nb f \rl^2_\Linf),
\]
which completes the proof for the first inequality. 

For the second $p(t,s) = \E[f(s_t)(s_t - s_0)|s_0 = s]$, Let $p_1(t,s) = \E[f(s_t)s_t |s_0 = s], p_2(t,s) = \E[f(s_t)|s_0 = s]$, then note that it satisfies the following PDE, 
\[
\pt_tp = \pt_t(p_1 - p_2 s) =  \mL_{\mu,\Sig} p_1 - (\mL_{\mu,\Sig} p_2)s = \mL_{\mu,\Sig} p + \mu p_2 + \Sigma\nb p_2, \quad \text{with }p(0,s) = 0.
\]
Multiplying it with $p^\top$ and letting $q = \nb p$ gives, 
\[
\pt_t(\frac12\ll p\rl_2^2) \leq \mL_{\mu,\Sig}(\frac12\ll p\rl_2^2) -\frac{\lammin}{2}\ll q\rl_2^2 + \frac{1}{2a} \ll \mu p_2 + \Sigma\nb p_2\rl^2_\Linf + \frac{a}{2}\ll p\rl_2^2, \quad \text{for }\forall a>0.
\]
Let $ c_2 = \ll \mu p_2 + \Sigma\nb p_2\rl^2_\Linf$, one has  
\begin{equation}\label{ineq_15}
    \pt_t(\frac12\ll p\rl_2^2) \leq \mL_{\mu,\Sig}(\frac12\ll p\rl_2^2) -\frac{\lammin}{2}\ll q\rl_2^2 + \frac{c_2}{2a} + \frac{a}{2}\ll p\rl_2^2, \quad \text{for }\forall a>0.
\end{equation}
Let $g(t,s) = \frac{1}{2}\ll p\rl_2^2e^{-at} + \frac{c_2}{2a^2} e^{-at}$, then,
\begin{equation}\label{ineq_14}
    \pt_t g \leq \mL_{\mu,\Sig} g, \quad \text{with }g(0,s) = \frac{c_2}{2a^2}.
\end{equation}
Similarly, by comparison theorem, $\ll g(t,s) \rl_\Linf \leq \ll g(0,s) \rl_\Linf = \frac{c_2}{2a^2}$, which implies, 
\[
\ll p(t,\cdot) \rl_\Linf \leq \frac{\sqrt{c_2}}{a}\sqrt{e^{at} -1} = \sqrt{c_2}\sqrt{e^t-1} \leq C_1(\ll \mu \rl_{C^1},\ll \Sigma \rl_{C^1}, \ll f \rl_{C^1})\sqrt{e^t-1}
\]
with 
\begin{equation}\label{def of C1}
    C_1(\ll \mu \rl_{C^1},\ll \Sigma \rl_{C^1}, \ll f \rl_{C^1}) = \ll \mu \rl_\Linf \ll f \rl_\Linf  + \ll \Sigma \rl_\Linf \l(\sqrt{\frac{C_{\nb\mu, \nb\Sig}}{\lammin}}\ll f \rl_\Linf + \ll \nb f \rl_\Linf\r) 
\end{equation}
where the first equality is obtained by selecting $a =1$. One obtains the last inequality by applying the first inequality in this Lemma to $\ll \nb p_2\rl_\Linf$ and $\ll p_2 \rl_\Linf \leq \ll f \rl_\Linf$.

On the other hand, letting $h(t,s) = \mu p_2 + \Sigma\nb p_2$, taking $\nb$ to \eqref{ineq_14} and multiplying $q^\top$ to it gives,
\begin{equation}\label{ineq_20}
\begin{aligned}
\pt_t (\frac12 \ll q \rl_2^2) =& \mL_{\mu,\Sig} (\frac12  \ll q \rl_2^2) - \frac12\sum_l (\nb q_l)^\top\Sig (\nb q_l) + q^\top \nb \mu \cdot q + \sum_l\frac12(\pt_{s_l}\Sig:\nb q) q_l + q\cdot\nb h ,
\end{aligned}
\end{equation}
\[
\begin{aligned}
\pt_t (\frac12 \ll q \rl_2^2) \leq& \mL_{\mu,\Sig} (\frac12  \ll q \rl_2^2) - \frac\lammin2\ll\nb q\rl_2^2  + \ll \nb\mu \rl_2\ll q \rl_2^2  + \frac12\ll \nb\Sig \rl_2\ll\nb q\rl_2\ll q \rl_2 + \frac{1}{2a}\ll \nb h \rl_\Linf^2 \\
&+ \frac{a}{2}\ll q\rl_2^2\\
\leq& \mL_{\mu,\Sig} (\frac12  \ll q \rl_2^2)   + \l(\ll \nb\mu \rl_\Linf  + \frac{\ll \nb\Sig \rl^2_\Linf}{8\lammin}\r)\ll q \rl_2^2+\frac{1}{2a}\ll \nb h \rl_\Linf^2+\frac{a}{2}\ll q\rl_2^2.
\end{aligned}
\]
Adding the above inequality to $c\times$\eqref{ineq_15} with $c = \frac{C_{\nb\mu, \nb\Sig}}{\lammin}$ and $C_{\nb\mu, \nb\Sig}$ defined in \eqref{def of c123}, one has
\[
\pt_t \l(\frac12 p^2 + \frac{c}2 \ll q \rl_2^2\r) \leq \mL_{\mu,\Sig} \l(\frac12 p^2 + \frac{c}2  \ll q \rl_2^2\r) +\frac{1}{2a}\l(\ll \nb h \rl_\Linf^2 +c\ll \mu \rl^2_\Linf\r) + a\l(\frac12\ll p \rl^2_2 + \frac{c}2\ll q \rl^2_2\r) . \\
\]
Let $g(s,t) = \l(\frac{c}2 p^2 + \frac{1}2 \ll q \rl_2^2\r) e^{ -a t } + \frac{1}{2a^2}(c\ll \mu \rl_\Linf^2 + \ll \nb h \rl_\Linf^2)e^{-at}$, then 
\[
\pt_t g \leq \mL_{\mu,\Sig} g, \quad \text{with } g(s,0) = \frac{1}{2a^2}(c\ll \mu \rl_\Linf^2 + \ll \nb h \rl_\Linf^2).
\]
By the comparison theorem, one has $g(s,t) \leq 0$, which implies
\[
\frac{c}2 \ll p\rl_\Linf^2 + \frac{1}2 \ll q \rl_\Linf^2 \leq \frac{1}{2a^2}(c\ll \mu \rl_\Linf^2 + \ll \nb h \rl_\Linf^2)(e^{at} - 1)
\]
\[
 \ll q \rl_\Linf \leq (\sqrt{c}\ll \mu \rl_\Linf + \ll \nb h \rl_\Linf)\sqrt{e^{t} - 1}
\]
Since one needs to bound $\nb^2 p_2$, we will estimate it first. Taking $\pt_{s_k}$ to \eqref{ineq18}, and letting $q_{lk} = \pt_{s_k}\pt_{s_l} p$, one has, 
\[
\begin{aligned}
    \pt_t q_{lk} = \mL_{\mu,\Sig} q_{lk}  + \mL_{\pt_{s_l}\mu,\pt_{s_l}\Sigma}q_k + \mL_{\pt_{s_k}\mu,\pt_{s_k}\Sig} q_l  + \mL_{\pt_{s_l}\pt_{s_k}\mu,\pt_{s_l}\pt_{s_k}\Sigma}p
\end{aligned}
\]
Multiplying $q_{lk}$ to both sides of the above equation, and summing it over $l, m$ gives,
\begin{equation}\label{ineq_19}
    \begin{aligned}
    &\pt_t\l(\frac12\ll \nb^2 p\rl^2_2\r) \\
    \leq& \mL_{\mu,\Sig} \l(\frac12\ll \nb^2 p\rl^2_2\r)  + \underbrace{\l(\ll \nb \mu\rl_\Linf^2 + \frac{\ll \nb \Sigma \rl^2_\Linf}{2\lammin} + \frac12 \ll \nb^2 \Sigma\rl_\Linf^2\r)}_{c_3}\ll \nb^2 p \rl^2_2 + \frac12\ll \nb^2 \mu\rl^2_2 \ll q \rl^2_2\\
\end{aligned}
\end{equation}
The equation \eqref{ineq_20} gives
\begin{equation}\label{ineq_21}
    \pt_t\l(\frac12\ll \nb p\rl^2_2\r) \leq \mL_{\mu,\Sig} \l(\frac12\ll \nb p\rl^2_2\r) -\frac{\lammin}{4}\ll \nb^2 p \rl_2^2 + \lammin c\ll q \rl_2^2.
\end{equation}
Now combining \eqref{ineq_12}, \eqref{ineq_21} and \eqref{ineq_19} leads to,
\begin{equation}\label{ineq_22}
\begin{aligned}
    &\eqref{ineq_12}\times \frac{2}{\lammin}\l(\frac12\ll \nb^2\mu\rl^2_2 + 4cc_3\r) + \eqref{ineq_21}\times \frac{4c_3}{\lammin} + \eqref{ineq_19}\\ 
    =&\frac12\pt_t \l[\frac{2}{\lammin}\l(\frac12\ll \nb^2\mu\rl^2_2  + 4cc_3\r)\ll p \rl_2^2 +  \frac{4c_3}{\lammin} \ll\nb p \rl_2^2+ \ll\nb^2 p \rl_2^2\r] \\
    \leq& \frac12\mL_{\mu,\Sigma} \l[\frac{2}{\lammin}\l(\frac12\ll \nb^2\mu\rl^2_2  + 4cc_3\r)\ll p \rl_2^2 +  \frac{4c_3}{\lammin} \ll\nb p \rl_2^2+ \ll\nb^2 p \rl_2^2\r] 
\end{aligned}
\end{equation}
By comparison theorem, one has
\[
\ll \nb^2 p \rl_2 \leq \frac{1}{\sqrt{\lammin}}\l(\ll \nb^2\mu\rl_2  + 2\sqrt{cc_3}\r)\ll f \rl_2^2 +  \frac{2\sqrt{c_3}}{\sqrt{\lammin}} \ll\nb f \rl_2+ \ll\nb^2 f \rl_2
\]
By the first inequality of this Lemma and the above inequality, one has, 
\[
\sqrt{c}\ll \mu \rl_\Linf + \ll \nb h \rl_\Linf \leq C_2(\ll \mu \rl_{C^2}, \ll  \Sigma \rl_{C^2}, \ll f \rl_{C^2})
\]
where 
\begin{equation}\label{def of C2}
\begin{aligned}
    & C_2(\ll \mu \rl_{C^2}, \ll  \Sigma \rl_{C^2}, \ll f \rl_{C^2}) \\
    = &\sqrt{c}\ll \mu \rl_\Linf  + \ll \nb \mu \rl_\Linf\ll f \rl_\Linf \\
&+ \l(\ll \mu \rl_\Linf +  \ll\nb \Sigma \rl_\Linf\r)(\sqrt{c}\ll f \rl_\Linf + \ll \nb f \rl_\Linf) \\
&+ \ll \Sigma \rl_\Linf \l[\frac{1}{\sqrt{\lammin}}\l(\ll \nb^2\mu\rl_2  
+ 2\sqrt{cc_3}\r)\ll f \rl_2^2 +  \frac{2\sqrt{c_3}}{\sqrt{\lammin}} \ll\nb f \rl_2+ \ll\nb^2 f \rl_2\r]
\end{aligned}
\end{equation}
which completes the proof.
\end{proof}

\subsection{Proof of Lemma \ref{lemma:pt_s l2 est}}\label{proof of lemma pt_s l2 est}

Based on Proposition \ref{lemma:l2 operator}, one has
\[
\begin{aligned}
    &\frac\b2 \ll V \rl_\rho^2 -\la r(s), V(s)\ra_\rho \leq  - \frac\lammin4 \ll\nb V(s) \rl_\rho^2;\\
    &\b \ll \nb V \rl_\rho^2 -\la \nb r(s), \nb V(s)\ra_\rho \leq \frac{C_{\nb\mu,\nb\Sig}}{2}\ll \nb V \rl_\rho^2.
\end{aligned}
\]
Multiplying $\frac{2C_{\nb\mu,\nb\Sig}}{\lammin}$ to the first inequality and adding it to the second one gives 
\begin{equation}\label{ineq_2}
    \begin{aligned}
    &\b\l(\frac{C_{\nb\mu,\nb\Sig}}{\lammin} \ll V \rl^2 + \ll\nb  V(s) \rl_\rho^2 \r)\\
    \leq& \frac1{2\b}\l(\frac{C_{\nb\mu,\nb\Sig}}{\lammin}\ll r \rl_\rho^2 + \ll \nb r\rl_\rho^2\r)+  \frac\b{2}\l(\frac{2C_{\nb\mu,\nb\Sig}}{\lammin}\ll V\rl^2_\rho+ \ll \nb V \rl^2_\rho\r),
\end{aligned}
\end{equation}
which implies
\[
\begin{aligned}
    &\ll\nb  V(s) \rl_\rho
    \leq \frac2{\b}\l(\sqrt{\frac{C_{\nb\mu,\nb\Sig}}{\lammin}}\ll r \rl_\rho + \ll \nb r\rl_\rho\r).
\end{aligned}
\]

For the second inequality, 
Let $p(s,t) = \E[r(s_t) | s_0 = s]$, then by Lemma \ref{lemma:plinf}, one has
\[
\begin{aligned}
    &\ll \nb \hV(s) \rl_\Linf = \ll \int_0^\infty e^{-\b t} \nb p(s,t) dt \rl_\Linf \leq  \int_0^\infty e^{-\b t} \ll \nb p(s,t)\rl_\Linf dt\\
    \leq&  \int_0^\infty e^{-\b t} \l(\sqrt{\frac{C_{\nb\hmu,\nb\hs}}{\lammin}} \ll r \rl_\Linf + \ll \nb r \rl_\Linf \r)dt\\ 
     \leq& \frac1\b\l(\frac1{\sqrt{\lammin}}\sqrt{2C_{\nb\mu,\nb\Sig} + \frac{1}{2\lammin}\ll \nb (\Sig - \hs) \rl_\Linf^2 + 2\ll \nb (\mu - \hmu) \rl_\Linf} \ll r \rl_\Linf + \ll \nb r \rl_\Linf \r).
\end{aligned}
\]
Similar to the estimate of the $\ll \nb \cdot (\Sig - \hs)\rl_\Linf$ in Lemma \ref{lemma:muhat - mu}, one has, 
\[
\ll \nb(\Sig - \hs_i ) \rl_\Linf \leq L_{\nb\Sig} \dt^i + o(\dt^i),
\]
where 
\begin{equation}\label{def of LnbSigrho}
    L_{\nb\Sig} =\h{C}_i \l[\sqrt{C_{\nb\mu, \nb\Sig}/\lammin} \ll h \rl_\Linf + \ll \nb h \rl_\Linf + 4\l( \sqrt{C_{\nb\mu, \nb\Sig}/\lammin} \ll \mu \rl_\Linf + \ll \nb \mu \rl_\Linf\r) \r],  
\end{equation}
with $\h{C}_i, h(s), C_{\nb\mu, \nb\Sig}$ defined in \eqref{def of Lmu}, \eqref{def of h f g}, \eqref{def of c123}.
Taking $\nb$ to \eqref{ineq_55} gives, 
\[
\ll \nb(\mu - \hmu_i ) \rl_\Linf \leq \frac{1}{\dt i!}\sum_{j=1}^i|\coef{i}_j| \int_0^{j\dt} \ll \nb p(s,t) \rl t^i dt
\]
with 
\[
p(s,t) = \E[\mL_{\mu,\Sig}^i\mu(s_t) | s_0 = s].
\]
By Lemma \ref{lemma:plinf}, one has,
\[
\ll \nb(\mu - \hmu_i ) \rl_\Linf \leq L_{\nb\mu} \dt^i
\]
with 
\begin{equation}\label{def of Lnbmu}
    L_{\nb\mu} = C_i \l(\sqrt{\frac{C_{\nb\mu, \nb\Sig}}{\lammin} }\ll\mL_{\mu,\Sig}^i\mu(s) \rl_\Linf + \ll\nb\mL_{\mu,\Sig}^i\mu(s) \rl_\Linf \r).
\end{equation}
Therefore, 
\[
\begin{aligned}
    &\ll \nb \hV(s) \rl_\Linf 
     \leq \frac1\b\l[\l(\sqrt{\frac{2C_{\nb\mu,\nb\Sig}}{\lammin}} + \frac{L_{\nb\Sig}}{\sqrt{2}\lammin}\dt^i + \sqrt{\frac{2L_{\nb\mu}}{\lammin}}\dt^{i/2}\r) \ll r \rl_\Linf + \ll \nb r \rl_\Linf \r] + o(\dt^i).
\end{aligned}
\]

\end{document}